\documentclass[a4paper]{amsart}

\usepackage{amsmath,multicol}
\usepackage{amssymb,mathrsfs}
\usepackage{amsthm}
\usepackage[all]{xypic}
\usepackage{amsfonts}
\usepackage{fullpage}
\usepackage{verbatim}
\usepackage{graphicx}
\usepackage{url}  

\usepackage{hyperref}

\numberwithin{equation}{section}
\theoremstyle{plain}
\newtheorem{theorem}[equation]{Theorem}
\newtheorem{proposition}[equation]{Proposition}
\newtheorem{corollary}[equation]{Corollary}
\newtheorem{lemma}[equation]{Lemma}

\newtheorem{assumption}[equation]{Assumption}

\newtheorem{sublemma}[equation]{Sublemma}

\theoremstyle{definition}
\newtheorem{definition}[equation]{Definition}

\newtheorem{notation}[equation]{Notation}

\newtheorem{example}[equation]{Example}

\newtheorem{basicfacts}[equation]{Basic~Facts}
\newtheorem{remark}[equation]{Remark}

 \makeatletter
\DeclareRobustCommand{\qed}{%
  \ifmmode
    \eqno \def\@badmath{$$}%$$
    \let\eqno\relax \let\leqno\relax \let\veqno\relax
    \hbox{\openbox}%
  \else
    \leavevmode\unskip\penalty9999 \hbox{}\nobreak\hfill
    \quad\hbox{\openbox}%
  \fi
}

\newcommand{\DMO}{\DeclareMathOperator}

\newcommand{\beq}{\begin{equation}}
\newcommand{\eeq}{\end{equation}}

\newcommand{\blank}{\mbox{$\underline{\makebox[10pt]{}}$}}

\newcommand{\bbar}[1]{\overline{#1}}

\DeclareMathOperator{\Hom}{{Hom}}

\DeclareMathOperator{\End}{{End}}
\DeclareMathOperator{\Ext}{{Ext}}

\DeclareMathOperator{\Aut}{{Aut}}
\DeclareMathOperator{\Proj}{Proj}

\DeclareMathOperator{\hd}{hd}

\DeclareMathOperator{\GK}{GKdim}
\DeclareMathOperator{\gr}{gr}
\DeclareMathOperator{\Div}{Div}

\DeclareMathOperator{\grk}{grk}

\newcommand{\QVB}{{Q_{V\hskip -1.5pt dB}}}
\newcommand{\mc}{\mathcal}

\newcommand{\kk}{{\Bbbk}}
\newcommand{\timesg}{{\raisebox{1pt}{$\scriptscriptstyle\bullet$}} g}

\newcommand{\dotms}{{\,\raisebox{1pt}{${\scriptscriptstyle\bullet}_{_{M\hskip -1pt S}}$}}}

\newcommand{\ZZ}{{\mathbb Z}}

\newcommand{\PP}{{\mathbb P}}

\newcommand{\mb}{\mathbb}
\newcommand{\wt}{\widetilde}

\newcommand{\sA}{\mc{A}}

\newcommand{\sL}{\mc{L}}
\newcommand{\sM}{\mc{M}}
\newcommand{\sN}{\mc{N}}
\newcommand{\sO}{\mc{O}}
\newcommand{\sP}{\mc{P}}

\newcommand{\sX}{\mc{X}}

\DeclareMathOperator{\coh}{coh}
\DeclareMathOperator{\rgr}{gr-\!}

\DeclareMathOperator{\rGr}{Gr-\!}

\DeclareMathOperator{\rMod}{Mod-\!}

\DeclareMathOperator{\lQgr}{\!-Qgr}

\DeclareMathOperator{\rQgr}{Qgr-\!}
\DeclareMathOperator{\rqgr}{qgr-\!}

\DeclareMathOperator{\GKdim}{GKdim}

\DeclareMathOperator{\SK}{\mb{A}}

\newcommand{\on}{\operatorname}

\newcommand{\dra}{\dashrightarrow}
\newcommand{\hra}{\hookrightarrow}

\newcommand{\ssm}{\smallsetminus}

\newcommand{\uExt}{\underline{\Ext}}
\newcommand{\uHom}{\underline{\Hom}}

\DeclareMathAlphabet{\mathpzc}{OT1}{pzc}{m}{it}
 \newcommand{\hilb}{{\sf hilb}\,}

\DMO{\gldim}{gldim}
\DMO{\pdim}{pdim}
\DMO{\injdim}{injdim}
\DMO{\Tot}{Tot}
\DMO{\cd}{cd}
\DMO{\Bl}{Bl}

\DeclareMathOperator{\Pic}{Pic}

\newcommand{\wh}{\widehat}
\newcommand{\too}{\longrightarrow}

\newcommand{\X}{\mc X}

\title[Ring-theoretic blowing down  II]{Ring-theoretic blowing down  II: Birational transformations}
\author{D. Rogalski,  S. J. Sierra, and J. T. Stafford}
\address{(Rogalski)
Department of Mathematics, UCSD, La Jolla, CA 92093-0112, USA. }
\email{drogalsk@math.ucsd.edu}
 \address{(Sierra) School of Mathematics,
University of Edinburgh, Edinburgh EH9 3JZ, Scotland.}
\email{s.sierra@ed.ac.uk}
\address{(Stafford) School of Mathematics,  The University of Manchester,   Manchester M13 9PL,
England.}
\email{Toby.Stafford@manchester.ac.uk}

\thanks{The first author was partially supported by the NSF grant DMS-1201572 and the NSA grant H98230-15-1-0317,  the second author is  partially   
supported by  EPSRC grants  EP/M008460/1 and EP/T018844/1, 
and the third author is partially supported by a Leverhulme Emeritus Fellowship.}

\date{\today}

\subjclass[2010]{Primary: 14A22,  14H52,  16E65,   16S38, 16W50; Secondary:  14J10,  16P40, 18E15; }
\keywords{Noncommutative projective geometry,  noncommutative surfaces, Sklyanin algebras,
noetherian  graded rings,
noncommutative  blowing~up and blowing down}

\begin{document}
 
  \begin{abstract}
  One of the major open  problems in noncommutative algebraic geometry is the classification of noncommutative projective 
  surfaces (or, slightly more generally, of  noetherian connected graded domains of Gelfand-Kirillov dimension 3). Earlier work of the authors  classified the connected graded noetherian subalgebras of Sklyanin algebras using a noncommutative  analogue of blowing up. 
  In a companion paper the authors also described a noncommutative version   of blowing down and, for example,  
  gave a noncommutative analogue of Castelnuovo's classic theorem that   lines of self-intersection $(-1)$
  on a smooth surface can be contracted.

  In this paper we will use these techniques to construct explicit birational transformations between  various noncommutative surfaces 
  containing an elliptic curve. Notably we show that Van den Bergh's quadrics can be obtained from the Sklyanin algebra by suitably blowing up and down, and we also provide a noncommutative analogue of the classical Cremona transform.

\end{abstract}

 \maketitle
 \tableofcontents
 
 %%%%%%%%%%%%%%%%%%%
  %%%%%%%%%%%%%%%%%%%

   \clearpage

 \section{Introduction}\label{INTRO}

 Throughout the paper, $\kk$ will denote an algebraically closed field, which in the introduction has characteristic zero, and  all rings will be $\kk$-algebras. 
A $\kk$-algebra $R$ is \emph{connected graded}\label{cg-defn-intro} or \emph{cg} if $R=\bigoplus_{n\geq 0}R_n$ is a finitely generated, $\mathbb{N}$-graded algebra with $R_0=\kk$. For such a ring $R$, the category of graded noetherian right $R$-modules will be denoted $\rgr R$, with quotient category $\rqgr R$ obtained by quotienting out the Serre subcategory of finite dimensional modules. 
We define $\rqgr R$ to be \emph{smooth} (which is synonymous with  nonsingular in our usage) if it has finite homological dimension.  An effective intuition    is  to regard  $\rqgr R$ as the category of coherent sheaves on the (nonexistent) space $\Proj(R)$. Thus, for example,  we define $\rqgr R$ to  be \emph{a noncommutative surface} if $R$ has Gelfand-Kirillov dimension $\GKdim R=3$.

 One of the main open problems in noncommutative algebraic geometry is the classification of noncommutative projective surfaces
 (or, slightly more generally, of  noetherian connected graded domains of Gelfand-Kirillov dimension 3). This
 has been solved in many particular cases and those solutions have led to some fundamental advances in the subject; see, for 
 example, \cite{ATV1990,  RSSlong, KRS, SV, VdB-blowups, VdB3} and the references therein. In \cite{Ar}, 
 Artin conjectured that, birationally at least, there is a short list of such surfaces, with the generic case being a Sklyanin algebra. 
 Here, the graded quotient ring $Q_{gr}(R)$ of $R$ is obtained by inverting the non-zero homogeneous elements and  two such  
 domains  $R, S$ are \emph{birational} if  $Q_{gr}(R)_0 \cong Q_{gr}(S)_0$; that is, if they have the same {\em function skewfield} (note that $Q_{gr}(R)_0$ is a division ring). 
 Although no solution of Artin's conjecture is in sight, it does lead naturally to the question of determining the surfaces 
within each birational class. We consider this problem in this paper.

In classical (commutative) geometry, birational projective surfaces can be obtained
 from each other by means of blowing up and down (monoidal transformations). 
 The noncommutative notion of blowing up points on a noncommutative 
 surface has been described in \cite{VdB-blowups} with a rather more ring-theoretic approach given in \cite{R-Sklyanin} and 
 \cite{RSSshort}. Moreover, an  analogue of blowing down (contracting) $(-1)$-curves is defined  and explored in  the companion paper 
 to this one \cite{blowdownI}. The latter paper proves, among other things, that there is a noncommutative analogue of 
 Castelnuouvo's classic theorem:   curves 
of self-intersection $(-1)$ can indeed be contracted in a smooth noncommutative surface. 
 
 The next step in understanding noncommutative surfaces is then to determine when two noncommutative surfaces are
  birational, and we take steps in this direction in the present paper. In particular, there are two known    classes of 
  generic  minimal noncommutative surfaces: Sklyanin algebras themselves, which are thought of as coordinate rings of noncommutative projective planes, 
and the 
  noncommutative quadrics defined by Van den Bergh \cite{VdB3}.  
  (See \cite{RSS-minimal} for discussion of the way in which these are noncommutative minimal models.) 
  One of the main aims of this paper is to prove that,
   given that they have the appropriate invariants, these algebras can indeed be transformed into each other by means of
    blowing up and down; thereby giving  a noncommutative version of the classical birational transform: 
     \beq \label{classical}
\begin{split} \xymatrix{
 & \Bl_{p,p'} \PP^2\,  \cong\,  \Bl_r (\PP^1 \times \PP^1) \ar[ld]_{\pi_{pp'}} \ar[rd]^{\pi_r} & \\
 \PP^2 && \PP^1\times \PP^1.  \ar@{-->}[ll]^{\phi} 
 }
 \end{split}
 \eeq

 In this paper, we apply the  birational geometry of noncommutative projective surfaces as developed in \cite{R-Sklyanin, blowdownI} to construct a noncommutative version of \eqref{classical}.   
 Doing this involves developing substantial new methods in noncommutative geometry; in particular, establishing a noncommutative version of the isomorphism in the top line of \eqref{classical} involves a recognition theorem for two-point blowups of 
 a noncommutative $\PP^2$ which is of independent interest.

 In order to state these results formally we need some  definitions.  
 We first remark that  as this paper  is a continuation of
 \cite{blowdownI} we will keep the same notation as in that paper, and will therefore refer the reader to that paper for 
 more standard notation.  In particular,  given an automorphism $\tau$ of an elliptic curve $E$, we always assume that $|\tau|=\infty$ and, as in  \cite[Section~2]{blowdownI}, write $B=B(E,\mathcal{M},\tau)$ for the \emph{twisted homogeneous coordinate ring or TCR} of $E$ relative to a line bundle $\mathcal{M}$. A domain $R$ is called an \emph{elliptic algebra} if it is a cg
 domain, with a central element $g\in R_1$ such that $R/gR\cong B(E,\mathcal{M},\tau)$  for some choice of $\{E,\mathcal{M}, \tau\} $.  
  We say that $E=E(R)$ and $\tau=\tau(R)$ are respectively the elliptic curve and the automorphism {\em associated to} $R$, and $\deg \sM$ is  the {\em degree} of $R$.   
  If we loosen our definition to allow $\tau$ to be the identity on $E$, then a commutative elliptic algebra is the same as the anticanonical ring of a (possibly singular) del Pezzo surface, and the degree of the algebra is the same as the commutative definition:  the self-intersection of the canonical class.

 Two examples of elliptic algebras are important to us. 
  The first is the 3-Veronese ring $T=S^{(3)}$ of the Sklyanin algebra $S=S(E,\sigma)$, as defined in \cite[(4.2)]{blowdownI}. This is called a \emph{Sklyanin elliptic algebra}, and the third Veronese is needed to ensure that $g\in T_1$.  The second, called a \emph{quadric elliptic algebra}, is the Veronese ring 
 $Q^{(2)}$ of a \emph{Van den Bergh quadric} $Q=\QVB=\QVB(E,\sigma,\Omega)$ for a second parameter $\Omega\in \mathbb{P}^1$. The latter is constructed in    \cite{VdB3}, and described in more detail in Example~\ref{eg:quadric}.

\emph{In this  paper we only consider elliptic algebras $R$ for which  $\rqgr R$ is smooth.}  This is not a restriction for the Sklyanin elliptic algebra $S^{(3)}$, and also holds for any blowup or contraction that we construct. But,  for each $\{E,\sigma\}$ it  does exclude a discrete family of quadrics 
 $\QVB(E,\sigma,\Omega)$.  However these can easily be dealt with separately; see Remark~\ref{non-smooth} for the details.

 The notions of blowing up an elliptic algebra $R$ at a  point $p\in E$ or blowing down a 
 line $L$ of self-intersection  $(-1)$ can be found in \cite{blowdownI}. 
 In brief,  if $R$ is an elliptic algebra of degree $\geq 2$, one can blow up any point $p \in E$ to obtain a subring  $\Bl_p(R)$. (For notational reasons, in the body of the paper we will usually write $R(p)$ in place of  $\Bl_p(R)$.)
   In particular, 
$ \Bl_p(R)/(g) = B(E, \sM(-p), \tau),$
and so blowing up decreases the degree of an elliptic algebra by one.
This is also what happens in the commutative setting:  given a smooth surface $X$, the canonical divisor of $\Bl_p(X)$ 
 is $\pi_p^* K_X + L_p$, where $L_p$ is the exceptional divisor \cite[Proposition~V.3.3]{Ha}. Thus the anticanonical ring of the blowup is a subring of the anticanonical ring of $X$.    In the noncommutative setting, 
it follows from the construction that $R$ and $\Bl_p(R)$ are also  birational; indeed  we even have $Q_{gr}(R) = Q_{gr}(\Bl_p(R))$.

There is a notion of the {\em exceptional line module $L_p$} of the noncommutative blowup $\Bl_p(R)$. Here a  {\em line module}  is a cyclic $R$-module with the Hilbert series $(1-s)^{-2}$ of $\kk[x,y]$. 
Moreover, as in the commutative setting, $L_p$ has self-intersection $(-1)$, where we are now using the noncommutative intersection notion of intersection theory due to Mori and Smith  \cite{MS} and  defined by
\[(L\dotms L'  )=\sum_{n\geq 0} (-1)^{n+1} \dim_\kk \Ext^i_{\rqgr R}(L,L').\]
In   \cite[Theorem~1.4]{blowdownI}, we  gave  a noncommutative version of Castelnuovo's contraction criterion:  Suppose that $R$ is an elliptic algebra with   $\rqgr R$ smooth, such that $R$ has a line module $L$ with $(L \dotms L) = -1$. Then we can contract $L$ in the sense that there exists a second elliptic algebra $R'$,  called the {\em blowdown} or {\em contraction of $R$ at $L$}, such that $\rqgr R'$ is smooth, $R = \Bl_q(R') $  is the blowup of   $R'$ at some point $q \in E$, and $L$ is the exceptional line module of this blowup.

The main result of this paper is the following noncommutative analogue of the classical diagram \eqref{classical}.

 \begin{theorem}\label{ithm:T}  {\rm (}See Theorem~\ref{thm:T}{\rm )}   Let $Q=\QVB(E,\alpha,\Omega)$ be a  
Van den Bergh quadric  such that $\rqgr Q$ is smooth and pick $r\in E$.  Then there exist  $\sigma \in \Aut(E)$ with $\sigma^3 = \alpha^2$ and $p, q \in E$ so that the blowups   $\Bl_r(Q^{(2)} )$  and $\Bl_{p,q}(S(E, \sigma)^{(3)} )$  are isomorphic.
\end{theorem}

Identify the rings $ \Bl_r(Q^{(2)})$ and $\Bl_{p,q} (S(E, \sigma)^{(3)})$ in Theorem~\ref{ithm:T}.
As blowups are inclusions in our setting, the effect of that result is to show that there are inclusions
\[ Q^{(2)}\  \supset\   \Bl_r(Q^{(2)}) \ =\  \Bl_{p,q} (S^{(3)})\  \subset \  S^{(3)},\]
for $S=S(E,\sigma)$.  Since we have seen that $R$ and $\Bl_x(R)$ are always birational, as a corollary we obtain:

\begin{corollary}\label{icor:birational}
Any Van den Bergh quadric $Q$ with $\rqgr Q$ smooth is birational to a Sklyanin algebra.\qed 
\end{corollary}

 We should emphasise that Theorem~\ref{ithm:T} is not the first proof of the birationality of $S^{(3)}$ and $Q^{(2)}$. Indeed, this was announced by Van den Bergh in \cite[Theorem~13.4.1]{StV} with complete proofs given in   
 \cite{Pr}.  The proof of Presotto and Van den Bergh is more geometric than the algebraic one given here. Our argument  has the advantage of providing a more explicit identity, which should also be useful elsewhere.

 In order to prove Theorem~\ref{ithm:T}, we develop new techniques in noncommutative birational geometry which are of independent interest.
Let $Q, r$ be as in the statement of Theorem~\ref{ithm:T}. 
The key problem in proving the theorem is to {\em recognise} $\Bl_r(Q^{(2)})$ as the two-point blowup of a Sklyanin algebra.
We do this through noncommutative intersection theory.  

In the commutative setting, the surface $ \Bl_r(\PP^1 \times \PP^1)$ contains three lines of self-intersection $(-1)$.
These are  the strict transforms $L, L'$ of the ``horizontal'' and ``vertical'' lines through $r$ on $\PP^1 \times \PP^1$, and  the exceptional divisor $L_r$ of the birational morphism $\pi_r$ from \eqref{classical}.  Their intersection theory is summarised in Figure~\ref{inttheory}: 
\begin{figure} 
\includegraphics[scale=.45]{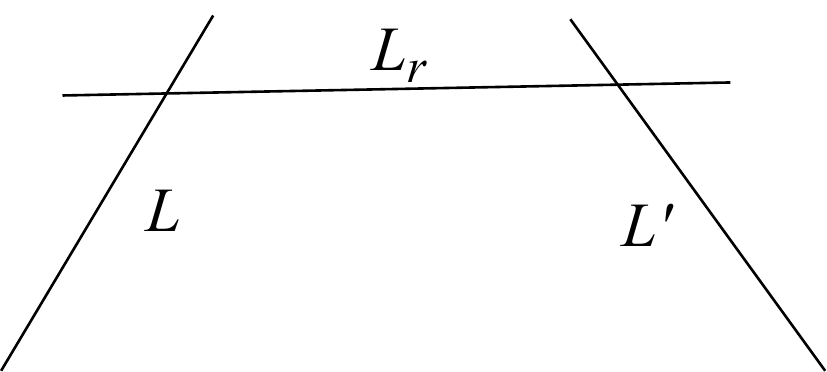}
\caption{The intersection theory of  $\Bl_r(\PP^1 \times \PP^1)$}
\label{inttheory}
\end{figure}
 $L\cdot L' = 0$ while  $L_r \cdot L = L_r \cdot L' = 1$.  
To realise $\Bl_r(\PP^1 \times \PP^1)$ as a two-point blowup of $\PP^2$ as in \eqref{classical}, one observes that $\phi \circ \pi_r$ is defined everywhere on $\Bl_r(\PP^1 \times \PP^1)$ and contracts $L$ and $L'$ to distinct points $p, p' \in \PP^2$.  
Further, $\phi \circ \pi_r = \pi_{p, p'}$ realises $L_r$ as the strict transform of the line through $p$ and $p'$.

 Proving Theorem~\ref{ithm:T} requires different techniques, since in the noncommutative context there is no good  analogue of an ``open set'' on which there is a well-defined ``morphism''.  One way to think about the problem is that we do not have an a priori inclusion of $S^{(3)}$ and $Q^{(2)}$ in the same graded quotient ring.
 Instead, we prove a recognition theorem for two-point blowups of   Sklyanin elliptic algebras, which roughly says that the corresponding    elliptic algebra  is characterised by the intersection theory described by Figure~\ref{inttheory}.  More specifically, we prove:
 
  \begin{theorem}\label{ithm:recog}
   {\rm (}Theorem~\ref{thm:recog}{\rm)}
 Let $R$ be a degree 7 elliptic algebra such that $\rqgr R$ is smooth and set $E=E(R)$ and   $\tau=\tau(R)$.  Suppose   that $R$ has line modules with the following properties.
 \begin{enumerate} 
 \item There are right line modules  $L, L'$ and $ L_r$ satisfying the  intersection theory:
 \begin{enumerate}
 \item[(i)] \ $(L \dotms L) = (L_r \dotms L_r) = (L' \dotms L' ) = -1$;
 \item[(ii)] \ $(L \dotms L')  = 0$ while $ (L_r\dotms L) =  (L_r \dotms L') = 1$.
  \end{enumerate}
\item The points where $E$ intersects $L$ and $L'$ are distinct.
 \end{enumerate}
 Then
   $R \cong \Bl_{p, q}(T)$, where $p \neq q \in E$ and 
 $T=S^{(3)} $ for  a Sklyanin algebra  $S=S(E,\sigma)$, where $\sigma^3 = \tau$.  
\end{theorem}

To prove Theorem~\ref{ithm:recog}, we first show that we can iteratively 
blown down the two lines $L$ and $L'$.  In order to show that this gives  a noncommutative $\mb{P}^2$, we show how to construct noncommutative analogues of the twisting sheaves $\mc{O}(1)$ and $\mc{O}(-1)$.  These determine a so-called \emph{principal $\mb{Z}$-algebra}, from which we can recover the Sklyanin algebra $S(E, \sigma)$.
Once we have proved Theorem~\ref{ithm:recog}, it is relatively straightforward to show that  $\Bl_r(Q^{(2)})$ has the necessary intersection theory and hence to  prove Theorem~\ref{ithm:T}.

A  converse to Theorem~\ref{ithm:recog} is provided by:

 \begin{proposition}\label{iprop:converse} { {\rm (}Proposition~\ref{prop:converse}{\rm)}}
Let $T$ be a Sklyanin elliptic algebra with associated elliptic curve $E$, and let $p \neq q \in E$.  Then $R= \Bl_{p,q}T$ satisfies the hypotheses of Theorem~\ref{ithm:recog}.
\end{proposition}

 Theorem~\ref{ithm:recog} is a useful new technique for understanding noncommutative surfaces.  To demonstrate its utility,
we use it to construct a noncommutative version of the Cremona transform
\[
Cr:  \PP^2  \dra \PP^2, \quad
[x:y :z]   \mapsto [yz: xz: xy],
\]
which factorises as 
\beq \label{Cremona}
\begin{split}
 \xymatrix{
 & \Bl_{p,q,r} \PP^2  \ar[ld] \ar[rd]& \\
 \PP^2 \ar@{-->}[rr]_{Cr}  && \PP^2
 }
 \end{split}
 \eeq
where $\{p,q,r \} = \{[1:0:0],[0:1:0], [0:0:1]\}$.  

We prove our version by showing that the blowup of $T=S^{(3)}$ at three general points $p,q,r$ has a hexagon 
\begin{equation}\label{fig2}
\begin{split}
{}_{ \includegraphics[scale=.4]{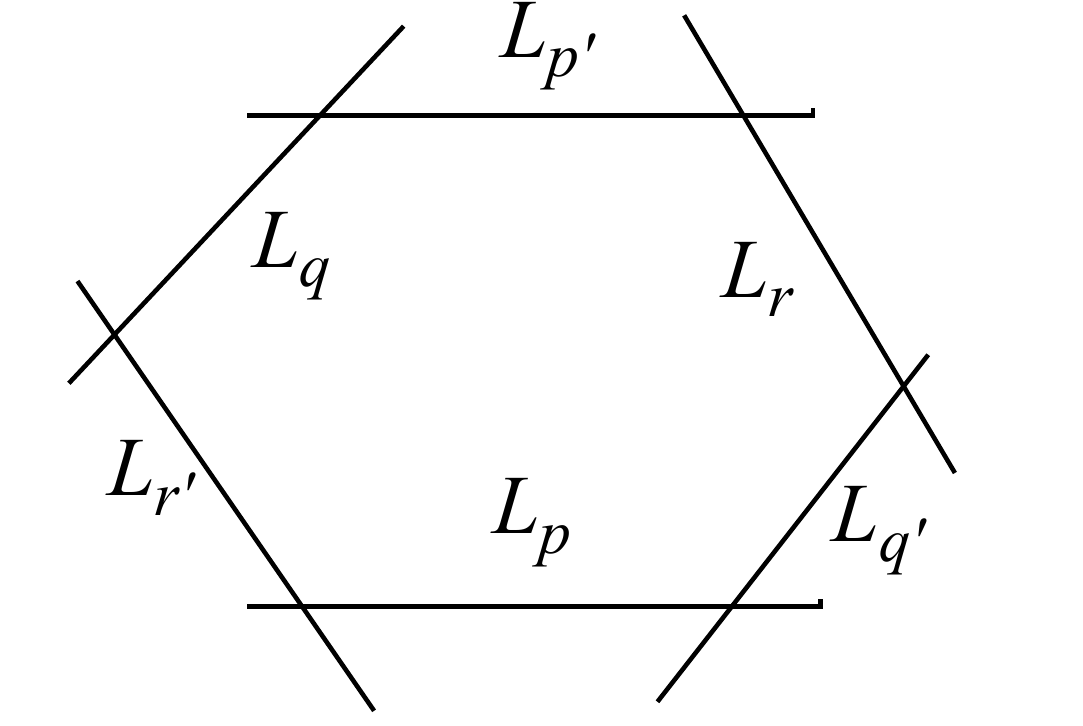}} 
\end{split}
 \end{equation}
 
 \vskip 5pt
 \noindent
 of $(-1)$ lines, just as happens for  the 3-point blowup of $\PP^2$ in classical geometry. Using intersection theory and the noncommutative Castelnuovo criterion \cite[Theorem~1.4]{blowdownI} we  show that each of the $(-1)$ lines  $L_{p'},L_{q'}$ and $L_{r'}$ 
 that are not exceptional for the original blowup can be contracted to give a ring to which Theorem~\ref{ithm:recog} applies. This  leads to  the following theorem:
   
 \begin{theorem}\label{intro-thm:Cremona}   {\rm (}Theorem~\ref{thm:Cremona} and Remark~\ref{rem:Cremona}{\rm )}
 Let $T = S^{(3)}$ be a Sklyanin elliptic algebra with associated curve $E$.
 Pick distinct points $p,q,r \in E$ satisfying minor conditions and let $R = \Bl_{p,q,r}T$ be the corresponding blowup.  Then $\rqgr R$ is smooth, and there is a  second elliptic algebra $T' \cong T $   obtained by blowing down the three lines   that are not exceptional for the original blowup.   Thus  $R = \Bl_{p_1,q_1,r_1,}T'$ for  suitable  points  $p_1, q_1, r_1 \in E$.    
 \end{theorem}
 
Once again,  the first noncommutative analogue of the Cremona transform was announced in \cite{StV} and proved in detail in  \cite{PV}.

 \begin{remark}\label{non-smooth}   We end the introduction by commenting on the smoothness hypothesis.
 As noted in \cite{SV} there do exist   quadrics $Q=\QVB$  for which $\rqgr Q$ is not smooth. 
 One can presumably extend the results of this paper to cover these examples, although this will require a more awkward notion of self-intersection, see \cite[Section~6]{blowdownI} and \cite[Definition~6.11]{blowdownI} in particular. However,  for most questions, in particular for those concerned with birationality,  this is unnecessary. Indeed if $Q$ is a  quadric  for which $\rqgr Q$ is not smooth, then  $Q$ has a Morita context to  a second quadric $\widetilde{Q}$ such that  $\rqgr \widetilde{Q}$ is smooth. 
 More precisely,  there exists a $(Q^{(2)},\widetilde{Q}^{(2)})$-bimodule $M$, finitely generated and of Goldie rank one on each side, such that $Q^{(2)}=\End_{\widetilde{Q}^{(2)}}(M)$ and $\widetilde{Q}^{(2)}=\End_{Q^{(2)}}(M)$. This is proved  by combining 
  \cite[Lemma~6.6 and Corollary~6.1]{RSS-minimal}  following results from  \cite{VdB1996} and  \cite{SV}. 
 Clearly $Q$ and $\widetilde{Q}$ are birational and so one can use $\widetilde{Q}$ to test birationality questions for $Q$.
 Thus, for example, Corollary~\ref{icor:birational} immediately extends to the non-smooth case.
\end{remark}

\begin{corollary}\label{icor:birational2}
Any Van den Bergh quadric $Q$  is birational to a Sklyanin algebra.\qed
\end{corollary}

\noindent {\bf Acknowledgements: }
We would like to thank Colin Ingalls, Shinnosuke Okawa, Dennis Presotto and Michel Van den Bergh for useful discussions and comments. 

This material originated in  work supported by the National Science Foundation under Grant No. 0932078 000, while the authors were in residence at the Mathematical Sciences Research Institute (MSRI) in Berkeley, California, during the spring semester of 2013. 
  We would like to thank both institutions for their support.

 %%%%%%%%%%%%%%%%
 %%%%%%%%%%%%%%%%

\section{The noncommutative geometry of elliptic algebras}\label{GEOMETRY}

In this section we consider properties of the category $\rqgr R$, where $R$ is an elliptic algebra.   
 We start however with some basic results and notation. As mentioned in the introduction, this paper is a continuation of \cite{blowdownI}, and so we refer the reader to that paper for basic notation.
 We fix an algebraically closed ground field $\kk$. 
  All algebras and schemes will be defined over $\kk$, and all maps will be $\kk$-linear unless otherwise specified.
  We begin with a few comments about the elliptic algebras defined in the introduction. We emphasise that \emph{throughout the paper we only consider elliptic algebras $R$ for which $\rqgr R$ is smooth.}
 
  \begin{notation}\label{elliptic-defn}  Let $R=\bigoplus_{n\geq 0}R_n$ be such an algebra, with central element $0\not=g\in R_1$ 
 such that $R/gR=B=B(E, \mc{M}, \tau)$  is a   TCR over the elliptic curve $E$. 
 Except when stated to the contrary, we always assume 
 that $|\tau|=\infty$ and that  $\deg R\geq3$. 
If $M$ is a graded $R$-module, the {\em $g$-torsion} submodule of $M$ consists of elements annihilated by some power of $g$, and {\em $g$-torsionfree} has the obvious meaning.   
Finally, we note that,  by \cite[Proposition 2.4]{RSSshort},  an elliptic algebra $R$ is always   Auslander-Gorenstein and   Cohen-Macaulay (CM)  in the sense of  \cite[Definition~2.1]{blowdownI}.  \end{notation}

Fix an elliptic algebra $R$ as above. Since $\GKdim(R) = 3$, recall that a graded module $M$ is \emph{Cohen-Macaulay (CM)} if $\Ext^i_R(M, R) = 0$ for all $i \neq 3-\GKdim(M)$, while  $M$ is 
\emph{maximal Cohen Macaulay (MCM)} if in addition $\GKdim(M) = 3$.

For an elliptic algebra $R$, let $\rGr R$ be the category of all graded right $R$-modules, with quotient category $\rQgr R$  obtained by quotienting out direct limits of finite-dimensional modules. Let $\pi:  \rGr R \to \rQgr R$ be the natural map, although we will typically abuse notation and  simply denote $\pi M$   by $M$ for any $M\in \rGr R$.
By the adjoint functor theorem, $\pi$ has a right adjoint $\omega:  \rQgr R \to \rGr R$.  
A graded $R$-module $M$ is  called \emph{saturated} if it is in the image of the section functor $\omega$.  By \cite[(2.2.3)]{AZ}, this is equivalent to $\Ext^1_R(\kk, M) = 0$.   The subcategory of finitely generated (hence noetherian) modules in $\rGr R$ is denoted $\rgr R$; similarly, $\rqgr R = \pi(\rgr R)$ is the subcategory of noetherian objects in $\rQgr R$.  Below 
we will write 
\[\X = \rqgr R.\]
  Note that  $\rqgr B \simeq \coh E$ by \cite{AV}.

Certain standard localizations of $R$ will be important.
If $M\in \rgr R$, let $M^\circ = (M \otimes_R R[g^{-1}])_0$.
By \cite[Lemma~6.8]{blowdownI}, the hypothesis that $\rqgr R$ be  smooth holds  if and only if $R^\circ$ has finite global dimension.  
The image of an element or set under the canonical surjection $R \to B$  is written  $x \mapsto \overline{x}$.
Let $R_{(g)}$ be the Ore localisation $R \mc C^{-1}$, where $\mc C$ is the set of  homogeneous elements of $R \ssm gR$.  Then 
$R_{(g)}/R_{(g)}g \cong Q_{gr}(B)$.  A right $R$-submodule $M$ of $R_{(g)}$ is {\em $g$-divisible} if $M \cap g R_{(g)} = Mg$, and hence $\overline{M} \cong M/Mg$. This is equivalent to $R_{(g)}/M$ being $g$-torsionfree.
If $M, M' \subseteq R_{(g)}$ are $g$-divisible, then so is $\Hom_R(M, M')$, by \cite[Lemma~2.12]{RSSlong}.

Let $M=\bigoplus_n M_n \in \rGr R$.  We say $M$ is \emph{left (right) bounded} if $M_n = 0$ for $n \ll 0$ (respectively $n \gg 0$).  If $M \in \rgr R$ then $M$ is left bounded.  It is clear from calculating with a resolution by finite rank graded free modules that 
$\Ext^i_R(M,N)$ is left bounded for all $i$ when $M, N \in \rgr R$.
Given $k\in \mathbb{N}$,  the shifted module 
  $M[k]$ is defined by  $M[k]_n = M_{n+k}$.
Set $\uHom_\X(M,N) = \bigoplus_n \Hom_\X(M, N[n])$ and  define $\uExt^i_\X(M, N)$  similarly.  
We let $H^i(\X, M) = \Ext^i_\X(R, M)$ denote the $i^{\rm th}$ {\em (sheaf) cohomology} of $M$.
By \cite[Theorem~5.3]{RSSshort}, $R$   satisfies the Artin-Zhang  $\chi$ condition which, 
 by \cite[Corollary~4.6]{AZ},   implies that   $\Ext^i_\X(M,N)$ is finite-dimensional for all $i$.  

Suppose that $M \in \rgr R$ is $g$-torsionfree and $N \in \rgr B$.  
We often use the facts that  
\begin{equation}
\label{equ:RtoB}
\Ext_R^i(M, N) \cong \Ext_B^i(M/Mg, N)\ \quad \text{and}\ \quad
\Ext_{\mc X}^i(M,N) \cong \Ext_{\rqgr B}^i(M/Mg, N).
\end{equation}
See \cite[Lemma 4.7]{blowdownI}.

With this background in hand, we now prove some needed preliminary results.
First, we show that maximal Cohen-Macaulay $g$-divisible submodules of $R_{(g)}$ have a nice characterization.

\begin{proposition}
\label{prop:MCMchar}
  Let $M \subseteq R_{(g)}$ be a nonzero $g$-divisible finitely generated module over an elliptic algebra $R$.   Then $M$ is MCM if and only if 
$\overline{M}$ is a saturated $B$-module.
\end{proposition}

\begin{proof}
We first claim that $M$ is MCM over $R$ if and only if $\overline{M}$ is MCM over $B$.   
Apply $\Hom_R(-, R)$ to the exact 
sequence $0 \to R \overset{\timesg}{\to} R \to B \to 0$ and consider the corresponding long exact sequence  
\begin{equation}
\label{eq:reduceRtoB}
\dots \too \Ext^i_R(M, R)[-1] \overset{\timesg}{\too} \Ext^i_R(M, R) \too \Ext^i_{B}(\overline{M}, B)  \too \Ext^{i+1}_R(M, R)[-1] \too \dots
\end{equation}
where we have used \eqref{equ:RtoB}.  Note that $\GKdim(M) = 3$ and $\GKdim(\overline{M}) = 2$.  If $M$ is MCM over $R$ then it follows from 
\eqref{eq:reduceRtoB} that $\overline{M}$ is MCM over $B$.  Conversely, if 
$\overline{M}$ is MCM over $B$, then for all $i \geq 1$ multiplication by $g$ gives a surjection of right $R$-modules $\Ext^i_R(M, R)[-1] \to \Ext^i_R(M, R)$.  
Since $\Ext^i_R(M,R)$ is left bounded, this forces $\Ext^i_R(M,R) = 0$.  
So $M$ is MCM over $R$, and the claim is proved.

To finish the proof we show that for any finitely generated graded $B$-submodule $N \subseteq Q_{gr}(B)$, $N$ is MCM over $B$ if and only if 
$N$ is saturated.   By multiplying by a homogeneous element of $B$ to clear denominators we can assume that $N \subseteq B$.  If $N$ is saturated, 
then $B/N$ is pure of GK dimension 1.  The critical $B$-modules 
of dimension $1$ are the shifted point modules \cite[Lemma 2.8]{RSSshort} and so  
 $B/N$ has a finite filtration with factors that are shifted point modules.  Since a point module is CM by 
\cite[Lemma 3.3]{blowdownI}, $B/N$ is CM of GK dimension 1.  Now $\Ext^i_B(N, B) \cong \Ext^{i+1}_B(B/N, B) = 0$ for $i \geq 1$, so $N$ is MCM of GK dimension 2.  Conversely, if $N$ is not saturated, then there is a graded module $N'$ with $N \subsetneqq N' \subseteq B$, where $N'/N$ is finite dimensional  and $N' = \omega \pi(N)$ is saturated.  
Since $N'$ is MCM by the argument above, $\Ext^i_B(N', B) = 0$ for $i \geq 1$.  Hence $\Ext^1_B(N,B) \cong \Ext^2_B(N'/N, B)$.  But using that $B$ is AS Gorenstein of dimension 2, we get $\Ext^2_B(N'/N, B) \neq 0$.  Thus $\Ext^1_B(N, B) \neq 0$ and  $N$ is not CM.
\end{proof}

Note that  a MCM $R$-module $M$ is automatically reflexive:
that is, $\Hom_R(\Hom_R(M, R), R) = M$.  This follows from the Gorenstein spectral sequence \cite[Theorem 2.2]{Lev1992}.

Next, we show that Ext groups in the categories $\rgr R$ and $\X=\rqgr R$ are 
equal in  certain  circumstances.
\begin{proposition}
\label{prop:equalhom}
Let $M, N \in \rgr R$. 
\begin{enumerate}
\item If $N$ is saturated and has zero socle, then $\Hom_R(M, N) =  \uHom_\X(M, N)$.
\item If $N \subseteq R_{(g)}$ is $g$-divisible and MCM, then $\Ext^1_R(M, N) =  \uExt^1_\X(M, N)$.
\end{enumerate}
\end{proposition}
\begin{proof}
(2) We claim that $\Ext^2_R(\kk, N) = 0$. Consider the exact sequence
$0 \to N[-1] \overset{\timesg}{\to} N \to \overline{N} \to 0$ and the corresponding exact sequence
\begin{equation}
\label{equ:monday}
\Ext^1_R(\kk, \overline{N}) \too \Ext^2_R(\kk, N)[-1] \too \Ext^2_R(\kk, N).
\end{equation}
Suppose that $\Ext^1_R(\kk, \overline{N}) \neq 0$.  Let $\overline{N} \subseteq P$ represent a nonzero class $\alpha$ in this extension group, where $P/\overline{N} \cong \kk$.  Necessarily $\overline{N}$ is essential in $P$.
If $P$ is not a $B$-module, then $Pg \neq 0$.  Since $\overline{N}g = 0$, in this case $Pg \cong \kk$.  By essentiality, $Pg \cap \overline{N} \neq 0$. This contradicts the fact  that $\overline{N} \subseteq Q_{gr}(B)$, so that $\overline{N}$ has no socle.  Hence 
$P \in \rgr B$.  This shows that $\alpha \in \Ext^1_B(\kk, \overline{N}) = 0$, because $\overline{N}$ is saturated by Proposition~\ref{prop:MCMchar}.  This contradiction proves that 
$\Ext^1_R(\kk, \overline{N}) = 0$.

Now \eqref{equ:monday} implies that there is a graded injective vector space map $\Ext^2_R(\kk, N)[-1] \hookrightarrow
 \Ext^2_R(\kk, N)$.  However, 
$\Ext^2_R(\kk, N)$ is finite dimensional  by the $\chi$ condition.  This forces $\Ext^2_R(\kk, N)= 0$ as claimed.

Next note that   $\Ext^1_R(\kk, N) = 0$, since   $g$-divisibility implies that $N$ is a saturated $R$-module.  Thus 
$\Ext^i_R(M/M_{\geq n}, N) = 0$ for $i = 1, 2$ and for all $n \in \ZZ$.  Given that 
$\uExt^1_\X(M, N) = \lim_{n \to \infty} \Ext^1_R(M_{\geq n}, N)$, the result follows from the standard long exact sequence. 

(1) Again $\Ext^1_R(\kk, N) = 0$ since $N$ is saturated, and $\Hom_R(\kk, N) = 0$ since $N$ has no socle.
Now an   argument  analogous to the proof of part (2) gives the result.  
\end{proof}

  \begin{notation}\label{line-notation}
 Since line modules play a vital role in noncommutative monoidal transformations, we recall some of their properties
 from \cite[Section~5]{blowdownI}.  If $L$ is a right (left) line module,
 then $L^\vee = \Ext^1_R(L, R)[1]$ is a left (right) line module, referred to as the {\em dual line module}.  Moreover, $L \cong L^{\vee \vee}$.
 There is a unique right (left) ideal $J$ of $R$ such that $L \cong R/J$, and, since $\Ext^1_R(L, R)_1 = \kk$, a unique module $M \subseteq Q_{\gr}(R)$ 
  such that $R \subseteq M$ with $M/R \cong L[-1]$.  We refer to $J$ as the {\em line ideal} of $L$ and $M$ as the {\em line extension} of $L$.  
 Note that line ideals and line extensions are MCM, and thus reflexive by
\cite[Lemma 5.6(2)]{blowdownI}.
 Further, the line ideal of $L^\vee$ is the reflexive dual of the line extension of $L$ \cite[Lemma 5.6(3)]{blowdownI}.   \end{notation}

For $M, N \in \X$, following Mori and Smith one defines the \emph{intersection number}
\[(M\dotms N  )=\sum_{n\geq 0} (-1)^{n+1} \dim_\kk \Ext^i_{\rqgr R}(M,N)\] (see \cite[Definition~8.4]{MS} or \cite[Definition~6.1]{blowdownI}).   The following result follows easily from this definition.

\begin{lemma}\label{lem6} 
Let $0 \to M' \to M \to M'' \to 0 $ be a short exact sequence in $\X$ and let $N \in \X$.
Then 
\[ (M \dotms N) = (M' \dotms N) + (M'' \dotms N)\quad \text{and}\quad   (N \dotms M ) = (N \dotms M') + (N \dotms M''). \qed  \] 
\end{lemma} 

The intersections between lines are especially important.
By \cite[Lemma~6.8 and Corollary~6.6]{blowdownI} 
and the fact that $\X=\rqgr R$ is smooth, we get the following identities. If  $L, L'$ are $R$-line modules,  then
\beq\label{peppermint}
(L \dotms L') = \grk \Ext^1_R(L, L') - \grk \Hom_R(L, L') = \grk \uExt^1_{\X}(L, L') - \grk \uHom_{\X}(L, L'),
\eeq
where $\grk V$ denotes  the torsion-free rank of $V$ as a $\kk[g]$-module.

Next, we give several versions of Serre duality for $\X$.  In the next two results, for a vector space $V$, the 
notation $V^*$ means the usual vector space dual $\Hom_\kk(V, \kk)$.

\begin{lemma}\label{lem1}
 The space $\sX$ has cohomological dimension  
  $\cd \X = 2$ in the sense that  $\Ext^i_{\X}( R, \blank) = 0$ for $i \geq 3$.  
Further, $R[-1]$ is the dualizing sheaf for $\X$ in the sense of \cite[(4-4)]{YZ}.
Finally,    
\beq \label{CM}
\Ext^i_\X( \blank, R[-1]) \cong  \Ext^{2-i}_\X(R, \blank)^* \quad \mbox{ for all $i$.}
\eeq
\end{lemma}

\begin{proof}  By \cite[Proposition~4.3]{blowdownI},   $R$ is AS-Gorenstein,  in the sense of  \cite[Definition~2.1]{blowdownI}.
 Thus there exists  $e\in \ZZ$ so that $\Ext^i_R(\kk, R) = \delta_{i,3} \kk[e]$ in $\rgr R$.
Then by \cite[Corollary 4.3]{YZ}, $R[-e]$ is the dualizing sheaf for $\X$, and $\Ext^i_\X( \blank, R[-e]) \cong \Ext^{2-i}_\X(R, \blank)^*$.
Using \eqref{equ:RtoB} we have a long exact sequence
\[ \cdots \to \Ext^i_\X(R, R[-1]) \to \Ext^i_\X(R, R) \to \Ext^i_{\rqgr B}(B, B) \to \cdots\]
Together with the fact that $B$ is the dualizing sheaf for $\rqgr B$, 
 (equivalently, $\sO_E$ is the dualizing sheaf for $\coh E$) this implies  that $e = 1$.
\end{proof}
 
For $g$-torsionfree modules we have the following stronger version of this result.

\begin{proposition}\label{prop2} 
Assume that  $N\in \gr R$ is   $g$-torsionfree.
Then $\Ext_\X^i(N, \blank) = 0$ for $i \geq 3$. For $0 \leq i \leq 2$, \
$ \Ext^i_\X(N, \blank)^* $ and $ \Ext^{2-i}_\X(\blank, N[-1])$
  are naturally isomorphic functors  from $\X  \to \rMod \kk$.
\end{proposition}

\begin{proof}
We first show that if $G \in \X$ is   $g$-torsion, then 
$\Ext^{\geq 2}_\X(N, G) = 0.$
Since $G$ is noetherian there is some $n$ such that $G g^n = 0$ and,
by induction, it suffices to prove that $\Ext^{\geq 2}_\X(N, G) = 0$ holds  when $Gg=0$. 
Let $i \geq 2$ and let $B = R/Rg$.
By \eqref{equ:RtoB} we have $\Ext^i_\X(N, G) \cong \Ext^i_{\rqgr B}(N/Ng, G)$, which is zero
since $ \rqgr B  \simeq \coh(E)$  for the  elliptic curve $E$.

Now let $F \in \X$ be noetherian and $g$-torsionfree.
Let $i \geq 3$.  As in the proof of \cite[Lemma~6.8]{blowdownI}, 
$\Ext^i_\X(N, F) $ is $g$-torsionfree and hence by \cite[Lemma~6.2]{blowdownI}  the natural map
\[\Ext^i_\X(N, F) \to \Ext^i_\X(N, F) \otimes_{\kk[g]} \kk[g, g^{-1}] \cong \Ext^i_{R^\circ}(N^\circ, F^\circ) \otimes_\kk \kk[g, g^{-1}] \] 
is injective.
As $R^\circ$ is CM,  the proof of \cite[Lemma~6.8]{blowdownI} implies that $R^\circ$ has   $\gldim R^\circ  \leq 2$.  
Thus $\Ext^i_\X(N, F) = 0$.

Let $M \in \X$ be arbitrary; then there is an exact sequence $0 \to G \to M \to F \to 0$ in $\X$, where $G$ is $g$-torsion and $F$ is $g$-torsionfree.  From the previous two paragraphs, we see that $\Ext^{\geq 3}_\X(N, M)  =0$.

By \eqref{CM},  $\Ext^2_\X(N, R[-n-1])^* \cong \Hom_\X(R, N[n])$. Thus,   if $N \neq 0$ then $\Ext^2_\X(N, \blank)  \neq 0$.
Let $C = \bigoplus_{n\geq 0} \Ext^2_\X(N, R[-n])^*$, which is a right $R$-module.
Using \eqref{CM}, we compute that
\[ C \cong \bigoplus_{n\geq 0} \Hom_\X(R, N[n-1]) = \omega\pi N[-1]_{\geq 0}.\] 
By Proposition~\ref{prop:equalhom}, $\uHom_\X(N, R)$ and $\uExt^1_\X(N, R)$ are left bounded.
Therefore,  by  \cite[Theorem~2.2]{YZ}, $\Ext^i_\X(\blank, C) \cong \Ext^i_\X(\blank, N[-1])$ is naturally isomorphic to $\Ext^{2-i}_\X(N, \blank)^*$.
\end{proof}

We remark that, as $\X$ is smooth,  it is surely the case that  $\hd \X = 2$   
and so  Proposition~\ref{prop2} will hold for all modules in $\X$; however we have not been able to prove this.

Given a $\mb{Z}$-graded vector space $M = \bigoplus_n M_n$,   write $M^*$ for the \emph{graded dual} 
$M^*= \bigoplus_n \Hom_\kk(M_{-n}, \kk)$.   The previous result has the following immediate extension to graded Ext groups.
\begin{corollary}
\label{cor:underline}
Let $N \in \rgr R$ be $g$-torsionfree.  For all $M \in \rgr R$ and $0 \leq i \leq 2$,
\[
\uExt^i_{\X}(N,M)^* \cong \uExt^{2-i}_{\X}(M, N)[-1]
\]
as graded vector spaces. \qed
\end{corollary}

We also have the following corollary for intersection numbers of lines. 

\begin{corollary}\label{cor:symm}
If $L, L'$ are $R$-line modules, then
\begin{enumerate}
\item  
$  (L \dotms L'[k] ) = ( L \dotms L') $
 for all $k \in \ZZ$. 
 \item Moreover  $ (L \dotms L') = (L' \dotms L) = (L^\vee \dotms (L')^\vee) = ((L')^\vee \dotms L^\vee).$
\end{enumerate} 
\end{corollary}
\begin{proof}  (1)   This  is  \cite[Proposition~6.4(1)]{blowdownI}. 

(2) The first and third equalities   follow from  Proposition~\ref{prop2}, together with part~(1).
The second equality  follows from \cite[Lemma~5.6(4)]{blowdownI} and \eqref{peppermint}.  
\end{proof}

Finally, we compute the cohomology of a line module.

\begin{lemma}\label{lem8} 
Let $L_R$ be a line module.  For all $n \in \ZZ$, we have
\[ \dim H^0(\X, L[n]) = \max(n+1, 0), \quad   \dim H^1(\X, L[n]) = \max(-n-1, 0),
  \quad \mbox{and} \quad   H^2(\X, L[n]) = 0.\]
Thus
$ (R \dotms L[n]) = -n-1 \  $ and $ \ (L[n] \dotms R) = -n.$
\end{lemma}

\begin{proof}
First, $\uHom_\X(L, R) = \Hom_R(L, R) = 0$ by Proposition~\ref{prop:equalhom}.
As $L$ is Goldie torsion,  $H^2(\X, L[n]) = 0$ follows from  \eqref{CM}.  Thus, by
again using Proposition~\ref{prop:equalhom} and \cite[Lemma~5.6(1)]{blowdownI}, 
\[\dim_\kk \Ext^1_\X(L[n], R) = \dim_\kk \Ext^1_{\rgr R}(L[n], R) =  \dim_\kk \Ext^1_R(L, R)_{-n} = \begin{cases} -n & n \leq 0 \\ 0 & \mbox{else.} \end{cases}.\]
By \eqref{CM}, we have $\dim \Ext^2_\X(L[n], R) = \dim \Hom_\X(R, L[n-1])$.
Note that a line module $L$ is saturated and $g$-torsionfree \cite[Lemmata~5.2 and~5.6(5)]{blowdownI}.  In particular, $L$ has 
zero socle.  By Proposition~\ref{prop:equalhom}, 
\[ \dim_{\kk} \Hom_\X(R, L[n-1]) = \dim_{\kk} \Hom_{\rgr R}(R, L[n-1] )
= \begin{cases} n & n \geq 0 \\ 0 & \mbox{else}. \end{cases}\]

Now using \eqref{CM} and the definition of the intersection product 
all of the claims follow. \end{proof}

%%%%%%%%%%%%%%%%%%%
  %%%%%%%%%%%%%%%%%%%

\section{Twisting sheaves for elliptic algebras of degree nine}\label{RECOG1}
  
In this section and the next, we study elliptic algebras $R$ of degree $9$ and develop a theory
 which will allow  us to recognise when  such an algebra   is a  Sklyanin elliptic algebra $T=S^{(3)}$ for $S=S(E,\sigma)$. 
  
If $T$ is such a  Sklyanin elliptic algebra, then  $\rqgr T \simeq \rqgr S$ and as $\rqgr S$ is a noncommutative projective plane
 it has   objects $\{ \mc{O}(k)  \, | \,  k \in \mb{Z} \}$ which play the role of the Serre twisting sheaves on $\PP^2$.
The objects $\mc{O}(k)$ correspond to the right $T$-modules $\bigoplus_{n  \in \ZZ} S_{3n+k}$; in particular, we have the $T$-modules $H = \bigoplus_{n\geq1} S_{3n-1} $ and 
$M = \bigoplus_{n \geq 0} S_{3n+1} $.  
Ignoring degree shifts, it is easy to check that  $ \End_T(H)  = \End_T(M) = T,$ while 
$ H^* = TS_1, $  and $M^* = TS_2.$  Here, we have returned to the standard notation $H^*=\Hom_T(TS_1,T)$, etc.   Similarly, 
$ \Hom_T(H, M) = MH^* = S_1 T S_1 = H$ and 
$ \Hom_T(M,H) = M.$
Consider $H \oplus T \oplus M$ as a column vector and $\End_T(H \oplus T \oplus M) $ 
as a subalgebra of $M_{3\times 3}(S)$.
Taking Hilbert series of the above, and with the correct degree shifts, it follows routinely that:
  
 \beq\label{hilbmat}
\hilb  \End_T(H \oplus T \oplus M) = 
\begin{pmatrix}  
 \frac{  1 + 7s + s^2}{(1-s)^3}  &    \frac{6s+3s^2}{(1-s)^3}&   \frac{3s+6s^2}{(1-s)^3}  \\
  \frac{3+6s}{(1-s)^3}  &   \frac{1 + 7s + s^2}{(1-s)^3} &   \frac{6s+3s^2}{(1-s)^3}  \\
   \frac{6+3s}{(1-s)^3}&   \frac{3+6s}{(1-s)^3}  &   \frac{1 + 7s + s^2}{(1-s)^3}
 \end{pmatrix}. 
\eeq
  
The main result of this section, Proposition~\ref{prop:recT}, gives sufficient conditions for a degree 9 elliptic algebra $T$ to have right modules  $H, M$  for which  \eqref{hilbmat} holds, and hence such that their  images in $\rqgr T$ play the role of $\mc{O}(-1)$ and $\mc{O}(1)$.  

Before stating that result we need a few observations from \cite{AV}. 
Given an elliptic algebra $T$, set $B = T/gT=B(E, \sN, \tau)$; thus $B=\bigoplus B_n$, where 
$B_n=H^0(E, \mc{N}_n)$ and $\mc{N}_n = \mc{N}\otimes\cdots\otimes \mc{N}^{\tau^{n-1}}$, under the notation  $\mc F^\tau = \tau^* \mc F$
  for a sheaf $\mc F\in \coh(E)$.
  The isomorphism classes of $B$-point modules  \label{point-defn} are in one-to-one correspondence 
with the closed points of $E$; explicitly, if $p \in E$ has skyscraper sheaf $\mc{O}_p\cong \mc{O}_E/\mc{I}_p$, 
then $p \in E$ corresponds to the point module $M_p = \bigoplus_{n \geq 0} H^0(E, \mc{O}_p \otimes \mc{N}_n)\cong B/I_p$   
for the  \emph{point ideal}
$I_p = \bigoplus_{n \geq 0} H^0(E, \mc{I}_p \otimes \mc{N}_n) \subseteq B$. 
The  images $\pi(M_p)$ of the point modules $M_p$ are the simple objects in $\rQgr B$.  We will also frequently use the 
fact that $M_p[n]_{\geq n}\cong M_{\tau^n p}$ for any $n$ (see, for example, \cite[(3.1)]{blowdownI}).
Of course simple objects in $B \lQgr$ are also parameterised by closed points of $E$, and we denote these left point modules by 
$M^\ell_p $, where $M^\ell_p := \bigoplus_{n \geq 0} H^0((\sO_p)^{\tau^{n-1}} \otimes \sN_n)$.

If $M$ is  a $g$-torsionfree right $T$-module such that $\pi (M/Mg)$ has finite length with composition factors $\pi(M_{p_1}), \dots, \pi(M_{p_n})$, we say that the {\em divisor of $M$} is $\Div M = p_1 + \dots + p_n$.   In particular, if $L$ is a line module then $\Div L$ is a single point.  
  Finally,   graded vector spaces $V, V'$ are  called {\em numerically equivalent}, written $V \equiv V'$, if $\hilb V = \hilb V'$.
We often write $V\equiv \hilb V.$  

\begin{notation}\label{notation:recT}  Let $T$ be an elliptic algebra of degree $9$, with $B = T/gT = B(E, \mc{N}, \tau)$.
We will be interested in the following three conditions on $T$.
\begin{enumerate}
\item [$(i)$]  $T$ has a $g$-divisible right ideal $H$ with
\[
\overline{H} = \bigoplus_{n \geq 0} H^0(E, \mc{N}_n(-a-b-c))
\]
 for some points $a, b, c \in E$,  such that $\hilb \End_T(H) = \hilb T$.
\item[$(ii)$] $T$ has a $g$-divisible right module $M$ with $T \subseteq M \subseteq T_{(g)}$
such that 
\[
\overline{M} = \bigoplus_{n \geq 0} H^0(E, \mc{N}_n(d+e+f))
\]
 where   $d, e, f \in E$,  such that $\hilb \End_T(M) = \hilb T$.
\item[$(ii)'$] 
$T$ has a $g$-divisible  left ideal $H^{\vee}$  
with 
\[
\overline{H^{\vee}} = \bigoplus_{n \geq 0} H^0(E, \mc{N}_n(-\tau^{-n}d - \tau^{-n}e - \tau^{-n}f))
\]
 where $d, e, f \in E$, and such that $\hilb \End_T(H^{\vee}) = \hilb T$.
\end{enumerate}
 
\end{notation}

\begin{proposition} 
\label{prop:recT}
Let $T$ be an  elliptic algebra of degree $9$,   
set $B = T/gT = B(E, \mc{N}, \tau)$ and
 consider the conditions for Notation~\ref{notation:recT}.  
\begin{enumerate}
\item  Assume that \eqref{notation:recT}$(i,ii)$ hold and set  $N^{(1)} = H$, $N^{(2)} = T$, $N^{(3)} = M$.  Then 
\begin{equation}\label{equ:recT1}
\overline{ \Hom_T (N^{(i)}, N^{(j)}) } = \Hom_B(\overline{N^{(i)}}, \overline{N^{(j)}})
\qquad \text{ for 
all }  i,j \in \{1, 2, 3 \}.
\end{equation}
Moreover \eqref{hilbmat} holds;  that is,
\begin{equation}\label{equ:recT}
\hilb 
\bigg( \Hom_T(N^{(j)}, N^{(i)})  \bigg)_{ij}
= 
\begin{pmatrix}
 \frac{1 + 7s + s^2}{(1-s)^3}&  \frac{6s+3s^2}{(1-s)^3}& \frac{3s+6s^2}{(1-s)^3}  \\
\frac{3+6s}{(1-s)^3}  &\frac{1 + 7s + s^2}{(1-s)^3} & \frac{6s+3s^2}{(1-s)^3}  \\
 \frac{6+3s}{(1-s)^3}& \frac{3+6s}{(1-s)^3}  & \frac{1 + 7s + s^2}{(1-s)^3}
 \end{pmatrix}. 
\end{equation}
\item Conditions \eqref{notation:recT}$(ii)$ and \eqref{notation:recT}$(ii)'$ are equivalent.  
Indeed,  $M$ satisfies \eqref{notation:recT}$(ii)$ if and only if 
 $H^{\vee} = \Hom_T(M, T)$ 
satisfies \eqref{notation:recT}$(ii)'$. Similarly $H^{\vee}$ satisfies \eqref{notation:recT}$(ii)'$
 if and only if $M = \Hom_T(H^{\vee}, T)$ satisfies \eqref{notation:recT}$(ii).$
\end{enumerate}
\end{proposition}

\begin{proof}
(1)
The    $N^{(i)}$ are $g$-divisible by hypothesis, and hence so is each    $\Hom_T(N^{(i)}, N^{(j)})$, by \cite[Lemma~2.12]{RSSlong}.
By hypothesis, we may write 
$\overline{N^{(i)}} = \bigoplus_{n \geq 0} H^0(E, \mc{G}_i \otimes \mc{N}_n)t^n$  where    $\mc{G}_1 = \mc{O}_E(-a-b-c)$, 
$\mc{G}_2 = \mc{O}_E$, and $\mc{G}_3 = \mc{O}_E(d + e + f)$.
Then \cite[Lemma~2.3]{blowdownI} implies that 
\begin{equation}\label{equ:recT00}
\Hom_B(\overline{N^{(j)}}, \overline{N^{(i)}}) \ = \ 
\bigoplus_{n \geq 0} H^0(E, (\mc{G}_j^{\tau^n})^{-1} \mc{G}_i \otimes \mc{N}_n)t^n
\  \subseteq  \  k(E)[t, t^{-1}; \tau]
\end{equation}
for all $i, j $.  A straightforward calculation, using  \cite[Lemma~2.2]{blowdownI}, 
shows that    $\hilb  
\big( \Hom_B(\overline{N^{(j)}}, \overline{N^{(i)}}) \big)_{ij}$
 is equal to $(1-s)$ times the matrix on the right hand side of   \eqref{equ:recT}.   
Since $\Hom_T(N^{(j)}, N^{(i)})$ is $g$-divisible,  
this shows that  equality in the $(i,j)$-entry of  \eqref{equ:recT} is equivalent to having 
$\overline{\Hom_T(N^{(j)}, N^{(i)})} = \Hom_B(\overline{N^{(j)}}, \overline{N^{(i)}})$.  Thus we just need to prove \eqref{equ:recT1}.

Using \eqref{equ:RtoB} there is an exact sequence
\begin{equation}
\label{equ:hi}
0 \too \Hom_T(N^{(j)}, N^{(i)})[-1] \overset{\timesg}{\too} \Hom_T(N^{(j)}, N^{(i)}) \too \Hom_{B}(\overline{N^{(j)}}, \overline{N^{(i)}}) \overset{\alpha}{\too} \Ext^1_T(N^{(j)}, N^{(i)})[-1]  
\end{equation}
and so for each $i,j$ is is sufficient to prove that 
$\Ext^1_T(N^{(j)}, N^{(i)}) = 0$.  Now each
$N^{(i)}$ is $g$-divisible, and it is easy to see that $\overline{N^{(i)}}$ is saturated by Hypotheses \eqref{notation:recT}(i,ii).  So each $N^{(i)}$ is MCM by Proposition~\ref{prop:MCMchar}.  Therefore,
 by Proposition~\ref{prop:equalhom} and Corollary~\ref{cor:underline}, we have 
\[
\Ext^1_T(N^{(j)}, N^{(i)}) = \uExt^1_{\X}(N^{(j)}, N^{(i)}) = 
\uExt^1_{\X}(N^{(i)}, N^{(j)})^*[1] =  \Ext^1_T(N^{(i)}, N^{(j)})^*[1].
\]
Thus for each $(i,j)$ we have $\Ext^1_T(N^{(j)}, N^{(i)}) = 0$ 
if and only if $\Ext^1_T(N^{(i)}, N^{(j)}) = 0$.

By assumption, $T \equiv \End_T(M) \equiv \End_T(H)$ and so 
  \eqref{equ:recT1}  holds  when $i = j$.
Trivially,  $\Ext^1_T(T, N^{(i)}) = 0$, and so 
$\Ext^1_T(N^{(i)}, T) = 0$  also holds for all $i$.  (This also follows 
from the CM property of $N^{(i)}$.)
As  $\Ext^1_T(M,H) = 0$ $\iff$ 
  $\Ext^1_T(H,M) = 0$, in order to prove part~(1) it  just   remains to prove that 
  $\Ext^1_T(M,H) = 0$.  
  
 Since  $\Ext^1_T(M,T)=0$,   we have the 
  exact sequence 
\begin{equation}
\label{equ:lolz}
0 \too \Hom_T(M,H) \too \Hom_T(M, T) \overset{\beta}{\too} \Hom_T(M, T/H) \too \Ext^1_T(M, H) \too   0.
\end{equation}
We would like to understand the  Hilbert series of  $\Hom_T(M, T/H)$. For 
this,  we  consider also the following exact sequence, where  we are applying \eqref{equ:RtoB}, 
\begin{equation}
\label{equ:rofl}
0 \too \Hom_T(M, T/H)[-1] \overset{\timesg}{\too} \Hom_T(M, T/H) \too
 \Hom_B(\overline{M}, \overline{T}/\overline{H}) \too \dots.
\end{equation}

First we 
prove that $\Hom_B(\overline{M}, \overline{T}/\overline{H})_0 = 0$.  
By \eqref{notation:recT}$(i)$,  there is a surjection  $\beta: \overline{T}/\overline{H} \twoheadrightarrow M_a$.  
But any nonzero, degree $0$ map 
$\theta: \overline{M} \to \overline{T}/\overline{H}$ is surjective, since $\overline{T}/\overline{H}$ is cyclic. 
 Thus $\rho=\beta\circ\theta$  gives  a surjection
$\rho: \overline{M} \twoheadrightarrow M_a$.  Since  $\overline{M}$ is saturated, it clearly follows from the equivalence $\rqgr B\simeq \coh E$ that
$\ker \rho = \bigoplus_{n \geq 0} H^0(E, \mc{N}_n(d +e + f -a)).$   Therefore, by  \cite[Lemma~2.2]{blowdownI},  
 $\ker \rho$ is generated in degree $0$. On the other hand, by construction, 
 $(\ker \rho)_0 = (\ker \theta)_0$  and so  $\ker \rho = \ker \theta$.  This is impossible, so no such
 $\theta$ can exist.  Thus $\Hom_B(\overline{M}, \overline{T}/\overline{H})_0 = 0$, as  desired.

Now   $(\overline{T}/\overline{H})_{\geq 1}$ is filtered by the shifted point modules  $(M_{\tau(a)})[-1]$, $(M_{\tau(b)})[-1]$, and  $(M_{\tau(c)})[-1]$.  Since  $\hilb \Hom_B(\overline{M}, M_p) =1/(1-s)$ for any point $p$ such that $\Hom_B(\overline{M}, M_p)\not=0$,  it follows that  
$\hilb \Hom_B(\overline{M}, \overline{T}/\overline{H}) \leq (3s)/(1-s)$.  Using   \eqref{equ:rofl}, we then get  
$\hilb \Hom_T(M, T/H) \leq (3s)/(1-s)^2.$

Now  $\hilb \Hom_B(\overline{M}, \overline{H}) = (3s + 6s^2)/(1-s)^2$, by \eqref{equ:recT00}.  So, considering \eqref{equ:hi} for 
$(j,i)=(3,1)$,  gives the upper bound 
$\hilb \Hom_T(M,H) \leq (3s + 6s^2)/(1-s)^3$.   On the other hand,
\[
(3s)/(1-s)^2 + (3s + 6s^2)/(1-s)^3 = (6s + 3s^2)/(1-s)^3 = \hilb \Hom_T(M, T),
\]
where the last equality follows since  we have already proven that the $ (3,2)$ entry of 
\eqref{equ:recT} is  correct.
By considering the first three terms of \eqref{equ:lolz} 
and the  upper bounds for  $\hilb \Hom_T(M, T/H)$ and $\hilb \Hom_T(M,H)$, this forces both bounds to be equalities.  
In particular,  $\hilb \Hom_T(M,H) = (3s + 6s^2)/(1-s)^3=\hilb \Hom(T)-\hilb (M, T/H)$ and so 
 the first three terms of \eqref{equ:lolz} form a short exact 
sequence. Thus  $\Ext^1_T(M,H) = 0$, and the proof of part~(1) is  complete.

(2) Suppose that Hypotheses \eqref{notation:recT}$(ii)$ holds and 
define $H^{\vee} = M^*=  \Hom_T(M, T)$.  Since $M$ is MCM, it is reflexive, so $M = \Hom_T(H^{\vee}, T)$.  
It follows that
 \[\End_T(M) = \{ x \in Q_{gr}(T)   \, | \,  xM \subseteq M \} = \{ x   \, | \,  H^{\vee} x M \subseteq T \}\]
and, similarly,  $\End_T(H^{\vee}) = \{ y \in Q_{gr}(T)   \, | \,  H^{\vee} y \subseteq H^{\vee} \} =
 \{ y   \, | \,  H^{\vee} y M \subseteq T \}$.
Thus $\End_T(M) = \End_T(H^{\vee})$.  Since $\End_T(M) \equiv T$ by assumption, we also get $\End_T(H^{\vee}) \equiv T$. 

Since $T$ is $g$-divisible,  the module $P = M/T$ is $g$-torsionfree.  
By Hypotheses \eqref{notation:recT}$(ii)$,  $P/Pg = \overline{M}/B \equiv (2+s)/(1-s)$ is filtered by the shifted right point modules $M_{\tau(d)}[-1] \cong (M_d)_{\geq 1}, M_e$, and $M_f$. Therefore,   \cite[Lemma~5.4]{blowdownI} applies to $P$.  
By that lemma, if $N = \Ext^1_T(P, T)$, then $N/Ng$ is filtered by the shifted left point modules $M^{\ell}_{\tau^{-1}d}$, $M^{\ell}_{\tau^{-2}e}[-1]$, and $M^{\ell}_{\tau^{-2}f}[-1]$.
Now consider the exact sequence
\[
0 \to H^{\vee}   \too T \overset{\beta}{\too} N \too \Ext^1_T(M,T) , 
\]
 and note that $ \Ext^1_T(M,T)=0$  was proved within  the proof of  part~(1).  Thus  $N \cong T/H^{\vee}$,    and so  
 $N/Ng \cong B/\overline{H^{\vee}}$, as $H^{\vee}$ is $g$-divisible.  Since we know the point modules in a filtration of $N/Ng$,
  this forces 
$\overline{H^{\vee}} = \bigoplus_{n \geq 0} H^0(E, \mc{N}_n(-\tau^{-n}d - \tau^{-n}e - \tau^{-n}f))$.
Thus $(ii)$ implies $(ii)'$.

The converse $(ii)'$ $\Rightarrow$ $(ii)$ is proved similarly, using the left-sided versions of 
 \cite[Lemmata~4.5 and~5.4]{blowdownI}.  The relevant exact sequence is 
\[
0 \to T \to \Hom_T(H^{\vee}, T) \to \Ext^1_T(T/H^{\vee}, T) \to 0,
\]
where  there is no extra $\Ext^1$ term to worry about.
We leave the details to the reader. 
\end{proof}

%%%%%%%%%%%%%%%%%%%
  %%%%%%%%%%%%%%%%%%%

\section{Recognition I:  Recognising Sklyanin elliptic algebras}\label{RECOG2}
 
We continue to consider when an elliptic algebra $T$ of degree $9$ can be shown to be the 
$3$-Veronese of a Sklyanin algebra. 
Let  $T$ be such an elliptic algebra with $T/gT \cong B(E, \mc{N}, \tau)$, and recall that we are always assuming
 that $\rqgr T$ is smooth. Suppose, further, that  $T$ has right modules $H $ and $M$ satisfying the conditions from
 Notation~\ref{notation:recT}.  The main result of this section is that, under an additional condition on 
 $\Div T/H$ and $\Div M/T$, then $T=S^{(3)}$ is the $3$-Veronese of a Sklyanin algebra $S=S(E,\sigma)$. 
 Note that, as discussed at the beginning of Section~\ref{RECOG1},  this also implies that $T \cong \End_T(H) \cong \End_T(M)$. 
  Neither of these conclusions  is  obvious a priori.

  To explain the condition on the divisors of $T/H$ and $M/T$, suppose for the moment that $T$ is a Sklyanin elliptic algebra; that is, $T = S^{(3)}$ where $S$ is the Sklyanin algebra with $S/gS = B(E, \sL, \sigma)$ and $\sL_3 = \sL \otimes \sL^{\sigma} \otimes \sL^{\sigma^2} \cong \sN$.
  If $0 \neq x \in S_1 = H^0(E, \sL)$ vanishes at $a,b,c \in E$, then $H := xS_2 T$ satisfies $\overline{H} = \bigoplus H^0(E, \sN_n(-a-b-c))$.  Recall that the isomorphism class of an invertible sheaf on $E$ is determined by its degree and its image under  the natural map from $\Pic E \to E$ \cite[Example~IV.1.3.7]{Ha}.   
  Thus $\sL \cong \sO_E(a+b+c)$  and so 
  \beq\label{need2} \sN \cong \sO_E(\tau^{-1}a +b+c)^{\otimes 3}. \eeq
Further, if we set $H^\vee = TS_2 x$ then
$ \overline{H^\vee} = \bigoplus H^0(E, \sN(-\tau^{-n} d  -\tau^{-n} e - \tau^{-n} f))$, for $d = \sigma a$, $e = \sigma b$, and $f=\sigma c$.  
 Thus 
 \beq\label{need1} \sO_E(\tau a+b+c) \cong \sO_E(d+e+f). \eeq

The following theorem, which is the main result of this section, shows that the necessary conditions \eqref{need2}, \eqref{need1} for $T$ to be a Sklyanin elliptic algebra are also sufficient.

\begin{theorem}\label{thm:Sposs}
Let $T$ be an elliptic algebra of degree $9$, such that $\rqgr T$ is smooth,  and with $T/Tg\cong B(E, \sN, \tau)$.   
Suppose that $T$ satisfies the conditions from Notation~\ref{notation:recT}, where the points $a,b,c,d,e,f$ satisfy \eqref{need2} and \eqref{need1}.
Then  $T \cong S^{(3)}$, where  $S \cong S(E, \sigma)$ is a Sklyanin algebra    with $\sigma^3 = \tau$.
\end{theorem}

We prove this theorem by means of $\mb Z$-algebras.
Recall that a \emph{$\mb{Z}$-algebra} is a $\kk$-algebra $A$ of the form 
$\bigoplus_{(i,j) \in \mb{Z} \times \mb{Z}} A_{i,j}$ 
such that $A_{i,j}A_{k,\ell}  = 0$ if $j \neq k$, and $A_{i,j} A_{j, \ell} \subseteq A_{i, \ell}$.   
The $\mb{Z}$-algebra $A$ 
is \emph{lower triangular} if $A_{ij} = 0$ for all $i < j$, and \emph{connected} if $A_{n,n} = \kk$ for all $n$.

We say that a homomorphism $\phi: A \to A$ of a $\mb{Z}$-algebra $A$ has \emph{degree $d$} if 
$\phi(A_{i,j}) \subseteq A_{i+d, j+d}$ for all $i, j$. The 
$\mb{Z}$-algebra $A$ is called \emph{$d$-periodic} if it has an automorphism of degree $d$.  
A $1$-periodic $\mb{Z}$-algebra $A$ is also called \emph{principal}.
Given any $\mb{Z}$-graded algebra $B = \bigoplus_{n \in \mb{Z}} B_n$, we can form a 
$\mb{Z}$-algebra $A = \wh{B}$ by putting $A_{ij} = B_{j-i}$ for all $i,j$, with the natural 
multiplication induced by the multiplication of $B$.  Clearly $\wh{B}$ is principal.  Conversely,
 if a $\mb{Z}$-algebra $A$ is principal then there exists a $\mb{Z}$-graded algebra $B$ such
  that $A \cong \wh{B}$, by \cite[Proposition~3.1]{S-equiv}; further, the multiplication on $B$ is determined by the given degree 1 automorphism of $A$. 
Given a $\mb{Z}$-algebra $A$, $m \in \mb{Z}$ and $d \geq 2$, the corresponding
 \emph{$d^{\rm th}$-Veronese} of $A$ is the $\mb{Z}$-algebra $A' = A^{(d)}$ where 
$A'_{i,j} = A_{m + di, m + dj}$, with multiplication induced by the multiplication in $A$.  
Clearly $A$ has $d$ distinct $d^{th}$ Veronese algebras, depending on the choice of $m$.  
We call the $\mb{Z}$-algebra $A$ a \emph{quasi-domain} if for all $0 \neq x \in A_{i,j}, 0 \neq y \in A_{j,k}$, 
we have $0 \neq xy \in A_{i,k}$.  If $B$ is a $\mb{Z}$-graded domain, 
the corresponding principal $\mb{Z}$-algebra $\wh{B}$ is a quasi-domain.

Consider  the  Sklyanin algebra $S = S(E, \sigma)$ with its factor  TCR $B = B(E, \mc{L}, \sigma)=S/gS$  for an invertible  sheaf $\mc{L}$ 
of degree $3$ and the central element $g\in S_3$, as described in \cite[(4.2)]{blowdownI}. Then  $S$ has a graded presentation
 $S\cong \kk \{ x_1, x_2, x_3 \}/(r_1, r_2, r_3)$, where $\deg x_i=1$ and $\deg r_j=2$ for each $i,j$.  So
  $B$ has a graded presentation  $B \cong \kk \{ x_1, x_2, x_3 \}/(r_1, r_2, r_3, r_4)$, where $r_4$ corresponds to the central element $g$ and hence  $\deg r_4 = 3$.    We now give a $\mb{Z}$-algebra version of this observation; more specifically,  we study $\mb{Z}$-algebras satisfying the following properties.

  \begin{assumption}\label{ass:Zalg1}  
Let $A$ be a connected lower triangular $\mb{Z}$-algebra which is a quasi-domain.  Assume there is an ideal of $I$ which is generated  by   elements $\{ 0 \neq g_n \in A_{n+3, n}  \, | \,  n \in \mb{Z} \}$, such that 
there is a degree~$0$ isomorphism 
$\rho: A/I \cong \wh{B}$, where $B = B(E, \mc{L}, \sigma)$ for some elliptic curve $E$, invertible 
sheaf $\mc{L}$ of degree $3$, and  translation automorphism  $\sigma\in  \Aut(E) $.  Let 
$\phi$ be the degree $1$ automorphism of $A/I$ corresponding under $\rho$ to the canonical degree $1$ automorphism of $\wh{B}$.  Finally, we can  identify $(A/I)_{n, n-1} = A_{n, n-1}$ for all $n$ and we assume that 
$g_n x = \phi^3(x) g_{n-1}$ for all $x \in A_{n, n-1}$ under this identification. 
\end{assumption}

\begin{proposition}
\label{prop:Zalg}
Let $A$ be a $\mb{Z}$-algebra satisfying Assumption~\ref{ass:Zalg1}.
Then there is a graded ring $S$ and a degree $0$ isomorphism $\gamma:\wh{S} \to A$, such that either 
\begin{enumerate}
\item[(i)] $S = S(E, \sigma)$ is a Sklyanin algebra, or else
\item[(ii)] $S = B[z]$, where $z$ is a degree $3$ indeterminate.
\end{enumerate}
Moreover, the principal automorphism $\phi$ of $A/I$ lifts to a principal  automorphism $\wt{\phi}$ of $A$, which corresponds under $\gamma$ to 
the canonical principal automorphism of $\wh{S}$. 
\end{proposition}

\begin{remark} Although we have a blanket assumption that $\rqgr R$ is smooth for an elliptic algebras $R$, we note for use elsewhere that  Proposition~\ref{prop:Zalg} and its proof hold without any smoothness assumptions.  It is also worth noting that
$\rqgr B[z]$ is not smooth  (see the final step in the proof of Theorem~\ref{thm:Sposs}.)  
\end{remark}

\begin{proof}
First,  since $\hilb B = (1+s+s^2)(1-s)^{-2}$, 
 we have 
$\dim_\kk (A/I)_{n+m, n} = 3m$ for $m \geq 1$, while $\dim_\kk (A/I)_{n,n} = 1$.  
Because of the relations 
$g_n x = \phi^3(x) g_{n-1}$, the ideal $I$ is generated by the $\{ g_n \}$ as a right ideal of $A$.  Then since $A$ is a quasi-domain, we easily get 
$\dim_\kk A_{n+m,n} = \binom{m+2}{2}$ for all $m \geq 0$. 

Let $F = \kk \langle V \rangle$ be the free algebra on a $3$-dimensional $\kk$-space $V$, and fix 
surjections 
\[
F \to S(E, \sigma) = F/(r_1,r_2, r_3) \to B = B(E, \mc{L}, \sigma) = F/(r_1, r_2, r_3, r_4)
\]
as above, where $r_1, r_2, r_3$ are quadratic relations and $r_4$ is cubic.
Then we get induced surjective maps of $\mb{Z}$-algebras $\psi: \wh{F} \to \wh{S(E, \sigma)}$ 
and $\theta: \wh{S(E, \sigma)} \to \wh{B}$.
Let $A'$ be the sub-$\mb{Z}$-algebra of $A$ generated by $\{A_{n+1, n}  \, | \,  n \in \mb{Z} \}$, and let
 $\mu: A' \to A/I$ be given by the 
natural inclusion $A' \to A$ followed 
by the natural surjection $A \to A/I$. 
 The 
isomorphism  $\rho: A/I \to \wh{B}$ gives an isomorphism 
\[
 A'_{n+1, n} = A_{n+1, n} = (A/I)_{n+1, n} \overset{\rho}{\too} \wh{B}_{n+1, n} = \wh{F}_{n+1, n}
\]
 for each $n$.   Then there is a clearly a unique surjection of $\mb{Z}$-algebras $\nu: \wh{F} \to A'$ 
 with $\rho \mu \nu = \theta \psi$.

We claim that $\nu$ factors through $\wh{S(E, \sigma)}$.  Since $I$ is generated in 
degrees $(n+3, n)$, we have isomorphisms  
\[
A'_{n+2, n} = A_{n+2, n} = (A/I)_{n+2, n} \overset{\rho}{\too} \wh{B}_{n+2, n}
\]
for each $n$.  (The first equality follows since $B$ is generated in degree $1$, which forces 
$A_{n+2, n+1} A_{n+1, n} = A_{n+2, n}$
since $A$ agrees with $\wh{B}$ in low degree.)  Thus $\rho \mu$ is an isomorphism
 in degree $(n+2, n)$.  It follows that $\nu$ must 
kill the same degree $(n+2, n)$ relations that $\theta \psi$ does.  This shows that there
 is a map $\chi: \wh{S(E, \sigma)} \to A'$ 
such that $\nu = \chi \psi$; in other words, $\nu$ factors  as claimed.  Now 
$\theta \psi  = \rho \mu \nu = \rho \mu \chi \psi$ and since $\psi$ is surjective,
$\theta = \rho \mu \chi$.

 To distinguish it from the elements $g_n \in A$,  let  $h \in S(E, \sigma)_3$ be the image of the degree 3 relation $r_4\in F$  
   and let $h_n = h \in \wh{S(E, \sigma)}_{n+3, n}$.  Since $h$ is   central in $S(E, \sigma)$,   we have 
\begin{equation}
\label{eq:central}
 \wh{S(E, \sigma)}_{n+4, n+3} h_n = h_{n+1} \wh{S(E, \sigma)}_{n+1,n}, \quad \text{for all}\ n \in \mb{Z}.
\end{equation}
Since the image of $\chi(h_n)$ in $A/I$ is $0$, we have $\chi(h_n) \in I$. As $I$ is generated by the elements $\{ g_n \}$, this  forces 
   $\chi(h_n) = \lambda_n g_n$ for  some scalars $\lambda_n \in \kk$.
Using that $A$ is a quasi-domain, applying $\chi$ to \eqref{eq:central}
shows that $\lambda_n = 0 $ $\iff$ $\lambda_{n+1} = 0$.
By induction, if some $\lambda_i = 0$ then $\lambda_n = 0$ for all $n$.

Consider the case that $\lambda_n \neq 0$ for all $n \in \mb{Z}$.  In this case we have $g_n \in A'$ for all 
$n$.  Since $A/I$ is generated by $\{ (A/I)_{n+1, n}  \, | \,  n \in \mb{Z} \}$,  and $I$ is generated by the 
$g_n$, clearly $A$ is generated as an algebra by $\{ A_{n+1, n}  \, | \,  n \in \mb{Z} \}$ and  
$\{g_n  \, | \,  n \in \mb{Z} \}$.  Thus $A = A'$ in this case and $\chi$ is a surjection $\wh{S(E, \sigma)} \to A$.  
As  $\dim_\kk A_{n+m,n} = \binom{m+2}{2}  = \dim_\kk S(E, \sigma)_m = \dim_\kk \wh{S(E, \sigma)}_{n+m,n}$ for all $n \geq \mb{Z}$ 
 and $m \geq 0$, $\chi$ is a degree 0 isomorphism from $\wh{S(E,  \sigma)}$ to $A$.  So (i) holds with $\gamma = \chi$.  There is a unique 
principal automorphism $\wt{\phi}$ of $A$ corresponding under $\gamma$ to the canonical degree $1$ automorphism of $\wh{S(E, \sigma)}$.  The identity 
$\theta = \rho \mu \chi$ implies that $\wt{\phi}$ lifts the given 
automorphism $\phi$ of $A/I$.

Otherwise, $\lambda_n = 0$ for all $n \in \mb{Z}$.  In this case 
$\chi(h_n) = 0$ for all $n$, and so $\chi$ factors through $\wh{B}$ 
to give a map $\overline{\chi}: \wh{B} \to A'$ such that $\overline{\chi} \theta = \chi$.  Then 
$\theta = \rho \mu \chi = \rho \mu \overline{\chi} \theta$ and since $\theta$ is surjective, $\rho \mu \overline{\chi} = 1_{\wh{B}}$. 
 In particular, the surjection $\overline{\chi}:\wh{B} \to A'$ must be an isomorphism, and then so is 
$\mu: A' \to A/I$.   

Let $z$ be a central indeterminate of degree 3 and let $z_n = z \in \wh{B[z]}_{n+3,n}$.
We claim that we can extend the isomorphism $(\rho \mu)^{-1}:  \wh B \to A'$ to a homomorphism $\gamma:  \wh{B[z]} \to A$ by defining $\gamma(z_n) = g_n$.
For each $n \in \mb{Z}$, define 
a map $V \to \wh{B}_{n+1, n}$ by the identification $V = \wh{F}_{n+1, n} \overset{\theta \psi}{\too} \wh{B}_{n+1, n}$.  
Write the image of $v \in V$ 
under this map as $v_n$. 
To show that $\gamma$ is well-defined, we need to check that $\gamma(z_{n+1} v_{n}) = g_{n+1} (\rho \mu)^{-1}(v_n)$ is equal to $\gamma(v_{n+3} z_n) = (\rho \mu)^{-1}(v_{n+3}) g_n$ for all $n \in \ZZ$ and $v \in V$.
But this follows from the equation $g_{n+1} x = \phi^3 (x) g_n$, since $\rho$ intertwines $\phi$ with the canonical principal automorphism of $\wh{B}$.

Similarly as in the first case, $A$ is generated by $A'$ and the $g_n$, so that $\gamma$ is surjective, and comparing Hilbert series, we see that $\gamma$ is an isomorphism.  Thus case (ii) holds.  Again, there is a unique 
principal automorphism $\wt{\phi}$ of $A$ corresponding under $\gamma$ to the canonical degree $1$ automorphism of $\wh{B[z]}$.  If $\pi:\wh{B[z]} \to \wh{B}$ is the canonical surjection, by the definition of $\gamma$ we have 
$\pi = \rho \mu \gamma$.  From this equation it follows similarly as in the first case that $\wt{\phi}$  lifts $\phi$. 
\end{proof}

In order to apply the proposition above, we need to be able to recognise when a $\mb{Z}$-algebra
 is $1$-periodic and isomorphic to $\wh{B}$ for a TCR $B$.  The next lemma will help us to do this.  
It is useful to consider $\mb{Z}$-algebras of the following form.

\begin{definition}
\label{def:stand}
Let $K = \kk(E)$ for an elliptic curve,  
 let $\tau\in  \Aut_{\kk}(K)$, and consider the skew-Laurent 
ring $Q = K[t, t^{-1}; \tau]$.  We call a $\mb{Z}$-algebra $A$ \emph{standard} if (i) it is connected,
 lower triangular, and generated by $\{ A_{n+1, n}  \, | \,  n \in \mb{Z} \}$; (ii) each piece $A_{i,j}$ has 
 the form $V_{i,j} t^{d_{i,j}} \subseteq Q$ for some finite-dimensional $\kk$-subspace 
 $V_{i,j} \subseteq K$ and $d_{i,j} \in \mb{Z}$; and (iii) for each $i \geq j \geq k$ the multiplication 
 map $A_{i,j} \otimes A_{j,k} \to A_{i,k}$ is given by the multiplication in $Q$.
\end{definition}
\noindent 
It is easy to see that once the ring $Q = K[t, t^{-1}; \tau]$ is fixed, given arbitrary choices of 
$V_n t^{d_n}$ with $V_n \subseteq K$ and $d_n \in \mb{Z}$ for each $n \in \mb{Z}$, there
 is a unique standard $\mb{Z}$-algebra $A$ with $A_{n+1, n} = V_n t^{d_n}$.

We show now that every standard $\mb{Z}$-algebra is isomorphic to one where each graded piece 
is contained in $K$ and no automorphism is involved in the multiplication.

\begin{lemma}
\label{lem:noaut}  Fix $K$ and $\tau$ as in Definition~\ref{def:stand}.
Let $A$ be a standard $\mb{Z}$-algebra associated to this data.  
Then there is a standard algebra $\wt{A}$ with  
$\wt{A}_{n+1,n} \subseteq K$ for all $n$ and a degree $0$ isomorphism 
$A \cong \wt{A}$.
\end{lemma}
\begin{proof}
Write $A_{n+1, n} = V_n t^{d_n}$ where $V_n \subseteq K$.
Let $e_1 = 0$,
 for each $n \geq 2$ let 
$e_n = -\sum_{i = 1}^{n-1} d_i$, and for $n \leq 0$ let $e_n =  \sum_{i = n}^{0} d_i$.  
Let $\wt{A}$ be the unique standard $\mb{Z}$-algebra with 
$\wt{A}_{n+1,n} = \tau^{e_{n+1}}(V_n) \subseteq K$ for each $n$.
Write $A_{i,j} = V_{i,j} t^{d_{i,j}}$ and $\wt{A}_{i,j} = W_{i,j}$ for each $i,j$ with $i \geq j$, for some 
$V_{i,j}, W_{i,j} \subseteq K$.   We have $d_{i,i} = 0$ because $A$ is connected.

We claim that the map 
$\phi: A \to \wt{A}$, given on the graded piece $A_{i,j} = V_{i,j} t^{d_{i,j}} \to \wt{A}_{i,j} = W_{i,j}$ by 
the formula $v t^{d_{i,j}} \mapsto \tau^{e_i}(v)$, is an isomorphism of $\mb{Z}$-algebras.   
It is not obvious 
that $\phi$ actually has image in $\wt{A}$, but it is at least a function from $A$ to the 
$\mb{Z}$-algebra which has every 
piece equal to $K$.  Thus we first check $\phi$ is a homomorphism and then it will be clear that it lands in $\wt{A}$.

Given $u \in V_{i,j}$ and $v \in V_{j,k}$, we have $x = ut^{d_{i,j}} \in A_{i,j}$ and 
$y = v t^{d_{j,k}} \in A_{j,k}$.   Then $\phi(x) \phi(y) = \tau^{e_i}(u) \tau^{e_j}(v)$ while 
\[
\phi(xy) = \phi( u \tau^{d_{i,j}}(v) t^{d_{i,j} + d_{j,k}}) = \tau^{e_i}( u \tau^{d_{i,j}}(v)) =
 \tau^{e_i}(u) \tau^{e_i + d_{i,j}}(v).
\]
Thus to see that $\phi$ is a homomorphism, we need $e_i + d_{i,j} = e_j$ for all $i \geq j$.  Note that
 $d_{n+1, n} = d_n$ by definition 
and that $d_{i,j} = d_{i-1}+ \dots + d_j$ for all $i > j$ since $A_{i,j} = A_{i, i-1} \dots A_{j+1, j}$.
By the definition of the $e_n$ we also have $d_{i-1}+ \dots + d_j = -e_i + e_j$.  Thus 
$\phi$ is a homomorphism, so  $\phi(A_{n+1, n}) = \wt{A}_{n+1, n}$ by definition. 
 Since $A$ is generated as an algebra by 
the $\{ A_{n+1, n} \}$ and $\wt{A}$ is generated by the $\wt{A}_{n+1, n}$, the homomorphism $\phi$ does have image in $\wt{A}$ and is even  surjective.  Then, 
 $\phi$
 is an isomorphism since it is injective 
on each graded piece  by definition.
\end{proof}

We are now ready for the proof of Theorem~\ref{thm:Sposs}, for which we need the following notation.

\begin{notation}\label{notation:1periodic}
Let $T$ be an elliptic algebra of degree $9$  with $T/Tg \cong B(E, \mc{N}, \tau)$ that
  satisfies  conditions (i) and (ii) of Notation~\ref{notation:recT}.  
Thus,  we have   right $T$-modules  
$N^{(1)}= H \subseteq N^{(2)} = T \subseteq N^{(3)} = M$. Set
\[\mb{E} = \End_T(N^{(1)} \oplus N^{(2)} \oplus N^{(3)} ) = \bigg( \Hom_T(N^{(j)}, N^{(i)})  \bigg)_{ij}.\]

Define a $\mb{Z}$-algebra $\mb{S}$ by $\mb{S}_{3m+i, 3n+j} = (\mb{E}_{i,j})_{m-n}$, 
for all $m, n \in \mb{Z}$ and $i,j \in \{1,2,3\}$, where the multiplication  is induced 
from that  in $\mb{E}$.
Let $\varphi:  \mb{S} \to \mb{S} $ be the degree 3 automorphism given by the identifications $\mb{S}_{3m+i, 3n+j} = (\mb{E}_{i,j})_{m-n} = \mb{S}_{3m+3+i,3n+3+j}$.
\end{notation}

\begin{proof}[Proof of Theorem~\ref{thm:Sposs}]
We assume Notation~\ref{notation:1periodic}.
Assume in addition that the points $a, b, c, d, e, f$ in   Notation~\ref{notation:recT}(i,ii)
satisfy \eqref{need2} and \eqref{need1}. 

The main step is to show that $\mb{S}$ satisfies the hypotheses of Proposition~\ref{prop:Zalg}.  
For any $n \in \mb{Z}$ we have $\mb{S}_{n+3, n} = \End_T(N^{(i)}, N^{(i)})_1$ (for some $i$), 
which contains the element $g$.  Thus we set 
$g_n = g \in \mb{S}_{n+3, n}$ for each $n$.   Note that the multiplication in $\mb{S}$ is induced by the multiplication in $T_{(g)}$; 
in particular, $\mb{S}$ is a quasi-domain.   Let $I$ be the ideal of $\mb{S}$ generated by 
$\{ g_n  \, | \,  n \in \mb{Z} \}$.  Since all of the $N^{(i)}$ are   $g$-divisible, so is each 
$\Hom_T(N^{(j)}, N^{(i)})$  by  \cite[Lemma~4.4]{blowdownI}.
  Thus given   $i, j \in \{1,2,3\}$ and $m,n \in \mb{Z}$ with  $m-n \geq 3$ we have 
\[
I_{3m+i, 3n+j} \supseteq \mb{S}_{3m+i, 3(n+1)+j} g_{3n+j} = 
\Hom_T(N^{(j)}, N^{(i)})_{m-n-1}g = \Hom_T(N^{(j)}, N^{(i)})_{m-n} \cap gT_{(g)}.
\]  
Conversely, since each $g_n$ is a copy of $g$ and the multiplication of $\mb{S}$ is induced by 
multiplication in $T_{(g)}$, 
clearly $I_{3m+i, 3n + j} \subseteq gT_{(g)}$.  Thus $I_{3m+i, 3n+j} = \Hom_T(N^{(j)}, N^{(i)})_{m-n} \cap gT_{(g)}$. 
We see that $\overline{\mb{S}} : = \mb{S}/I$ is 
the $\mb{Z}$-algebra with $\overline{\mb{S}}_{3m+i, 3n+j} =  (\overline{\mb{E}_{i,j}})_{m-n}$ 
for $i,j \in \{1, 2, 3 \}$   where, by Proposition~\ref{prop:recT},  
\[
\overline{\mb{E}_{i,j}} = \overline{ \Hom_T (N^{(j)}, N^{(i)}) } = \Hom_B(\overline{N^{(j)}}, \overline{N^{(i)}}),
\] 
    Clearly the multiplication in
 $\overline{\mb{S}}$ is induced by the multiplication in $\kk(E)[t, t^{-1}; \tau] = T_{(g)}/gT_{(g)}$.  

As in the proof of Proposition~\ref{prop:recT},
for each $i$ we can write 
$\overline{N^{(i)}} = \bigoplus_{n \geq 0} H^0(E, \mc{G}_i \otimes \mc{N}_n)t^n$, where 
$\mc{G}_1 = \mc{O}_E(-a-b-c)$, $\mc{G}_2 = \mc{O}_E$, and $\mc{G}_3 = \mc{O}_E(d + e + f)$, 
and then
\[
\Hom_B(\overline{N^{(j)}}, \overline{N^{(i)}}) = \bigoplus_{n \geq 0} H^0(E, (\mc{G}_j^{\tau^n})^{-1} \mc{G}_i \otimes \mc{N}_n)t^n \subseteq \kk(E)[t, t^{-1}; \tau]
\]
for all $i, j \in \{1,2,3 \}$.  
In particular,  
\[\overline{\mb{S}}_{3n+i, 3n+ i-1} = \begin{cases}
\,\, H^0(E, \mc{N}(-\tau^{-1}(d)-\tau^{-1}(e) - \tau^{-1}(f) -a-b-c)) t,  &   i = 1 \\ 
\,\, H^0(E, \mc{O}_E(a+b+c)),  &  i = 2 \\ 
\,\, H^0(E, \mc{O}_E(d + e + f)),  &  i = 3.
\end{cases}
\]
It is now easy to see that $\overline{\mb{S}}$ is generated as an algebra by 
$\{ \overline{\mb{S}}_{n+1, n}  \, | \,  n \in \mb{Z} \}$, 
using  \cite[Lemma~2.2]{blowdownI}.  We have now checked that $\overline{\mb{S}}$
 is a standard algebra in the sense 
of Definition~\ref{def:stand}.

We would like to show that now that $\overline{\mb{S}}$ is isomorphic to $\wh{C}$ 
where $C = B(E, \mc{L}, \sigma)$, with  
$\mc{L}$ some sheaf of degree $3$ and $\sigma\in  \Aut(E)$  such that $\sigma^3 = \tau$.
Let $\mc{L}$ be any sheaf of degree $3$ and $\sigma$ any automorphism for the moment, and write 
$C = B(E, \mc{L}, \sigma) = \bigoplus_{n \geq 0} H^0(E, \mc{L}_n) u^n \subseteq \kk(E)[u, u^{-1}; \sigma]$.
We use Lemma~\ref{lem:noaut} to change $\overline{\mb{S}}$ and $\wh{C}$ to isomorphic 
standard algebras involving only 
multiplication in $\kk(E)$.  Writing $\overline{\mb{S}}_{n+1, n} = V_n t^{d_n}$ with 
$V_n \subseteq \kk(E)$, we have $d_n = 1$ when $n$ is a multiple of $3$, $d_n = 0$ 
otherwise.   Thus by the proof of Lemma~\ref{lem:noaut} there is a degree $0$ isomorphism 
$\overline{\mb{S}} \cong D$, where $D$ is the standard algebra with 
$D_{n+1, n} = \tau^{- \lfloor n/3 \rfloor}(V_n)  = H^0(E, \mc{D}_n)$ where the sheaf 
$\mc{D}_n$ is given by    
\begin{align*}
\mc{D}_n &= [\mc{N}(-\tau^{-1}(d)-\tau^{-1}(e) - \tau^{-1}(f) -a-b-c)]^{\tau^{- (n/3)}}   \hskip -40pt
 &    n \equiv 0 \mod 3, \\
\mc{D}_n &= [\mc{O}_E(a+b+c)]^{\tau^{- (n-1)/3}}  & n \equiv 1 \mod 3, \\ 
\mc{D}_n &= [\mc{O}_E(d+e+f)]^{\tau^{- (n-2)/3}}  &   n \equiv 2 \mod 3.
\end{align*}
Similarly, writing $\wh{C}_{n+1, n} = W_n u$, where $W_n = H^0(E, \mc{L})$ for all $n$, 
we get a degree $0$ isomorphism $\wh{C} \to F$ for the standard algebra $F$ with 
$F_{n+1, n} = \sigma^{-n}(W_n) = H^0(E, \mc{F}_n)$, where $\mc{F}_n = \mc{L}^{\sigma^{-n}}$.

Now it is easy to see that if for all $n \in \mb{Z}$ there is an isomorphism of sheaves 
$\mc{F}_n \cong \mc{D}_n$, then 
there will be a degree $0$ isomorphism $F \cong D$.  Let $\mc{L} = \mc{D}_0 = \mc{N}(-\tau^{-1}(d)-\tau^{-1}(e) - \tau^{-1}(f) -a-b-c)$.  
 We now show that we can choose 
 $\sigma$ so that $\mc{F}_n = \mc{L}^{\sigma^{-n}}\cong \mc{D}_n$ for all $n \in \ZZ$.
  
To have $\mc{F}_n \cong \mc{D}_n$ for $n = 1, 2$, we need that $\mc{L}^{\sigma^{-1}} \cong \mc{D}_1 = \mc{O}_E(a + b + c)$, and $\mc{L}^{\sigma^{-2}} 
\cong \mc{D}_2 = \mc{O}_E(d + e + f)$.  Fix a group operation $\oplus$ on $E$ and let $t \in E$ be such that 
$\tau:E\to E$ is given by $x \mapsto x \oplus t$.  (Such $t$ exists because $\tau$ is infinite order.)
 Let $\sum:  \Pic E \to E$ be the natural map, so $\sum \sO_E(a_1+\dots+a_n) = 
 a_1 \oplus \dots \oplus a_n$.  
Two invertible sheaves $\sM, \sM'$ on $E$ are isomorphic if and only if $\deg \sM = \deg \sM'$ and $\sum \sM = \sum \sM'$.

Equations \eqref{need2} and \eqref{need1} are equivalent to
  \beq\label{need2prime} \sum \sN \oplus 3t = 3a \oplus 3b \oplus 3c \eeq
  and
  \beq\label{need1prime} a \oplus b \oplus c \oplus t = d \oplus e \oplus f. \eeq
Since the group structure on an elliptic curve is divisible,  we can choose $s \in E$ 
such that $3s= t$; then the translation automorphism $\sigma(x) = x \oplus s$ satisfies 
$\sigma^3 = \tau$.  
From \eqref{need1prime} and \eqref{need2prime} we obtain 
\[a \oplus b \oplus c \ominus t = \sum \sN \ominus d \ominus e \ominus f \ominus a \ominus b \ominus c \oplus 3t = \sum \mc{L},\]
and thus $\mc{L}^{\sigma^{-1}} \cong \sO(a+b+c ) = \mc{D}_1$.
Then by \eqref{need1prime}, 
\[
\mc{L}^{\sigma^{-2}} \cong \mc{D}_1^{\sigma^{-1}} \cong \sO_E(\sigma a + \sigma b + \sigma c) \cong \sO_E(d+ e+f)\cong\mc{D}_2,\]
as needed.

Fix these choices of $\sigma$ and $\mc{L}$ and the isomorphisms $\mc{D}_r \cong \mc{F}_r$
 for $r = 0, 1, 2$ given above.
Then for each $n \in \mb{Z}$, writing $n = 3q + r$ with $r  \in \{ 0,1, 2 \}$ 
we get an induced isomorphism $\mc{D}_n = \mc{D}_r^{\tau^{-q}} \cong \mc{F}_r^{\sigma^{-3q}} 
\cong \mc{F}_n$.
We conclude that for the choice of $\sigma$ and $\mc{L}$ as above, there are degree $0$ 
isomorphisms of $\mb{Z}$-algebras 
$\mb{S}/I = \overline{\mb{S}} \cong D \cong F \cong \wh{C}$.  Let $\phi$ be the degree 
$1$ automorphism of $\overline{\mb{S}}$ 
that corresponds under this chain of isomorphisms to the canonical principal automorphism of $\wh{C}$.

In order to apply Proposition~\ref{prop:Zalg}, the last hypothesis that remains to be checked is that 
identifying $\overline{\mb{S}}_{n+1, n} = \mb{S}_{n+1, n}$ we 
have $g_{n+1} x = \phi^3(x) g_{n}$ for all $x \in \mb{S}_{n+1, n}$.  Since by construction 
$\mb{S}_{n+1, n}$ and $\mb{S}_{n+4, n+3}$ 
are naturally identified, and $g_n$ is equal to the copy of the central element $g$ in $S_{n+3, n}$ 
for each $n$, the needed equation 
amounts to showing that $\phi^3: \mb{S}_{n+1, n} \to \mb{S}_{n+4, n+3}$ is simply the natural
 identification for each $n$; in other words, showing that $\phi^3 = \varphi$ on degree $(n+1, n)$ elements. 
In other words, we must show for all $n$ that the diagram
\[ \xymatrix{
\mb{S}_{n+1, n} \ar[r] \ar@{=}[d] & \overline{\mb{S}}_{n+1,n} \ar[r] \ar@{=}[d] & \wh{C}_{n+1,n} \ar@{=}[d] \\
\mb{S}_{n+4, n+1} \ar[r] & \overline{\mb{S}}_{n+4,n+1} \ar[r]  & \wh{C}_{n+4,n+1} 
}\]
commutes, where the horizontal arrows are the maps constructed above.

Let $n = 3q+ r$, where $0 \leq r \leq 2$.
Let $\alpha_n:  D_{n+1,n} \to F_{n+1,n} $ be such that the map 
$\overline{\mb{S}}_{n+1,n} \to \wh{C}_{n+1,n}$ is given by the composition
\[ \overline{\mb{S}}_{n+1,n} = V_n t^{d_n} \stackrel{\tau^{-q}}{\too} D_{n+1,n} \stackrel{\alpha_n}{\too} F_{n+1,n} \stackrel{\sigma^n}{\too} \wh{C}_{n+1,n}.\]
Now by construction $\alpha_n$ is induced by pullback from our choice of $\alpha_r$: in other words, $\alpha_n = \tau^{-q} \alpha_r \tau^q$.  
Thus the composition $\overline{\mb{S}}_{n+1,n} \too \wh{C}_{n+1,n}$ is $\sigma^r \alpha_r$ and depends only on the residue of $n$ mod 3, as we need. 
 
We may now apply Proposition~\ref{prop:Zalg} to $\mb{S}$. We get immediately from that proposition that there is a degree $0$ isomorphism 
$\gamma: \wh{S} \to \mb{S}$, where $S$ is either the Sklyanin algebra $S(E, \sigma)$  or else $B[z]$.  In addition, the canonical principal automorphism of $\wh{S}$ corresponds under this isomorphism to a degree 1 automorphism $\wt{\phi}$ of $\mb{S}$, which lifts the automorphism $\phi$ of $\mb{S}/I$.  Since $\phi^3 = \varphi$ on elements in $\mb{S}_{n+1,n}$, we see that $(\wt{\phi})^3 = \varphi$ as automorphisms of $\mb{S}$.

By construction there is a $3$-Veronese $\mb{S}^{(3)}$, where 
$\mb{S}^{(3)}_{m,n} = \mb{S}_{3m+2,3n+2} = \Hom_T(T, T)_{m-n} = T_{m-n}$.  By the definition of the multiplication in $\mb{S}$ we have an identification of $\mb{Z}$-algebras $\mb{S}^{(3)} = \wh{T}$.  Moreover
 the degree $3$ automorphism $\varphi$ of $\mb{S}$ induces a principal automorphism $\varphi$ of $\mb{S}^{(3)}$ which is just the canonical principal automorphism of $\wh{T}$ under this identification.  
It follows from \cite[Proposition~3.1]{S-equiv} that there is an isomorphism of $\mb{Z}$-graded algebras $S^{(3)} \cong T$. 

To conclude the proof we just need to rule out the case $S = B[z]$.  In this case $T = B^{(3)}[z]$, where $z$ now has degree $1$.
By the smoothness assumption, $T^\circ$ has finite global dimension.  On the other hand, $T^{\circ} = T[z^{-1}]_0 \cong B(E, \sN, \tau)$, which has infinite global  dimension (use the proof of \cite[Theorem~5.4]{R-Sklyanin}).  Thus this case does not occur.
\end{proof}

\begin{remark}\label{rem:deg9}
We remark that we do not know of a degree 9 elliptic algebra $T$  (with $\rqgr T$ smooth) which is not isomorphic to an algebra 
 $S$ as in  Theorem~\ref{thm:Sposs}.  
However, we do not see how to prove that \eqref{need2} and \eqref{need1} always hold.
\end{remark}

%%%%%%%%%%%%%%%%%%%
  %%%%%%%%%%%%%%%%%%%
 
\section{Iterating blowing down}\label{ITERATE}

Our ultimate goal is to give a recognition theorem for 2-point blowups of a noncommutative $\PP^2$.
In order to do this, we need to study the process of iterating blowing down, and we do that in this section. 
 Throughout the section, we continue to have a standing assumption that 
   $\rqgr R$ is smooth for any elliptic algebra $R$.

Our first result is more general:  we study how to ``blow down'' a reflexive right ideal of an elliptic algebra.  
For motivation, suppose that $T = S^{(3)}$ is the $3$-Veronese of a Sklyanin algebra 
$S = S(E, \sigma)$.  Let $p \neq q \in E$, and consider how we would construct the
 right $T$-module $H$ in Proposition~\ref{prop:recT} from the blowup $R = T(p+q)$.
Besides the line modules that come from blowing up $p$ and $q$, there is a third $R$-line 
module constructed as follows.  Let $x \in S_1$ be the line that vanishes at $p,q$.  
 Then $x S_2 \subseteq R_1$ and one may compute that $R/xS_2R$ is a line module.  If we
  let $J = xS_2 R$, then $J_1T =H$.  The next lemma generalises this process.

\begin{lemma}
\label{lem:bbd-ideal}
Let $R$ be an elliptic algebra  
  with $R/gR = B = B(E, \mc{M}, \tau)$, and fix a line module $L = R/J$ satisfying $(L \dotms L) = -1$.  
Let $K \subseteq R$ be a $g$-divisible MCM right ideal, generated in degree $1$ as an $R$-module, with 
$\dim_\kk K_1 \geq 2$.

Let $\wt{R}$ be the blowdown of $R$ at $L$, as constructed by  \cite[Theorem~1.4]{blowdownI}.   As in \cite[(8.1)]{blowdownI}, let 
\[ 
\wt{K}  \ = \ \sum_{\alpha}\Bigl\{  N_{\alpha} :   K \subseteq N_{\alpha}\subset Q_{\gr} (R) 
\ \text{ with }\   N_{\alpha}/K \cong L[-i_\alpha] \ \text{ for some } i_\alpha\in \ZZ\Bigr\}.
\]
Then:
\begin{enumerate}
\item $\hilb \Ext^1_R(L,K) = s^i/(1-s)^2$, for  $i \in\{0,1,2\}$;  
thus $\hilb \wt{K} = \hilb K + s^i/(1-s)^3$. Also,   $i = 0$ if and only if $K = J$.
In any case, $\wt{K} = \Hom_R(J, K) R $. 
\item $\wt{K}$ is  a $g$-divisible MCM right $\wt{R}$-module.
\item If $i = 1$ or $i = 2$ then $\wt{K}$ is also generated in degree $1$ as a right $\wt{R}$-module.
In particular, if $i = 2$ then $\wt{K} = K_1 \wt{R}$.
If $i=0$ then $\wt{K} = \wt{R}$.
\end{enumerate}
\end{lemma}

\begin{proof}
We first note that since $K$ is MCM it is reflexive, and so
$\wt{K} = \Hom_R(J, K) R $  holds by  \cite[Lemma~8.2]{blowdownI}.
Also, by Proposition~\ref{prop:MCMchar} $\overline{K}$ is a saturated $B$-module.

(1)     As vector spaces, $\Ext^1_R(L,K) \cong \Hom_R(J, K)/K$.  Note that 
  $\overline{J} = \bigoplus_{n \geq 0} H^0(E, \mc{M}_n(-p))$, for $p=\Div L$. Similarly,  since $\overline{K}$ is saturated,  
$\overline{K} = \bigoplus_{n \geq 0} H^0(E, \mc{M}_n(-D))$ for some effective divisor $D$.  
If $D = 0$, then $\overline{K}_0 \neq 0$, a contradiction.
The $g$-divisible proper right ideal $K$ cannot contain $g$, and so the hypothesis 
$\dim_{\kk} K_1 \geq 2$ implies that $\dim_{\kk} \overline{K}_1 \geq 2$ as well.
In conclusion, we have $0 < \deg D \leq \deg \mc{M} - 2$.

Now \cite[Lemma~2.3]{blowdownI} gives 
\[
\Hom_B(\overline{J}, \overline{K}) = \bigoplus_{n \geq 0} H^0(E, \mc{M}_n(-D + \tau^{-n}(p))).
\]
Thus
 $
\hilb \Hom_B(\overline{J}, \overline{K})/\overline{K} = s^j/(1-s),
$
where either $j = 0$ if $D = p$ or $j = 1$ otherwise.
By \cite[Lemma~4.4(1)]{blowdownI},  $\Hom_R(J, K)$ is $g$-divisible since $J$ and $K$ are, and so 
\[
\hilb [\overline{\Hom_R(J,K)}/\overline{K}]/(1-s) = \hilb \Hom_R(J,K)/K = \hilb \Ext^1_R(L, K).
\]  

We have $\overline{K} \subseteq \overline{\Hom_R(J,K)} \subseteq \Hom_B(\overline{J}, \overline{K})$.
Since $J$ is a line ideal, it is $g$-divisible and reflexive by \cite[Lemma~5.6(2)]{blowdownI}.  
Thus \cite[Lemma~4.8]{blowdownI}  applies and  shows that $\overline{\Hom_R(J,K)}$ is equal in 
large degree to $\Hom_B(\overline{J}, \overline{K})$.
Since $\hilb \End_R(J) = \hilb R$ by \cite[Theorem~7.1]{blowdownI},  
 \[\overline{\End_R(J)} = \End_B(\overline{J}) = B(E, \mc{M}(-p + \tau^{-1}(p)), \tau)\]
  is a full TCR, and $\overline{\Hom_R(J,K)}$ is a right 
module over this ring.   Using $\deg D \leq \deg \mc{M} - 2$ as mentioned above, $\mc{M}(-D + \tau^{-1}(p))$ has degree at least $3$.   By  \cite[Lemma~2.3]{blowdownI},  
for any $i \geq 1$ we have
$\Hom_B(\overline{J}, \overline{K})_i \End_B(\overline{J}) = \Hom_B(\overline{J}, \overline{K})_{\geq i}$.  This equation 
also clearly holds for $i = 0$ in case $D = p$ and $\overline{J} = \overline{K}$. 
Now if for some $i \geq j$ we have 
$\overline{\Hom_R(J,K)}_i = \Hom_B(\overline{J}, \overline{K})_i$, then we will have 
$\overline{\Hom_R(J,K)}_{\geq i} = \Hom_B(\overline{J}, \overline{K})_{\geq i}$ by multiplying 
on the right  by $\overline{\End_R(J)}$.  We conclude that the actual value of 
$\hilb \overline{\Hom_R(J,K)}/\overline{K}$ must be $s^i/(1-s)$ for some $i \geq j$.
This implies that $\overline{\Hom_R(J,K)} = \overline{K} + \Hom_B(\overline{J}, \overline{K})_{\geq i}$.   By $g$-divisibility we have $\hilb \Hom_R(J,K)/K = \hilb \Ext^1_R(L,K) = s^i/(1-s)^2$.

Now the case $i = 0$ can occur only if $j = 0$, in which case $D = p$ and $\overline{J} = \overline{K}$, and
 also $\Hom_R(J,K)_0 \neq 0$.  But $\Hom_R(J,K) \subseteq \Hom_R(J,R)$, and we know by 
\cite[Lemma~5.6(3)]{blowdownI} that $\Hom_R(J,R)_0 = \kk$.  If $\kk J \subseteq K$ then $J \subseteq K$.  
 Since $J$ and $K$ are $g$-divisible and $\overline{J} = \overline{K}$ this forces $J = K$.

Assume   now   that $J \neq K$ and so $i \geq 1$. 
Now  $\wt{K}$ is a right $\wt{R}$-module by  \cite[Corollary 8.5]{blowdownI}, so $\overline{\wt{K}}$ must be a right 
module over $\overline{\wt{R}} = B(E, \mc{M}(\tau^{-1}(p)), \tau)$.  If $i > 2$, then 
\[    \begin{array}{rl}
\overline{\wt{K}}_1 \overline{\wt{R}}_1   & =
 H^0(E, \mc{M}_1(-D)) H^0(E, \mc{M}_1^{\tau}(\tau^{-2}(p)))   \\ \noalign{\vskip 3pt}
&  \quad =\  H^0(E, \mc{M}_2(-D + \tau^{-2}(p))) 
\ \nsubseteq\  H^0(E, \mc{M}_2(-D)) \ = \overline{K}_2  = \overline{\wt{K}_2},   \end{array}
\]
a contradiction.
Thus $i = 1$ or $i = 2$.  Clearly the case $i = 1$ occurs if $\Hom_R(J,K)_1 \supsetneqq K_1$, while $i = 2$ occurs if $\Hom_R(J,K)_1 = K_1$.

The Hilbert series of $\wt{K}$ follows directly from $\hilb \Ext^1_R(L,K) = s^i/(1-s)^2$ and \cite[Lemma~8.2]{blowdownI}.

(2) As noted above,  $\wt{K}$ is a right $\wt{R}$-module.  Write $\overline{\wt{R}} = B' = B(E, \mc{N}, \tau)$, where 
$\mc{N} = \mc{M}(\tau^{-1}(p))$.   

If $i = 0$, then  $K = J$ and  so $\wt{K} = \Hom(J, J) R = \wt{R}$, which is certainly $g$-divisible and MCM. 
If $i \geq 1$, then 
\[
\overline{\wt{K}} = \overline{K} + \Hom_B(\overline{J}, \overline{K})_{\geq i}\overline{R} = 
\overline{K} + \Hom_B(\overline{J}, \overline{K})_i \End_B(\overline{J}) \overline{R} = \overline{K} + \Hom_B(\overline{J}, \overline{K})_i \overline{\wt{R}}.
\]

If $i = 1$ then $\overline{\wt{K}}_i = H^0\bigl(E, \mc{M}(-D + \tau^{-1}(p))\bigr) = H^0(E, \mc{N}(-D))$, 
while if $i = 2$, then we have $\overline{\wt{K}}_i = H^0(E, \mc{M}_2(-D + \tau^{-2}(p))) = 
H^0(E, \mc{N}_2(-D - \tau^{-1}(p)))$.  Using \cite[Lemma~2.3]{blowdownI}, in either case we get 
\begin{equation}
\label{eq:down}
\overline{\wt{K}} = \overline{K} + \bigoplus_{n \geq i} H^0(E, \mc{M}_ n(-D + \tau^{-i}(p) + \dots + \tau^{-n}(p))).  
\end{equation}
It follows from   \eqref{eq:down}  that $\hilb \overline{\wt{K}} = 
\hilb \overline{K} + s^i/(1-s)^2$.  
On the other hand, we showed $\hilb \wt{K} = \hilb K + s^i/(1-s)^3$ in part (1).  This forces 
$\hilb \wt{K} = \hilb( \overline{\wt{K}})/(1-s)$ 
and thus $\wt{K}$ is $g$-divisible.

Considering the formula for $\overline{\wt{K}}_i $ from above \eqref{eq:down}, either $i =1$ and 
$\overline{\wt{K}} = \bigoplus_{n \geq 0} H^0(E, \mc{N}_n(-D))$, or else 
$i = 2$ and $\overline{\wt{K}} = \bigoplus_{n \geq 0} H^0(E, \mc{N}_n(-D-\tau^{-1}(p)))$.  In either 
case,  $\overline{\wt{K}}$ is saturated as a $B'$-module.  Thus $\wt{K}$ is 
MCM by Proposition~\ref{prop:MCMchar}.

(3) When $i =1$ or $ 2$,  then  $\overline{\wt{K}}$ is generated in degree $1$ as a right $\overline{\wt{R}}$-module, 
by the calculations in part (2) and \cite[Lemma~2.3]{blowdownI}.  Since $\wt{K}$ is $g$-divisible, $\wt{K}$ is therefore generated in degree
 $1$ as a $\wt{R}$-module by the graded Nakayama lemma.  When $i = 2$, $\wt{K}_1 = K_1$ by 
 construction and so $\wt{K} = K_1 \wt{R}$. 
\end{proof}

By combining Lemma~\ref{lem:bbd-ideal} with  results from \cite{blowdownI}, we get the following useful characterizations of the intersection 
numbers of two distinct lines.

\begin{lemma}\label{lem1-examples} 
Let $R$ be an elliptic algebra with $\deg R \geq 3$.  Let $L \not \cong L'$ be line modules,  with line ideals $J$ and $J'$, respectively. Assume that  $(L \dotms L) = -1$.   
Then
\begin{align*} 
(L \dotms L') = 1 \ & \iff \ \dim \Hom_R(J, J')_1 = \dim R_1 -2 \ \iff \  \Hom_R(J, J')_1 = J'_1, \quad \text{while} \\
 (L \dotms L') = 0 \ & \iff \ \dim \Hom_R(J, J')_1 = \dim R_1 -1 \ \iff \  \Hom_R(J, J')_1 \neq J'_1.
\end{align*}
\end{lemma}

\begin{proof}  By \cite[Lemma 5.6]{blowdownI}, line ideals are $g$-divisible, MCM, and generated in degree 1, and so we can apply 
  Lemma~\ref{lem:bbd-ideal} with $K = J'$.  Since $L \not \cong L'$, we have $J' \neq J$ and so, in the notation of Lemma~\ref{lem:bbd-ideal}, 
$i = 1$ or $i = 2$.
Examining the proof of that lemma, $\hilb \Hom_R(J, J')/J' = s^i/(1-s)^2$.   Thus if $i = 1$ we have 
$\hilb \Hom_R(J,J') = \hilb R - 1/(1-s)$, and so $\dim \Hom_R(J, J')_1  = \dim R_1 - 1 \neq \dim J'_1$. If $i = 2$ then  
$\hilb \Hom_R(J, J') = \hilb R - (1 + s)/(1-s)$. Thus  $\dim \Hom_R(J,J')_1 = \dim R_1 - 2 = \dim (J')_1$ and so 
$\Hom_R(J,J')_1 = (J')_1$.  This gives the second equivalence on each line.

Since $\rqgr R$ is smooth,  $J^{\circ}$ and $(J')^{\circ}$ are projective by  \cite[Remark 7.2]{blowdownI}.  
Thus by \cite[Theorem~7.6]{blowdownI}, we have $(L \dotms L') = 1$ if and only if $\hilb R - \hilb \Hom_R(J, J')  = (1+s)/(1-s)$.  
Since $(L \dotms L') \in \{0,1\}$ in any case by \cite[Lemma 7.4]{blowdownI}, this gives the first equivalence on each line.
 \end{proof}

The following technical lemma will also help us in our analysis of intersection numbers of lines.
\begin{lemma}
\label{lem:new} 
Let $R$ be an elliptic algebra with line modules $L = R/J$, $L'= R/J'$, with $L \not \cong L'$.
Let $K$ be a reflexive graded right $R$-submodule of $Q_{\gr}(R)$.
Suppose that $\Hom_R(J,K)_1 \subseteq \Hom_R(J', K)_1$.  Then $\Hom_R(J,K)_1 = K_1$.
\end{lemma}
\begin{proof}
Since $L \not \cong L'$, we have $J \neq J'$.
Given $x \in \Hom_R(J, K)_1$,  then 
 $x(J +J') \subseteq K$.  Let $I: = J + J'$.
Since $R/J$ is a line module,  it is $2$-critical and so $J \neq J'$ forces $\GKdim R/I \leq 1$.  
Then $xI \subseteq K$ implies that $\GKdim (xR + K)/K \leq 1$, 
but since $K$ is reflexive, this forces $x \in K$ by \cite[Lemma~4.5]{blowdownI}.  Consequently we conclude that $\Hom_R(J, K)_1 \subseteq K_1$. The reverse inclusion is immediate.
\end{proof}

We now study the process of blowing down two lines in succession.
Recall that  $U \equiv V$ if $U, V$ are graded vector spaces with $\hilb U = \hilb V$.
Similarly, $U\equiv \hilb U$.

\begin{lemma}
\label{lem:iterate}
Fix an elliptic algebra $R$ with $R/gR = B(E, \mc{M}, \tau)$.  Let $L_p = R/J_p$ and $L_q = R/J_q$ be line modules for $R$ with $(L_p \dotms L_p) = (L_q \dotms L_q) = -1$, where $\Div L_p = p$ and $\Div L_q = q$.  Suppose further that $(L_p \dotms L_q) = 0$.
\begin{enumerate}
\item
Let $\wt{R}$ be the  blowdown of 
$R$ along the line $L_p$.   Let $\wt{J_q}$ be the blowdown to $\wt{R}$ of the right ideal $J_q$.  
Then $\wt{L}_q = \wt{R}/\wt{J_q}$ is a line module over $\wt{R}$ with $\Div \wt{L}_q = q$.  Moreover, $(\wt{L}_q \dotms \wt{L}_q) = -1$; in particular, we can blow down $\wt{L}_q$ starting from $\wt{R}$ to obtain a ring $T$.  
 
\item We have $(L_q \dotms L_p) = 0$   and so part (1) also applies with the roles of $p$ and $q$ 
reversed, leading to a ring $T'$.  Then $T' = T$; thus 
  the order in 
which one blows down the two lines is irrelevant.  
\end{enumerate}
\end{lemma}
 
\begin{proof} 
(1) Note that the conditions on the intersection numbers force  $L_p \not \cong L_q$; thus $J_p \neq J_q$.   Applying Lemma~\ref{lem:bbd-ideal} with $J= J_p$ and $K =J_q$, 
as in the proof of Lemma~\ref{lem1-examples} the condition $(L_p \dotms L_q) = 0$ means we are in the case $i=1$.  Then Lemma~\ref{lem:bbd-ideal} shows that
\[ 
\hilb \wt{J}_q = \hilb J_q + s/(1-s)^3 = \hilb \wt{R} - 1/(1-s)^3 + s/(1-s)^3 = \hilb \wt{R} - 1/(1-s)^2. 
\]
Thus $\wt{L}_q = \wt{R}/\wt{J}_q$ is a line module for $\wt{R}$ as claimed.

Note that $\overline{J_q} = \bigoplus_n H^0(E, \mc{M}_n(-q))$.   
Thus  $\overline{\wt{J}_q} = \bigoplus_{n \geq 0} H^0(E, \mc{N}_n(-q))$, where 
$\overline{\wt{R}} = B' = B(E, \mc{N}, \tau)$ with $\mc{N} = \mc{M}(\tau^{-1}(p))$, 
as was calculated in the proof of Lemma~\ref{lem:bbd-ideal}(2).  This shows that $\Div \wt{L}_q= q$.

By \cite[Corollary 9.2]{blowdownI}, $\wt{R}$ is again a  elliptic algebra for which $\rqgr \wt{R}$ is smooth.  
Next we want to show that  $(\wt{L}_q \dotms \wt{L}_q) = -1$ which, by \cite[Remark~7.2 and Theorem~7.1]{blowdownI}, 
is equivalent to 
$\End_{\wt{R}}(\wt{J_q}) \equiv \wt{R}$. 

By \cite[Lemma~2.3]{blowdownI},  we have 
\[
\overline{\End_{\wt{R}}(\wt{J_q})} \subseteq \End_{B'}(\overline{\wt{J_q}}, \overline{\wt{J_q}}) =
 B(E, \mc{N}(-q + \tau^{-1}(q)), \tau).
\]
Since $\wt{J}_q$ is  $g$-divisible by Lemma~\ref{lem:bbd-ideal}, $\End_{\wt{R}}(\wt{J}_q)$
 is also $g$-divisible, and it suffices to prove that the inclusion above is an equality.  Since the TCR 
 $B(E, \mc{N}(-q + \tau^{-1}(q)), \tau)$ is generated in degree $1$, it is enough to show that 
\[
\dim_\kk \End_{\wt{R}}(\wt{J_q})_1 = \dim_\kk \wt{R}_1 = \dim_\kk R_1 + 1  = \dim_\kk \End_R(J_q)_1 + 1.
\]
Now we claim that $\End_R(J_q) \subseteq \End_{\wt{R}}(\wt{J}_q)$. Recalling that $\wt{J}_q = \Hom_R(J_p,J_q)R$, 
if $xJ_q \subseteq J_q$, then $x \Hom_R(J_p, J_q) \subseteq \Hom_R(J_p, J_q)$ since $x \Hom_R(J_p, J_q) J_p \subseteq xJ_q \subseteq J_q$.
Therefore $x \Hom_R(J_p, J_q)R \subseteq \Hom_R(J_p, J_q) R$ as well, proving the claim.   
Thus it suffices to 
show that $\End_R(J_q)_1 \neq  \End_{\wt{R}}(\wt{J}_q)_1$.  

Suppose, instead,  that   $\End_R(J_q)_1  =  \End_{\wt{R}}(\wt{J}_q)_1$.   Since 
$\wt{J}_q \subseteq \End_{\wt{R}}(\wt{J}_q)$,  certainly  
$$(\wt{J}_q)_1  = \Hom_R(J_p, J_q)_1 \subseteq \End_R(J_q)_1 = \Hom_R(J_q, J_q)_1.$$  
Since we know that $J_p \neq J_q$, Lemma~\ref{lem:new} gives $\Hom_R(J_p, J_q)_1 = (J_q)_1$.  
But then the hypothesis $(L_p \dotms L_q) = 0$ contradicts Lemma~\ref{lem1-examples}.

(2)  We have $(L_p \dotms L_q) = (L_q \dotms L_p)$ by Corollary~\ref{cor:symm}, so 
we can indeed apply part (1) with the roles of $p$ and $q$ reversed to produce a ring $T'$.
We know that $T$ and $T'$ are elliptic of degree at least 3 and so generated as algebras in degree $1$, so it suffices to prove that $T_1 = T'_1$.   The argument of part (1) showed that 
$\Hom_R(J_p, J_q)_1 \nsubseteq \End_R(J_q)_1$.  Since 
we saw that $\dim \End_{\wt R}(\wt{J}_q)_1= \dim \End_R(J_q)_1 + 1$, this implies that 
$\End_{\wt R}(\wt{J}_q)_1 = \End_R(J_q)_1 + \Hom_R(J_p, J_q)_1 = \End_R(J_q)_1 + (\wt{J_q})_1$.
 
Now by \cite[Theorem~8.3]{blowdownI}, since $T$ is the blowdown of $\wt{R}$ along the 
line $\wt{L}_q = \wt{R}/\wt{J}_q$, we have $T = \End_{\wt{R}}(\wt{J}_q, \wt{J}_q) \wt{R}$, 
so in particular  $T_1 = \End_{\wt{R}}(\wt{J}_q)_1 +  \wt{R}_1$.  Similarly, 
$\wt{R}_1 = \End_R(J_p)_1 + R_1$.
Thus
\[
T_1 = \wt{R}_1 + \End_{\wt R}(\wt{J}_q)_1  =  \wt{R}_1 + (\wt{J}_q)_1 + \End_R(J_q)_1  
= \wt{R}_1 + \End_R(J_q)_1 = 
R_1 + \End_R(J_p)_1 + \End_R(J_q)_1.
\]
A symmetric argument shows that $T'_1$ is equal to the same  vector space.
\end{proof}

We next ask when we can begin with a line ideal of an elliptic algebra $R$ and blow down
 twice to obtain a right ideal with the properties of the right ideal $H =xS_2T$ in a Sklyanin elliptic algebra.  
 Note that we do not assume that $\deg R = 7$.

\begin{lemma}
\label{lem:iterate2}
Let $R$ be an elliptic algebra with three line modules $L_p = R/J_p, L_q = R/J_q, L_r = R/J_r$, 
where $\Div L_p = p, \Div L_q = q, \Div L_r = r$ with $p \neq q$.   Assume that 
\begin{enumerate}
\renewcommand{\theenumi}{\alph{enumi}}
\item   $(L_z \dotms L_z) = -1$ for all $z \in \{p,q,r\}$;  
\item  $(L_p \dotms L_q) = 0$; and 
\item  $(L_p \dotms L_r) = (L_q \dotms L_r) = 1$. 
\end{enumerate}

As in Lemma~\ref{lem:iterate}, blow down $R$ along $L_p$ to obtain a ring $\wt{R}$, and then
 blow down $\wt{R}$ along $\wt{L}_q$ to obtain a ring $T$.  Let $K = (J_r)_1 T$.  Then $K$ 
 is a $g$-divisible MCM right ideal of $T$ with $\hilb K = \hilb T - (1+2s)/(1-s)^2$, 
 $\Div T/K = r + \tau^{-1}(p) + \tau^{-1}(q)$, and $\End_T(K) \equiv T$.
\end{lemma}

\begin{proof}     
We first show that $K = (J_r)_1 T$ is precisely the right ideal obtained by blowing down 
the right ideal $J_r$ to the ring $\wt{R}$ and then 
blowing down the resulting right ideal to $T$.  This will produce the desired Hilbert series 
for $K$ automatically, and the value of $\Div (T/K)$ will also follow immediately.   

By \cite[Remark 7.2 and  Theorem~7.6]{blowdownI},   
 Hypothesis~(c) is equivalent to 
\[
\hilb \Hom_R(J_p, J_r) = \hilb \Hom_R(J_q, J_r) = \hilb R -  (1 + s)/(1-s).
\]
We have $\Ext^1_R(L_p, J_r) = \Hom_R(J_p, J_r)/J_r$ so 
$\hilb \Ext^1_R(L_p, J_r) = s^2/(1-s)^2$.  In this case $J_r$ satisfies the hypotheses of 
Lemma~\ref{lem:bbd-ideal} with $i = 2$.  Thus if $\wt{R}$ is the blowdown 
of $R$ along $L_p$, then we may apply Lemma~\ref{lem:bbd-ideal} to obtain a right
 ideal $\wt{J_r}$ of $\wt R$, which satisfies 
$\hilb \wt{J_r} = \hilb \wt{R} - (1+s)/(1-s)^2$ and $\wt{J_r} = (J_r)_1 \wt{R}$.  Moreover,
 Lemma~\ref{lem:bbd-ideal} shows that 
$\wt{J_r}$ is again $g$-divisible, MCM,  and generated in degree $1$ 
as an $\wt{R}$-module.  
Write $\wt{L_q} = \wt{R}/\wt{J_q}$.
If we now blow down $\wt{R}$ along $\wt{L_q}$ using 
Lemma~\ref{lem:iterate}, obtaining the ring $T$, then 
Lemma~\ref{lem:bbd-ideal} again applies to blow down the right ideal $\wt{J}_r$ to a 
right ideal $K$ of $T$.  We have  $\hilb \Ext^1_{\wt R}(\wt{L_q}, \wt{J}_r) = s^i/(1-s)^2$, 
for  $i \in \{0,1,2 \}$, or equivalently  
$\hilb \Hom_{\wt{R}}(\wt{J_q}, \wt{J_r}) = \hilb \wt{R} + (s^i - s- 1)/(1-s)^2$.  Now if $i = 0$ then
$\wt{J_q} = \wt{J_r}$, which is not true since $\wt{J_r}$ is not a line ideal; so $i \in \{1, 2\}$. 

We claim that $i = 2$.    
 Suppose   that $x \in \Hom_{\wt{R}}(\wt{J_q}, \wt{J_r})_1$.  
Since $(\wt{J_q})_1 = \Hom_R(J_p, J_q)_1$ by Lemma~\ref{lem:bbd-ideal}(i), 
 $x \Hom_R(J_p, J_q)_1 \subseteq (\wt{J_r})_2 = [\Hom_R(J_p, J_r)R]_2$.  The assumption
 on the Hilbert series 
of $\Hom_R(J_p, J_r)$ implies that 
$\Hom_R(J_p, J_r)_1 = (J_r)_1.$
  Thus 
\[  [\Hom_R(J_p, J_r)R]_2 = 
 \Hom_R(J_p, J_r)_2 + (J_r)_1 R_1
 = \Hom_R(J_p, J_r)_2.  \]  
Hence  $x \Hom_R(J_p, J_q)_1 J_p \subseteq J_r.$
Let $I = \Hom_R(J_p, J_q)_1 J_p$, 
  a right $R$-module contained in $J_q$.   
Looking at the images in $\overline{R}$, 
since $\overline{\Hom_R(J_p, J_q)_1} = H^0(E, \mc{M}(-q+\tau^{-1}(p	)))$ and
 $\overline{J_p} = \bigoplus_n H^0(E, \mc{M}_n(-p))$, 
one sees that $\overline{I}_{\geq 2} = (\overline{J_q})_{\geq 2}$.   Since $J_q$ is $g$-divisible 
and $\overline{J_q}$ and $\overline{I}$ agree in large degree, 
this implies that $\GKdim J_q/I \leq 1$.  But now since $J_r$ is reflexive, by 
\cite[Lemma~4.5]{blowdownI} it follows that $x I \subseteq J_r$ implies that 
$x J_q \subseteq J_r$.  Thus $x \in \Hom_R(J_q, J_r)_1$.  This proves that
 $\Hom_{\wt R}(\wt{J_q}, \wt{J_r})_1 = \Hom_R(J_q, J_r)_1$ and  so  the
 proof of Lemma~\ref{lem:bbd-ideal}  implies that  $i = 2$ as claimed.  

Thus by Lemma~\ref{lem:bbd-ideal}, again, $K = (\wt{J}_r)_1T = (J_r)_1 \wt{R}T = (J_r)_1 T$, and 
$\hilb K = \hilb T - (1+2s)/(1-s)^2$.  Now $T/gT = B(E, \mc{N}, \tau)$ for  
$\mc{N} = \mc{M}(\tau^{-1}(p) + \tau^{-1}(q))$, 
and 
\[
\overline{K} = \overline{(J_r)_1} \overline{T} = 
H^0(E, \mc{N}(-r - \tau^{-1}(p) - \tau^{-1}(q))) \bigoplus_{n \geq 0} H^0(E, \mc{N}_n^{\tau}) 
= \bigoplus_{n \geq 1} H^0(E, \mc{N}_n(-r - \tau^{-1}(p) - \tau^{-1}(q))).
\]
Thus  $\Div T/K = r + \tau^{-1}(p) + \tau^{-1}(q)$.  The right ideal $K$ is also $g$-divisible 
and MCM by Lemma~\ref{lem:bbd-ideal}.

It remains to prove that $\End_T(K) \equiv T$.   
Since $K = (J_r)_1 T$, clearly $\Hom_T(K,K) = \Hom_R(J_r, K)$.  Consider the exact sequence 
\[
0 \to \End_R(J_r) \to \Hom_R(J_r, K) \to \Hom_R(J_r, K/J_r) \to \Ext^1_R(J_r, J_r) \to \dots.
\]
Since  $(L_r \dotms L_r) = -1$,    it follows that $\Ext^1_R(J_r, J_r) = 0$ by 
\cite[Theorem~7.1]{blowdownI}.   
  We have calculated $\overline{K}$ above, and it easily follows 
from \cite[Lemma~2.3]{blowdownI} 
that $\hilb \End_{\overline{T}}(\overline{K}) = \hilb \overline{T}$.  Since $K$ is $g$-divisible, 
$\End_T(K)$ is $g$-divisible; thus 
it suffices to prove that $\overline{\End_T(K)}  = \End_{\overline{T}}(\overline{K})$.
Since $\End_{\overline{T}}(\overline{K})$ is a TCR generated in degree $1$, it will suffice
 to show that $\dim_\kk \End_T(K)_1 \geq \dim_\kk \End_R(J_r)_1 + 2$.  Thus from the 
 exact sequence above, we need to prove that $\dim_\kk \Hom_R(J_r, K/J_r)_1 \geq 2$.

Now $K/J_r$ contains the $R$-submodule $\wt{J_r}/J_r$, where $\hilb \wt{J_r}/J_r = s^2/(1-s)^3$
by Lemma~\ref{lem:bbd-ideal}.  Then by \cite[Lemma 8.2]{blowdownI}, 
$\wt{J_r}/J_r \cong \bigoplus_{i \geq 2} L_p[-i]$ as right $R$-modules.  In particular, 
$K/J_r$ contains a submodule isomorphic to $L_p[-2]$.  Now we could have done 
all of part (1) by blowing down in the other order,  as we saw in Lemma~\ref{lem:iterate}.   In particular, if 
$R'$ is the blowdown of $R$ along $L_q$, then $(J_r)_1 R'$ is the blowdown of the right ideal 
$J_r$ to the ring $R'$.  Since $T$ contains $R'$ 
we see that $K/J_r$ also contains $[(J_r)_1 R']/J_r$ which is isomorphic to 
$\bigoplus_{i \geq 2} L_q[-i]$.  Thus $K/J_r$ also contains a submodule isomorphic to $L_q[-2]$.  

By  Hypothesis~(c)  and \cite[Theorem~7.6]{blowdownI},    $\Hom_R(J_r, L_p[-2])_1 \neq 0 \not=
\Hom_R(J_r, L_q[-2])_1$. Note that  
$0\not= \theta\in \Hom_R(J_r, L_p[-2])_1$ provides a surjection $J_r[-1] \to L_p[-2]$, since both modules are  generated in degree $2$; 
following   $ \theta$ by the
 embedding of $L_p[-2]$ in $K/J_r$
gives a nonzero element $\theta_p \in \Hom_R(J_r, K/J_r)_1$.   Similarly, we get  
$0\not= \theta_q \in \Hom_R(J_r, K/J_r)_1$ from $L_q[-2] \hra K/J_r$.   
Now the image of $\theta_p$ is isomorphic to $L_p[-2]$ and the image of $\theta_q$
 is isomorphic to $L_q[-2]$,  while 
$L_p\not\cong L_q$ by  Hypothesis~(b). Thus $\theta_p$ and $\theta_q$ are linearly independent.  Thus $\dim_\kk \Hom_R(J_r, K/J_r)_1 \geq 2$ as required.
\end{proof}

%%%%%%%%%%%%%%%%%%%%%%
%%%%%%%%%%%%%%%%%%%%%%

\section{Recognition II:  Two-point blowups of Sklyanin elliptic algebras}\label{RECOG3}

The goal of this section is to prove Theorem~\ref{ithm:recog} from the introduction, which we state here in full detail.

 \begin{theorem}\label{thm:recog}
 Let $R$ be a degree 7 elliptic algebra with $R/gR \cong B(E, \mc{M}, \tau)$ and such that  $\rqgr R$ is smooth. 
 Suppose that $R$ satisfies the following:
 \begin{enumerate}
 \item there are right line modules  $L_p=R/J_p$, $L_q=R/J_q$ and  $L_r=R/J_r$ satisfying:
 \begin{enumerate}
 \item $(L_x \dotms L_x)   = -1$ for $x\in \{p,q,r\}$;
 \item $(L_p \dotms L_q)  = 0$;
 \item $ (L_r\dotms L_y) = 1$ for $y \in \{p,q\}$;
  \end{enumerate}
\item $p = \Div L_p  \neq \Div q= L_q $.
 \end{enumerate}
 Then  $R \cong T(\tau^{-1}p + \tau^{-1}q)$, where 
 $T$ is the 3-Veronese of the Sklyanin algebra $S(E,\sigma)$ for some $\sigma \in 
\Aut E$
with  $\sigma^3 = \tau$.
\end{theorem}

We begin, however, with a weaker recognition theorem for 2-point blowups which follows quickly from the results of the previous section and from Theorem~\ref{thm:Sposs}.

\begin{proposition}\label{prop:DP7}
Let $R$ be an elliptic algebra  that satisfies the hypotheses and notation 
of  Theorem~\ref{thm:recog}, with $\Div L_r=r$, and in addition assume  that 
\beq\label{need}
\sM(\tau^{-1} p + \tau^{-1}q) \cong \sO_E(\tau^{-1} p + \tau^{-1} q + \tau^{-1}r)^{\otimes 3}.
\eeq
Then $R \cong T(\tau^{-1}p + \tau^{-1}q)$, where $T$ is the $3$-Veronese of the 
Sklyanin algebra $S(E, \sigma)$ for some $\sigma \in \Aut E$ with $\sigma^3 = \tau$.
\end{proposition}

\begin{proof}   
All the hypotheses of 
 Lemma~\ref{lem:iterate} hold, and we can apply \cite[Theorem~1.4]{blowdownI} to successively blow down $L_p$ and $L_q$ to obtain 
an elliptic algebra $T$. By \cite[Theorems~7.1 and~8.3]{blowdownI},  $R \cong T(\tau^{-1}p+\tau^{-1} q)$ is the iterated blowup of 
$T$ at $\tau^{-1}p$ and $\tau^{-1} q$; thus  $T/gT \cong B(E, \sN, \tau)$ for  $\sN = \sM(\tau^{-1}p+\tau^{-1} q)$.
We will show that $T$ is a Sklyanin elliptic algebra.
 
Lemma~\ref{lem:iterate2}   shows that if $H = (J_r)_1 T$, then $H$ is a $g$-divisible MCM right ideal of $T$ with 
$\overline{H}  = \bigoplus_{n \geq 0} H^0(E, \sN_n(-a-b-c))$,
where $a = r,$ $ b= \tau^{-1}p$ and $ c= \tau^{-1}q$.
Moreover, $\End_T(H) \equiv T$.
Thus condition (i) of Notation~\ref{notation:recT} is satisfied.

By Corollary~\ref{cor:symm}, the left line modules $L_p^{\vee}, L_q^{\vee}, L_r^{\vee}$ satisfy the same intersection theory as $L_p, L_q, L_r$; that is, they satisfy the hypotheses from Theorem~\ref{thm:recog}(i)(a,b,c).
We can therefore  use the left-hand analogue of the first paragraph of this proof to  blow down the left line 
modules $L_p^\vee$ and $L_q^\vee$.  
By  \cite[Theorem~8.3]{blowdownI}, this gives the same ring $T$.
Further, $H^\vee = T J_r^\vee$ is a left ideal of $T$ with
$\overline{H^\vee} =  \bigoplus_{n \geq 0} H^0(E, \sN_n(-\tau^{-n} d - \tau^{-n} e - \tau^{-n} f))$ for some $d,e,f$, and $\End_T(H^\vee) \equiv T$.
Thus we also have condition $(ii)' $ of Notation~\ref{notation:recT}.
By Proposition~\ref{prop:recT}(2) $M = \Hom_T(H^\vee, T)$ satisfies condition $(ii)$.

To compute $d,e,f$, note that by  \cite[Lemma~5.4]{blowdownI}, 
\[
 \overline{(J_r^\vee)_1} = H^0(E, \sM(-\tau^{-2} r )) = H^0(E, \sN(-\tau^{-2} r - \tau^{-1} p - \tau^{-1} q)).
\]
So we can take $d= \tau^{-1} r,$ $ e= p$, and $ f= q$. Thus
\[ \sO_E(\tau a+b+c)=\sO_E(\tau r + \tau^{-1}p + \tau^{-1}q) \cong \sO_E(\tau^{-1} r + p + q) = \sO_E(d+e+f)\]
and  \eqref{need1} is satisfied.  Note that \eqref{need2} follows directly from \eqref{need}.
Finally, by  \cite[Theorem~9.1]{blowdownI}, $T$ is a  elliptic algebra with $\rqgr T$ smooth.
Thus all hypotheses of Theorem~\ref{thm:Sposs} are satisfied.  By that theorem we have $T \cong S^{(3)}$, where $S = S(E, \sigma)$ is a 
Sklyanin algebra for some $\sigma$ with $\sigma^3  =\tau$.
\end{proof}

We will now  use the intersection theory developed in Section~\ref{GEOMETRY} to prove  Theorem~\ref{thm:recog}; in other words,
we will show that the final condition \eqref{need} of Proposition~\ref{prop:DP7} is superfluous.  

For the rest of the section 
assume that $R$ is a degree 7 elliptic algebra satisfying the hypotheses of Theorem~\ref{thm:recog}.  In addition, set $\X = \rqgr R$, let $B = R/Rg \cong B(E, \mc M, \tau)$ and let $L = L_r$ and $J = J_r$.  Let
$ r = \Div L$.
Fix  a group law $\oplus$ on $E$ and let $\tau$ be given by translating by $t \in E$. 
We will show under these hypotheses that \eqref{need} is automatic, or, equivalently, that 
\beq\label{needneed}
\sum \sM = 2p \oplus 2 q \oplus3 r \ominus 7t
\eeq
We will prove this by constructing two different $R$-submodules of $Q_{\gr}(R)$, which we will show are isomorphic.  Factoring out $g$ will give us \eqref{needneed}.

To understand our strategy, consider for a moment the projective surface $X$ which is the blowup of $\PP^2$ at  the points $p \neq q$.  
We use $\sim$ to denote linear equivalence of divisors on $X$.
The exceptional lines $L_p$ and $L_q$, together with the pullback $H$ of a line in $\PP^2$, form a basis for the divisor group $\Div(X)$.  
The third $(-1)$ line on $X$ is the strict transform $L$ of the line through $p$ and $q$, and we have $L \sim H - L_p -L_q$.
Further, the canonical class $K = K_X$ satisfies $K \sim -3H + L_p +L_q\sim -3L -2L_p -2L_q$.
Thus
\beq\label{dagdag}
\sO_X(2L+L_p+L_q) \cong \sO_X(-L-L_p-L_q-K).
\eeq
We want to show that  $R \cong T(\tau^{-1} p+\tau^{-1} q)$, so $R$ deforms the anticanonical coordinate ring of $X$.  
The sheaves in \eqref{dagdag} should therefore deform to give graded $R$-modules.
We will construct in Lemma~\ref{lemA} an $R$-module $Z$ that corresponds to $\sO_X(2L+L_p+L_q)$, and in Lemma~\ref{lemB} a module $Y$ corresponding to $ \sO_X(-L-L_p-L_q)$, and then show that $Z \cong Y[1]$.

We note the following easy result.

\begin{lemma}\label{lem5} 
Let  $M \subseteq R_{(g)}$ be reflexive, and let $L$ be a shifted line module.  If $0 \to M \to N \stackrel{\alpha}{\to} L \to 0$ is a nonsplit 
extension in $\rGr R$, then $N$ is Goldie torsionfree and may be regarded  as a submodule of $R_{(g)}$.
\end{lemma}

\begin{proof}
This is similar to the proof of \cite[Lemma~8.2(1)]{blowdownI}.   
Let $Z$ be the Goldie torsion submodule of $N$.  Since $M$ is Goldie torsionfree, $Z \cap M = 0$ and so $\alpha$ gives an injection $Z \hra L$.   
By definition, $N/Z$ is Goldie torsionfree and there is an exact sequence $0 \to M \to N/Z \to L/\alpha(Z) \to 0$, which is necessarily nonsplit.   
If $ Z \cong \alpha(Z)$ is non-zero, then as $L$ is 2-critical, $\GK L/\alpha(Z) \leq 1$.  But by \cite[Lemma~4.5]{blowdownI}, 
$Q_{\gr}(R)/M$ is 2-pure, a contradiction.  
Thus $Z =0$.    

Now we obtain that $M \to N$ is an essential extension in $\rGr R$; else there is a nonzero submodule $P$ of $N$ such that 
$P \cap M = 0$, so that $P$ is isomorphic to a submodule of $N/M \cong L$ and so is Goldie torsion, a contradiction.   
Thus we may regard $N$ as a submodule of the graded quotient ring $Q_{gr}(R)$ of $R$.   Since $Q_{gr}(R)/R_{(g)}$
 is $g$-torsion, we get $N \subseteq R_{(g)}$.
\end{proof}

\begin{lemma}\label{lemA}
Let $M$ be the line extension of $L$.  There are right $R$-modules $M \subset N \subset Z \subset R_{(g)}$ so that 
$N$ is MCM, $N/M \cong L_p \oplus L_q$, and $Z/N \cong L$.
\end{lemma}

\begin{proof}
We first construct $N$.  By definition of the line extension, we have an exact sequence of graded modules 
$0 \to R \to M \to L[-1] \to 0$.
Let $y \in \{p,q\}$.   By Lemma~\ref{lem8}, $(L_y\dotms R)=0$. 
Thus   by assumption,  Lemma~\ref{lem6}, 
 and  Corollary~\ref{cor:symm},  $(L_y\dotms M) = (L_y\dotms R) + (L_y \dotms L) = 1$.  
By definition, 
\[ (L_y \dotms M) = - \dim \Hom_\X(L_y, M) + \dim \Ext^1_\X(L_y, M) - \dim \Ext^2_\X( L_y, M)\]
and so $\Ext^1_\X(L_y, M) \neq 0$.
By \cite[Lemma 5.6(3)]{blowdownI}, $M$ is $g$-divisible and MCM.   
By Proposition~\ref{prop:equalhom}, $\Ext^1_{\rgr R}(L_y, M) \neq 0$.
Thus there is a nonsplit extension 
\beq \label{eqNi} 0 \to M \to N_y \to L_y \to 0 \eeq 
of graded $R$-modules, 
and by Lemma~\ref{lem5} we have $N_y \subset R_{(g)}$.  

By construction, 
$\overline{N_y} = \bigoplus_{n \geq 0} H^0(E, \sM_n( \tau^{-1} r + y)).$
Thus 
\[ \hilb \overline{N_y}/(1-s) = (2+5s)/(1-s)^3 = (1+5s+s^2)/(1-s)^3 + (1+s)/(1-s)^2 = \hilb N_y,\]
and so $N_y$ is $g$-divisible.

Certainly $M \subseteq N_p \cap N_q$; we claim that $N_p \cap N_q = M$.  We prove this by induction on  degree: 
 to start we have  $(N_p \cap N_q)_{-1} = M_{-1} = 0$.
So suppose that $M_k = (N_p \cap N_q)_k$.
As $N_p$ and $N_q$ are $g$-divisible, so is $N_p \cap N_q$:  thus
$(N_p\cap N_q)_{k+1} \cap g R_{(g)} = g(N_p\cap N_q)_k = g M_k$.
Now,  
\[\begin{aligned} \overline{(N_p \cap N_q)}_{k+1} \subseteq (\overline{N_p} \cap \overline{N_q})_{k+1} =  &
  H^0(E, \sM_{k+1}(\tau^{-1} r + p)) \cap H^0(E, \sM_{k+1}(\tau^{-1} r + q)) \\  & \qquad  = H^0(E, \sM_{k+1}(\tau^{-1} r)),
\end{aligned}\]
where the last equality is because $p \neq q$.  
Thus $\overline{(N_p\cap N_q)}_{k+1} \subseteq \overline{M}_{k+1}$ and as 
\[(N_p\cap N_q)_{k+1} \cap g R_{(g)} = g(N_p \cap N_q)_k = gM_k  = M_{k+1} \cap g R_{(g)}\]
 we have $(N_p \cap N_q)_{k+1} \subseteq M_{k+1}$.  This proves that $N_p \cap N_q = M$.

Now let $N = N_p + N_q$.  We have $\hilb N/M = \hilb N_p   + \hilb N_q - 2 \hilb M =  \hilb L_p \oplus L_q$.
There is a surjection $N_p/M \oplus N_q/M \cong L_p \oplus L_q \twoheadrightarrow N/M$.  
Comparing Hilbert series, we  get $L_p \oplus L_q \cong N/M$.
Note that 
\beq\label{Nbar} 
\overline{N} = \bigoplus_{n \geq 0} H^0(E, \sM_n(\tau^{-1} r + p + q)).
\eeq
From \eqref{Nbar} it is easy to see that $\hilb N = \hilb \overline{N}/(1-t)$, so that $N$ is $g$-divisible.  It is 
also clear from \eqref{Nbar} that $\overline{N}$ is saturated, so $N$ is MCM by Proposition~\ref{prop:MCMchar}.

As before, we have
\[ (L \dotms N) = (L \dotms R) + (L \dotms L ) + (L \dotms L_p) + (L \dotms L_q) = 1\]
and so $\Ext^1_{\X}(L,N) = \Ext^1_{\rgr R}(L, N) \neq 0$, using Proposition~\ref{prop:equalhom}. 
Thus as before there is a nonsplit extension $0 \to N \to Z \to L \to 0$, and by Lemma~\ref{lem5}, $Z \subset R_{(g)}$. 
\end{proof}

\begin{lemma}\label{lemB}
Let $J$ be the line ideal of $L$.  There is a MCM graded right ideal $Y$ of $R$ such that $Y \subseteq J$ with 
$J/Y \cong L_p[-1] \oplus L_q[-1]$.
\end{lemma}
\begin{proof}
As we noticed in the proof of Proposition~\ref{prop:DP7}, the left line modules $L_p^{\vee}, L_q^{\vee}, L^{\vee}$ have the same 
intersection theory as the line modules $L_p, L_q, L_r$.  We may therefore prove a left-handed version of Lemma~\ref{lemA}, 
to obtain MCM left modules $M^\vee $ and $N^\vee$ with
$M^\vee/R \cong L^\vee[-1]$ and $N^\vee/M^\vee \cong L_p^\vee\oplus L_q^\vee$;
in fact,  $M^\vee = J^* = \Hom_R(J, R)$.    Applying $\Hom_R(-, R)$ to the exact sequence 
$0 \to M^{\vee} \to N^{\vee} \to  L_p^\vee\oplus L_q^\vee \to 0$ gives an exact sequence 
\[ 0 \to (N^\vee)^* \to (M^\vee)^* \to (L_p \oplus L_q)[-1] \to 0,\]
where we have used  that $\Ext^1_R(N^{\vee}, R) = 0$ as $N^{\vee}$ is MCM, and  that $(L_x^{\vee})^{\vee} \cong L_x$  
 by  \cite[Lemma 5.6]{blowdownI}.

Finally, noting that $(M^\vee)^* = J^{**} = J$,  the lemma holds for  $Y = (N^\vee)^*$.
\end{proof}

\begin{proposition}\label{prop:ZY}
Let $Z,Y$ be as constructed in Propositions~\ref{lemA} and \ref{lemB}.
Then  $Z \cong Y[1]$ as right $R$-modules.
\end{proposition}
\begin{proof}
Since $\hilb R = \frac{1+5s+s^2}{(1-s)^3}$, we first note that
$$ \hilb Z = \hilb R + (3+s)/(1-s)^2 = (4+3s)/(1-s)^3,$$
and
\[ \hilb Y = \hilb R - (1+2s)/(1-s)^2 = (4s+3s^2)/(1-s)^3.\]
Thus $\hilb Z = \hilb Y[1]$.

Now using the construction of $Z$ and $Y$ and Lemma~\ref{lem6}, we compute:
\[\begin{aligned}
(Z \dotms Y[1])\  = & \ 
 (R \dotms R[1]) + ((L[-1] \oplus L \oplus L_p \oplus L_q) \dotms R[1])  \\
&\quad  - (R \dotms (L[1]\oplus  L_p \oplus L_q) )- ((L \oplus L \oplus L_p \oplus L_q )\dotms (L \oplus L_p \oplus L_q)).
 \end{aligned}\]
 Note that we have used Corollary~\ref{cor:symm}(1) to remove the shifts in the final term.

 We next want to compute $(R \dotms R[1])$. 
 By Proposition~\ref{prop:equalhom}, $\uHom_\X(R, R)  = R$ for $\X = \rqgr R$ and 
$\uExt^1_\X(R, R)  = 0$.
Thus,   \eqref{CM} gives  $\dim \Ext^2_\X(R, R[1]) =   \dim \Hom_\X(R[1], R[-1]) =0$ and hence
 $$
(R \dotms R[1]) =  -\dim \Hom_\X(R, R[1]) -\dim \Ext^2_\X(R, R[1]) = -\dim R_1 =-8,$$
as desired.  
 By Lemma~\ref{lem8}, $((L[-1] \oplus L \oplus L_p \oplus L_q) \dotms R[1]) = 5 $ and
 $(R \dotms (L[1]\oplus  L_p \oplus L_q)) = -4$.
 Finally, our intersection theory assumptions give
 \[ ((L \oplus L \oplus L_p \oplus L_q) \dotms (L \oplus L_p \oplus L_q)) = 2,\]
 so $(Z \dotms Y[1]) = -8 +5 +4-2 = -1$.
 Thus $\Hom_\X(Z, Y[1]) \oplus \Ext^2_\X(Z, Y[1]) \neq 0$.
 
By Proposition~\ref{prop2} we have $\dim \Ext^2_\X(Z, Y[1]) = \dim \Hom_\X(Y[1], Z[-1])$.  The module $Z$ is certainly
 saturated (for example, since $N$ and $Z/N \cong L$ are) and so $\Hom_\X(Y[1], Z[-1]) = \dim \Hom_R(Y[1], Z[-1])$ by Proposition~\ref{prop:equalhom}(1).   Note that for all $n \geq 0$ we have $\dim Z_{n-1} < \dim Z_{n}$; this can be seen either
  by directly computing the formula for $\dim Z_n$ or by noting that the analogous property holds both for $\hilb R$ and 
  $\hilb Z/R$.  Thus $\dim Y[1]_n = \dim Z_n  > \dim Z_{n-1} = \dim Z[-1]_n$ for all $n \geq 0$ and there are no degree 0
   injective maps from $Y[1] \to Z[-1]$.  Since $Y$ and $Z$ are both Goldie rank 1 and torsionfree, all nonzero maps must
    be injective, so $\Ext^2_\X(Z, Y[1]) = 0$.
 Thus $\Hom_\X(Z, Y[1]) = \Hom_R(Z, Y[1]) \neq 0$, where we have used Proposition~\ref{prop:equalhom} and that $Y$ is saturated again.

 As above, all nonzero elements of $\Hom_R(Z,Y[1])$ are injective, but as $\hilb Z = \hilb Y[1]$ any injective map must be an isomorphism.  Thus $Z \cong Y[1]$ in $\rgr R$. 
\end{proof}

\begin{proof}[Proof of Theorem~\ref{thm:recog}]
We must show that given the hypotheses of Theorem~\ref{thm:recog}, condition \eqref{needneed} follows.
Let $Z, Y$ be the modules constructed above.  
By construction,  
\[ Z/Zg = \bigoplus_{n \geq 0} H^0(E, \sM_n(\tau^{-1} r + r + p + q))\]
and
\[ Y[1]/Y[1]g = \bigoplus_{n \geq 0} H^0(E, \sM_{n+1}(-r - \tau^{-1} p - \tau^{-1} q)^{\tau^{-1}}).\]
By Proposition~\ref{prop:ZY}, 
$ \sM(-r - \tau^{-1} p - \tau^{-1} q)^{\tau^{-1}} \cong \sO(\tau^{-1} r + r + p + q);$
equivalently, 
\[\sum \sM \ominus r \ominus p \ominus q \oplus 6 t = 2r\oplus p \oplus q  \ominus t,\]
establishing \eqref{needneed}. 
Thus by Proposition~\ref{prop:DP7}, we obtain the result.
\end{proof}

%%%%%%%%%%%%%%%%%%%
  %%%%%%%%%%%%%%%%%%%

\section{Transforming a noncommutative quadric surface  to a noncommutative  $\PP^2$ } \label{QUADRIC}

 Here  we will prove Theorem~\ref{ithm:T} from the introduction:
 if one blows up a point on a smooth noncommutative quadric and then blows down two lines of self-intersection $(-1)$, one obtains a noncommutative  $\mb{P}^2$.  
  As was discussed in the introduction, this  is a noncommutative version of  the standard commutative result \eqref{classical}.
 In this section we will also assume that char$\, \kk=0$. This is simply because   \cite{VdB1996} and  \cite[Section~10]{SV} make that assumption; we  conjecture that our results  hold in arbitrary characteristic.

  For the reader's convenience we recall the definition of the Van den Bergh quadrics.
  
\begin{example}\label{eg:quadric}
  Let
   $\SK$ denote a 4-dimensional Sklyanin algebra; thus $\SK$ is the $\kk$-algebra
with  $4$ generators $x_0,\dots ,x_3$ and  $6$ relations
$$x_0x_i-x_ix_0=\alpha_i(x_{i+1}x_{i+2}+x_{i+2}x_{i+1}),\quad 
x_0x_i+x_ix_0=
x_{i+1}x_{i+2}-x_{i+2}x_{i+1},$$ where 
$i\in \{1,2,3\}$ mod $3$ and the $\alpha_i$ satisfy $\alpha_1\alpha_2\alpha_3
+ 
\alpha_1+\alpha_2+\alpha_3=0$ and $\{\alpha_i\}\cap\{0,\pm
1\}=\emptyset$.
The ring $\SK$  has a two-dimensional space of  central  homogeneous elements $V\subset \SK_2$. The factor  
$\SK/\SK V \cong B(E,\sA,\alpha)$, where $E$ is an elliptic curve, with a line bundle $\sA$ of degree 4 and  $\alpha$ is an automorphism.  
We  always assume that $|\alpha|=\infty$. 
 Fix an arbitrary group law $\oplus$ on $E$.   
The automorphism $\alpha$ is then translation by a point in $E$, which will be   denote   by  $a$.  
 
 Given   $0\not=\Omega\in V$, \emph{the Van den Bergh quadric}  is $\QVB=\QVB(\Omega)=\SK/\SK\Omega$.  
 Then  $\rqgr \QVB(\Omega)$ is smooth for generic $\Omega$, with a precise description of the smooth cases given by \cite[Theorem~10.2]{SV}.   We always assume below that  $\Omega$ is chosen 
so that $\rqgr \QVB(\Omega)$ is smooth.  As was discussed in Remark~\ref{non-smooth}  and is stated explicitly in Corollary~\ref{cor:birational2} below, this implies  the birationality result for arbitrary quadrics.

Let $A = \SK /\SK \Omega$ as above.  Fix a basis element $g \in A_2$ for the image of $Z(\SK)_2$ in $A_2$.  As usual,
we write $\overline{x}$ for the image of $x \in A$ under the map $A \mapsto A/gA = B(E, \mc{A}, \alpha)$.
Note that $T' := A^{(2)} $ is an elliptic algebra, called a  \emph{a quadric elliptic surface}. Moreover,  $\rqgr T' \simeq \rqgr A$ is again smooth and 
$T'/gT' \cong B(E, \mc{A}, \alpha)^{(2)} \cong B(E, \sP, \tau)$, 
where $\sP = \sA \otimes \sA^\alpha$ and $\tau = \alpha^2$.  
 \end{example}

   \begin{basicfacts}\label{basic facts}
 We recall some facts about $A=\QVB$, mostly drawn from \cite{SV}.     
 Identify $A_1=\SK_1=H^0(E, \sA)$.  Given an effective divisor $D$ on 
$E$, set $V(D) = H^0(E,\sA(-D)) \subseteq A_1 = \SK_1$.  
For any point $p \in E$, $A/V(p)A$ is a point module for $A$. These are usually the only point modules for $A$, although    for a discrete  family  of $\Omega$ there   will exist  one extra point module. These extra modules are described, for example, in \cite[Lemma~6.6]{RSS-minimal} and will   play no further r\^ole in the present discussion.
An analogous result holds on the left, and   $V(p) A_1 = \{ x \in A_2  \, | \,  \overline{x} \in H^0(E, \sA_2(-p)) \} = A_1 V(\alpha p)$
(use \cite[Remark~3.2]{blowdownI}).

Similarly, let $L$ be a line module over $\SK$.  Then there are points $p, t$ on $E$ so that $L = \SK/V(p+t)\SK$; 
 and any two points $p, t \in E$ give a line module  (see \cite[(10.3)]{SV}). 
There are two  points $z, z' \in E$ so that $p \oplus t \in \{ z,z'\}$ $\Leftrightarrow$ $\Omega \in V(p+t) \SK$ 
(this is proved for left line modules in \cite[(10.3)]{SV}, but by \cite[Remark~3.2]{blowdownI}, again, the same 
result holds for right line modules). 
Thus for fixed $\Omega$, $t$, there are two points $p, q \in E$ so that  $y \in \{p, q\}$ $\iff$ 
$\Omega \in V(y+t) \SK$.  By \cite[Theorem~10.2]{SV},  $p \neq q$ $\iff$ $z\neq z'$ $\iff$ $\rqgr A$ has finite homological dimension; thus 
this will always be the case under our assumption on $\Omega$.
If $\Omega \in V(y+t) \SK$ then $A/V(y+t)A \cong  \SK/V(y+t)\SK$ is also a line module over $A$.  
In particular,
\begin{equation}\label{equ:y-point}
\text{\it if  $y\in \{p,q\}$, then $ A/V(y+t)A$ is a line  module,}
\end{equation}
 and   we write $I_y=V(y+t)A$ for the corresponding line ideal of $A$. Using \cite[(10.3)]{SV} and \cite[Remark~3.2]{blowdownI},
 one can  even describe $q$ in terms of $p$, which we leave to the interested reader.
 \end{basicfacts}

The rest of this section is devoted to the proof of the following theorem:

\begin{theorem}\label{thm:T}  Assume that 
$\on{char}\, \kk=0$  and let $T'=Q^{(2)}$ for a Van den Bergh quadric $Q$ such that $\rqgr Q$ is smooth.
Keep the notation,    and in particular the point $t \in E$ (which determines $\{p, q\} \subset E$) 
from Basic Facts~\ref{basic facts}.

 Then there exists $\sigma\in  \Aut_{\kk}(E)$ with $\sigma^3 =\tau$ so that $T'(t) \cong T(\tau^{-1}p+\tau^{-1}q)$, where $T \cong  S^{(3)}$ is the $3$-Veronese of the Sklyanin algebra $S = S(E, \sigma)$.
 \end{theorem}

As was noted in Corollary~\ref{icor:birational2},   this theorem can also be applied to quadrics $Q$  for which $\rqgr Q$ is not smooth.
More precisely, combining the theorem with the discussion from Remark~\ref{non-smooth} gives:

\begin{corollary}\label{cor:birational2}
Any Van den Bergh quadric $Q$ defined  over a field of characteristic zero   is birational to a Sklyanin algebra.   \qed
\end{corollary}

We will prove Theorem~\ref{thm:T}   by appealing to 
 Theorem~\ref{thm:recog}, so the proof will be through a series of lemmata to show that 
$R = T'(t)$ satisfies the hypotheses of that result.
We  note that $R_1 = V(t)A_1$, since point ideals of  $\SK$ are generated in degree 1.
To match the notation of Theorem~\ref{thm:recog}, let $\sM = \sP(-t)$, so $\overline{R} = B(E, \sM, \tau)$.

We first construct the three lines on $R$.  Two of the lines are induced from the two rulings on $A$. 
 For $y \in \{p,q\}$, let $K_y = (I_y)^{(2)} = V(y+t) A_1 T'$.

\begin{lemma}\label{lem:constructlines}
Let $y \in \{p,q\}$ and let $J_y := (K_y)_1R$, for $R=T'(t)$.  Then $J_y$ is a line ideal of $R$.
\end{lemma}

\begin{proof}
The proof is similar to the proof of \cite[Theorem~5.2]{R-Sklyanin}.
We consider
$J_y = V(y+t) A_1 T'(t)$ as a subspace of $A^{(2)}$ which depends on $t$, and still write $R=T'(t)$.  
(Note that since we always have $y \oplus t \in \{z, z'\}$, therefore $y $ also varies with $t$.)
By \cite[Lemma~3.1]{R-Sklyanin} 
\beq \label{Jybar}
\overline{J_y} = \overline{V(y+t)A_1} \overline{R} = \bigoplus_{n \geq 0} H^0(E, \sM_n(-y))
\eeq
 is  a point ideal of $\overline{R}$,  
 and so 
\[ \hilb J_y \geq \hilb R - 1/(1-s)^2 = \hilb T' - s/(1-s)^3 - 1/(1-s)^2,\]
with equality if and only if $J_y$ is $g$-divisible.

Suppose now that $y \neq \tau^{-j} t$ for any $j \geq 0$, that is $2t \not  \in\{ \tau^j z, \tau^j z' \ | \ j \in \ZZ\}$. 
Let $K = K_y $   and let $J = J_y = K_1 R$.  
We claim  for all  $n\geq 0$ that, first, $K_n  \cap R_n = J_n$, and, second, this has codimension $n+1$ in $R_n$.   
Both are  trivial for $n =0,1$.    So assume it is true for $n$.  Note that $K$ is $g$-divisible (since $T'/K$ is
 $g$-torsionfree), as is $R$.  By induction, then, we have
$K \cap R_{n+1} \cap gT'  = gK \cap gR_n= g J_n$.   By induction, this is codimension $n+1$ in $gR_n$.

Working modulo $g$, we have:
\[ H^0(E, \sM_{n+1}(-y)) = 
 \bbar{K_{n+1}} \cap \bbar{R_{n+1}}  
 \supseteq \bbar{K_{n+1} \cap R_{n+1}}
\supseteq \bbar{J_{n+1}} =  
H^0(E, \sM_{n+1}(-y)).
\]
Thus all are equal, and  this vector space   clearly has codimension $1$ in $\bbar{R_{n+1}}$.  Since 
$J \subseteq  K\cap R$, the claim is proved. 

Thus for fixed $n \geq 1$, we have
$\dim (J_y)_n = \dim T'_n - \binom{n+1}{2} - (n+1)$ for a dense set of $t \in E$.
By lower semi-continuity, $\hilb J_y \leq \hilb T' - s/(1-s)^3 - 1/(1-s)^2$ for all $t \in E$.
Combined with the first paragraph, this proves the lemma.
\end{proof}

We thus obtain two line modules  $L_y := R/J_y$    for $y \in \{p,q\}$ for $R$ coming from the two rulings on $A$.   Since  \eqref{Jybar},
respectively Basic Facts~\ref{basic facts},  implies that 
$\Div L_y =y$, respectively $z \neq z'$,     \emph{Hypothesis~(2) of 
Theorem~\ref{thm:recog} holds. }

Let $r = \tau(t)$.  The third line module comes  from the blowup $R = T'(t) \subseteq T'$. By construction this produces an exceptional line module $L = L_r$ such that $(R+T'_1R)/R \cong L[-1]$ and $(R + R T'_1)/R \cong L^\vee[-1]$, where the divisor of $L_r$ is 
$r$, as the notation suggests. See  \cite[Theorems 8.3 and~8.6]{blowdownI} for the details.  Write $L_r = R/J_r$.  Then $J_r =  
(R+RT'_1)^* := \Hom_R(R + R T'_1,R)$ by \cite[Lemma 5.6]{blowdownI}. For future reference we note that \cite[Lemma 5.6]{blowdownI} also implies that  $J^* = R +  R T'_1 $. 
 The line ideal $J_r$  also has the following 
 explicit description. 
 
 \begin{lemma}\label{Jr-description}
 $J_r = V( r) V(\alpha^{-1} r) R$. 
 \end{lemma}
 
 \begin{proof}  Certainly 
\[
T'_1 V(r) V(\alpha^{-1}r) R = A_1A_1 V(r) V(\alpha^{-1}r) = V(\alpha^{-2}r)A_1 V(\alpha^{-2}r) A_1 R = V(t) A_1 V(t) A_1 R \subseteq R.
\]
Thus $V(r) V(\alpha^{-1}r) R \subseteq J_r$.  On the other hand, 
$\overline{V(r) V(\alpha^{-1}r) R} = \bigoplus_{n \geq 1} H^0(E, \mc{M}_n(-\alpha^{-2}(r)))$, which has Hilbert series 
$\hilb R - 1/(1-t)$.  Thus $\hilb V(r) V(\alpha^{-1}r) R \geq \hilb -1/(1-t)^2$.  Since $J_r$ is a line ideal, we conclude 
that $J_r = V( r) V(\alpha^{-1} r) R$ as claimed.  
\end{proof}

\begin{lemma}\label{lem:1a} 
$\rqgr R$ is smooth and $(L_r \dotms L_r) = -1$. 
\end{lemma}

\begin{proof}
Set $J = J_r $; thus   $J^* = R +  R T'_1 $ from the discussion above.  Hence  $J J^* \supseteq J_1T'_1 = V(r)A_1V( r)A_1 = T'(r)_2$.  Recall that $N^\circ = N[g^{-1}]_0$ for an 
$R$-module $N$. 
 Thus  $J^{\circ} (J^*)^{\circ} = T'(r)^{\circ}$.  Thus $J^{\circ}$ is projective by the Dual Basis Lemma.  
By \cite[Theorem~8.6]{blowdownI}, $\End_R(J) = T'(r)$ and so 
$(L_r \dotms L_r ) = -1$ follows from \cite[Theorem~7.1]{blowdownI}.  
By our choice of $\Omega$, $\rqgr A \simeq \rqgr T'$ has finite homological dimension.
As $\pdim L_{r}^\circ = 1$, $\rqgr R$ is smooth by \cite[Theorem~9.1]{blowdownI}.
\end{proof}

The next result verifies Hypothesis~(1c) of Theorem~\ref{thm:recog}.

\begin{lemma}\label{lem:c}
Let $y \in \{p, q\}$.  Then $(L_r \dotms L_y) = 1$ and $\Hom_R(J_r, J_y) \equiv \hilb R - \frac{1+s}{1-s}$.   
\end{lemma}
\begin{proof}
First, note that $(J_y)_1T' = K_y$ is $g$-divisible, while $(J_r)_1 T' \supseteq (J_r)_1 T'_1 T' = T'(r)_2 T' \supseteq g^2 T'$
(as calculated in Lemma~\ref{lem:1a}) and so $(J_r)_1 T'$ is not.  Thus  $J_r \neq J_y$ and so $L_r \not\cong L_y$.  
By Lemma~\ref{lem:bbd-ideal},   $\Ext^1_R(L_r, J_y) \equiv s^i/(1-s)^2$ 
for $i \in \{0,1,2\}$, where $i\not=0$  since $J_r \neq J_y$.

Let $\wt{J_y}$ be the right ideal of $T'$ given by applying    
Lemma~\ref{lem:bbd-ideal} to blow down $J_y$ at $ L_r$.  
By that lemma, 
\[ \wt{J_y}  = \Hom_R(J_r,J_y) R \subseteq  \Hom_R(J_r, R) R= T'.\]
(For the final equality, see \cite[Lemma 8.2]{blowdownI}.)
If $i = 1$, then  from Lemma~\ref{lem:bbd-ideal}, $T'/\wt{J_y} \equiv 1/(1-s)^2 \equiv R/J_y$.  
As $T'$ has no left line modules by \cite[Lemma~6.14]{RSS-minimal}, this is impossible.  So $i = 2$.

Now as we saw in the proof of Lemma~\ref{lem:bbd-ideal}, $i = 2$ implies that 
\[
s^2/(1-s)^2 \equiv \Ext_R^1(L_r,J_y) \equiv \hilb \Hom(J_r, J_y)/J_y= \hilb \Hom(J_r, J_y) - (\hilb R - 1/(1-s)^2)
\]
since $J_y$ is a line ideal in $R$.  Thus $\hilb \Hom(J_r, J_y) = \hilb R - (1+s)/(1-s)$.
 Since  $\rqgr R$ is  smooth and $(L_r \dotms L_r) = -1$ by Lemma~\ref{lem:1a},  it  therefore follows from \cite[Remark 7.2 and  Theorem 7.6]{blowdownI} that  $(L_r \dotms L_y) = 1$.
\end{proof}

\begin{corollary}\label{cor:pizza}
Let $y \in \{p,q\}$.
Then $\Hom_R(J_y, L_r) \equiv s^{-1}/(1-s)^2$.
\end{corollary}
\begin{proof}
By Corollary~\ref{cor:symm} and  Lemma~\ref{lem:c}, $(L_y \dotms L_r) = 1$.
Now use  Lemma~\ref{lem:1a}   combined with  \cite[Remark~7.2 and Theorem~7.6]{blowdownI}.
\end{proof}
 
Together with Lemma~\ref{lem:1a}, the next result verifies Hypothesis~(1a) of Theorem~\ref{thm:recog}.

 \begin{lemma}\label{lem:aa}
 Let $y \in \{p,q\}$.  
 Then $(L_y \dotms L_y) = -1$.
 \end{lemma}
 \begin{proof}
 By \cite[Theorem~7.1]{blowdownI} and Lemma~\ref{lem:1a}, it suffices to show that $\hilb \End_R(J_y) \equiv \hilb R$.
We saw in the proof of Lemma~\ref{lem:c} that  $\Ext^1_R(L_r, J_y) \equiv s^2/(1-s)^2$.  
 Recall that $K_y = (I_y)^{(2)}$, so that $K_y$ is generated in degree $1$ as a right ideal of $T'$.  
By Lemma~\ref{lem:constructlines}, $(J_y)_1 = (K_y)_1$ and $J_y = (J_y)_1 R$.  Then applying Lemma~\ref{lem:bbd-ideal} 
to blowdown $J_y$ along $L_r$, we obtain $\wt{J_y} = (J_y)_1 T' = K_y$, since $i = 2$.
By \cite[Lemma~8.2]{blowdownI}(2)(4), $K_y/J_y \cong \bigoplus_{i \geq 2}L_r[-i]$ and $\Ext^1_R(L_r, K_y) =0$.
 
As $K_y/J_y$ is Goldie torsion we have  the exact sequence
\beq\label{one} 0 \to \End_R( K_y) \to \Hom_R(J_y, K_y) \to  \bigoplus_{i \geq 2} \Ext^1_R( L_r[-i], K_y),  \eeq
and   the final term in \eqref{one}  is zero by the above paragraph. Thus $\Hom_R(J_y,K_y) = \End_R(K_y) = \End_{T'}(K_y)$. 
 This is the 2-Veronese of $\End_A(I_y)$, and it follows from  \cite{VdB1996} 
  (see \cite[Lemma~5.7]{RSS-minimal} for the explicit statement)  
 that $\bbar{\End_{T'}(K_y)}$ is the TCR of a degree 8 line bundle on $E$ and that 
 $\End_{T'}(K_y) \equiv T'$.  
In particular, $\dim_{\kk} \End_{T'}(K_y)_1 = 9$.

We also have the exact sequence 
\beq \label{two} 0 \to \End_R(J_y) \to \Hom_R(J_y, K_y) \to \bigoplus_{i \geq 2} \Hom_R(J_y,  L_r[-i]). \eeq
  By Corollary~\ref{cor:pizza}   
$\hilb \bigoplus_{i \geq 2} \Hom_R(J_y,  L[-i]) = s/(1-s)^3$,  and so \eqref{two} induces an exact sequence
\[ 0 \to \End_R(J_y)_1 \to \End_R(K_y)_1 \to \kk.\]
Thus $\dim_\kk \End_R(J_y)_1$ is 8 or 9.  
On the other hand, $\End_B(\bbar{J_y}) = B(E, \sM(-y+\tau^{-1}(y)), \tau)$ by \eqref{Jybar} and so 
\[ \dim_\kk \End_R(J_y)_1 \leq 1 + \dim_\kk \End_B(\bbar{J_y})_1 = 8.\]

Thus $\dim_\kk \End_R(J_y)_1  = 8$ and so $\bigl(\bbar{\End_R(J_y)}\,\bigr){}_1 = \End_B(\bbar{J_y})_1$.
Since $\End_B(\bbar{J_y})$ is a TCR, follows that  $\bbar{\End_R(J_y)} = \End_B(\bbar{J_y})$. Finally,  as $\End_R(J_y)$ is $g$-divisible
and $\hilb \overline{\End_R(J_y)} = \hilb \overline R$,  it follows that $\hilb \End_R(J_y) \equiv \hilb R$, as required.
 \end{proof}

The next result proves Hypothesis~(1b) of Theorem~\ref{thm:recog}.
\begin{lemma}\label{lem:bb}
Let $y\neq x \in \{p,q\}$.  Then $(L_y \dotms L_x) = 0$. 
\end{lemma}

\begin{proof}
By  
\cite[Lemma~7.4]{blowdownI}, $(L_y \dotms L_x) \in \{0, 1\}$.
By \cite[Theorem~7.7]{blowdownI}, $(L_y \dotms L_x ) = 0$ $ \iff$ $\hilb \Hom_R(J_y, J_x ) = h_R - \frac{1}{1-s}$.  Similarly, 
as $L_x\not\cong L_y$  \cite[Theorem~7.6]{blowdownI} implies that  $(L_y \dotms L_x) = 1$ $\iff$ $\hilb \Hom_R(J_y, J_x) = h_R - \frac{1+s}{1-s}$.
Thus   it   suffices to prove that 
$\dim_\kk \Hom_R(J_y, J_x)_1 \geq \dim R_1 - 1 = 7$.  But:
\[ V(x+s)V(\alpha^{-1} s)(K_y^*)_0 (J_y)_1 \subseteq V(x+s) V(\alpha^{-1} s) T'_1 
= V(x+s) A_1 V(s) A_1 = (J_x)_2,\]
and  so $\Hom_R(J_y, J_x)_1 \supseteq V(x+s)V(\alpha^{-1} s)(K_y^*)_0$.   By the discussion after 
\cite[(6.13)]{RSS-minimal} $\dim (K_y^*)_0=2$.    
Since $\bbar{V(x+s)V(\alpha^{-1} s)(K_y^*)_0} $ is a product of spaces of global sections, it is therefore easy to 
compute that its dimension is $2+3+2=7$.  The result follows.\end{proof}     

\begin{proof}[Proof of  Theorem~\ref{thm:T}]
We have proved that $R$ satisfies all of the hypotheses of Theorem~\ref{thm:recog},  from which   Theorem~\ref{thm:T} follows.    
\end{proof}

%%%%%%%%%%%%%
%%%%%%%%%%%%%
 
\section{A converse result}  
 In this short section we prove Proposition~\ref{iprop:converse} from the introduction, thereby giving  the converse to  Theorem~\ref{ithm:recog}. 
 
We will need  some preliminary results.  
The following result gives a method for proving  that a right ideal is a line ideal, without explicitly calculating its Hilbert series.

\begin{lemma}
\label{lem:lineideal}
Let $R$ be an elliptic algebra.  Suppose that $I \subseteq R$ is a right ideal and that $\overline{I}$ is a point ideal in $B = R/gR$.  
 If $I^*\not= R$, then $I$ is a line ideal.  In particular 
$I$ is $g$-divisible.
\end{lemma}
\begin{proof}  We first show that $I$ is $g$-divisible. So, let $J=\widetilde{I}=\{x\in R \,|\,xg^n \in I \text{ for some } n\geq 1\}$ be the 
$g$-divisible hull of $I$ and suppose that $I\subsetneqq J$. We first claim that $\overline{J}\supsetneqq \overline{I}$. Indeed,  pick a homogeneous element $a\in J\smallsetminus I$ of minimal degree. If  $\overline{J} = \overline{I}$, then $a=b+gc$, where 
$b\in I$ with $\deg b \leq \deg a$. Since $gc\in J$ and $J$ is $g$-divisible, $c\in J$. The minimality of
 $\deg a$ implies that $c\in I$ and hence $a\in I$, giving the required contradiction.
   Hence  $\overline{J}\supsetneqq \overline{I}$.  
 As $\overline{I}$ is a point ideal, $\overline{R}/\overline{I}$ is $1$-critical and so $\overline{R}/\overline{J}$ is finite dimensional. 
 As $gR\cap J=gJ$, it follows that $\GKdim R/J\leq 1$. Since $R$ is CM (see Notation~\ref{elliptic-defn}), this implies that $J^*=R$. 

As $I^*\not=R$, we may pick a homogeneous element $x\in Q_{\gr}(R)\smallsetminus R$ such that $xI\subseteq R$.  As $J$ is finitely generated, 
$g^mJ\subseteq I$ and hence $xg^mJ\subseteq R$ for some $m\geq 1$. Since $J^*=R$, we conclude that $xg^m\in R$. 
Thus  $x=yg^{-r}$ for 
some $y\in R$ and $r\geq 1$, where we may also assume that $y\not\in gR$. However, this implies that 
$yI=g^rxI\subseteq g^rR\subseteq gR$.  Since $y\not\in gR$ and $gR$ is a completely prime ideal, it follows that 
$I\subseteq gR$, contradicting the fact that $\overline{I}\not=0$.  This proves that $I$ is indeed $g$-divisible.

Finally, as $gI=gR\cap I$ and $\hilb \overline{R}/\overline{I}=(1-s)^{-1}$  it follows that $\hilb R/I = (1-s)^{-2}$, as required.
\end{proof}

 Let $S = S(E, \sigma)$ be a Sklyanin algebra, where $|\sigma| = \infty$, and write $S/(g) = B(E, \sL, \sigma)=: B$ for some degree 3 invertible sheaf $\sL$ on $E$.

\begin{definition}\label{defn:vanishing}
Let $X \subseteq S_m$ be a subspace.  We say that $X$ is \emph{defined by vanishing conditions on $E$} if $\overline{X}= H^0(E, \mc{L}_m(-p_1- \dots -p_n))$ 
for some $p_1, \dots p_n \in E$, and $X \supseteq gS_{m-3}$.
 Note that if $X$ is defined by vanishing conditions on $E$, where $n < \deg \mc{L}_m$,  
then $\dim_\kk X$ is immediately determined to be 
$\dim_\kk S_m - n = \binom{m+2}{2} - n$.

For $a, b_1, \dots, b_n \in E$ write  
\[ W(a) = H^0(E, \sL(-a)) \subset S_1 \qquad \text{and}  \qquad V(b_1+ \dots + b_n) = H^0(E, \sL_2(-b_1-\dots -b_n)) \subseteq S_2.\]   Recall from \cite[Lemma~4.1]{R-Sklyanin} that $S_1 W(a) = W(\sigma^{-1} a) S_1$, a fact that will be used without further comment.   Similarly, we will use \cite[Lemma~3.1]{R-Sklyanin} to compute products in $B$ without comment.
\end{definition}

We next  investigate  products of the spaces above and when they are defined by vanishing conditions.
We say $a,b,c \in E$ are {\em collinear} if there is $x \in S_1$ that vanishes at $a,b,c$; equivalently, if $\sL \cong \sO_E(a+b+c)$. 

\begin{lemma}\label{lem:product}
Let $a,b,c \in E$.
\begin{enumerate}
\item If $c \neq \sigma^{-2} b$ then $W(b) W(c)$ and $W(b) W(c) S_1$ are defined by vanishing conditions on $E$; 
in particular $W(b) W(c) = V(b + \sigma^{-1} c)$, $\dim W(b) W(c) = 4$, and $\dim W(b) W(c) S_1 = 8$.  
On the other hand $\dim W(b) W(\sigma^{-2}b) = 3$ and $\dim W(b) W(\sigma^{-2}b) S_1 = 7$.

\smallskip
\item $V(b+c)S_1 = S_1V(\sigma b + \sigma c)$, and this space is defined by vanishing conditions on $E$.
Moreover, $V(a + b + c)S_1$ is defined by vanishing conditions on E if and only if $a, b, c$ are not collinear, while
$S_1 V(a + b + c)$ is defined by vanishing conditions on $E$ if and only if $\sigma a, \sigma b, \sigma c$ are not collinear.
 
 \smallskip
\item $W(a) V(b+c)$ is defined by vanishing conditions on $E$ if and only if $\sigma^{-2} a \not \in \{ b,c\}$; otherwise $\dim W(a) V(b+ c) = 6$.   
Similarly $V(a+b) W(c)$ is defined by vanishing conditions on $E$ if and only if $\sigma c \not \in \{a, b \}$, else $\dim V(a+b)W(c) = 6$.
\end{enumerate}
\end{lemma}
\begin{proof}
$(1)$ The dimensions of these spaces are given in \cite[Lemmata~4.1 and~4.6]{R-Sklyanin} while, from the proof of   \cite[Lemma~4.1]{R-Sklyanin},  $W(b)W(c)$ is  defined by vanishing conditions.  
 The other claims follow easily.

$(2)$ The first sentence follows from (1) once one notes that at least one of $V(b + c) = W(b) W(\sigma c)$ or $V(b + c) = W(c)W(\sigma b)$ must hold.

For the second sentence, we prove the first claim, as the other follows symmetrically.
If $a,b,c$ are collinear then $V( a + b + c) = xS_1$ for some $x \in S_1$.  Thus 
$V(a+b+c) S_1 = xS_2$ and $\dim V(a + b + c)S_1 = 6$; in particular this space is not equal to 
$\{ x \in S_3 \, | \,  \overline{x} \in H^0(E, \sL_3(-a-b-c)) \}$ which has dimension $7$.

It therefore suffices   to prove that
$\dim V(a+b+c) S_1= 7$ if $a,b,c$ are not collinear.
Indeed, since   $\dim \overline{V(a+b+c) S_1} =  6$,   it is enough to show that $g \in 
 V(a+b+c)S_1$.
 
 Let $V = V(a+b+c)$ and let $d \in E \ssm \{ \sigma^{-1} a, \, \sigma^{-1}b,\,  \sigma^{-1} c\}$.  We claim that $g\in V W(d)$.
 To see this, write $W(d) = \kk x + \kk y$ where $wy+xz = 0$ and $\{w,x\}$ is a basis of $W(\sigma^{2} d)$.
 Then $V y \cap Vz = Y wy$ where
 \[ Y = \{ r \in S_1\, |\, r W(\sigma^{2} d) \subseteq V\}.\]
Since $\sigma d \not \in \{a,b,c\}$, clearly $Y= W(a) \cap W(b) \cap W(c) =  0$, since $a,b,c$ are not collinear.  
 Thus   $\dim V W(d)  = 6$.  As $\overline{VW(d) } = H^0(E, \sL_3(-a-b-c-\sigma^{-2}d))$, which has dimension 5, we have $ g\in VW(d) \subset VS_1$, as required.

  $(3)$ By symmetry it suffices to determine  $\dim W(a) V(b+c)$, for which we follow the proof of \cite[Lemma~4.6]{R-Sklyanin}. 
Let $\{w,x\}$ be a basis of $W(a)$ and let $\{y,z\}$ be a basis of $W(\sigma^{-2} a)$ so that $wy+xz=0$.
Then $wS \cap xS = wyS = xzS$ and so 
\[w V(b+c) \cap x V(b+c) = wy Y\qquad \text{for} \qquad 
 Y = \{  r \in S_1 \, | \,  W(\sigma^{-2} a) r \subseteq V(b+c)\}.\]
If $\sigma^{-2}a \in \{b, c\}$ then without loss of generality $b = \sigma^{-2}a$ and $Y = W(\sigma c)$.
In this case \[ \dim W(a) V(b+c) = 2 \dim V(b+c) - \dim Y = 6.\]
Otherwise, $\dim Y = 1$ and $\dim W(a) V(b+c) = 7$.
\end{proof}

 We can now prove the converse to Theorem~\ref{ithm:recog}.  
 
\begin{proposition}\label{prop:converse}
Let $T$ be a Sklyanin elliptic algebra with associated elliptic curve $E$, and let $p \neq q \in E$.  Then $R:= T(p+q)$ satisfies the hypotheses of Theorem~\ref{ithm:recog}.
\end{proposition}
 \begin{proof}
Certainly $R$ is a degree 7 elliptic algebra.   

We now change our earlier notation and write 
 $L_p$ for  the exceptional  line module obtained from writing $R$ as the blowup at $p$ of $T(q)$, with line ideal $J_p$.
This differs from the notation in the earlier sections since now, by  \cite[Lemma~9.1]{R-Sklyanin}, $\Div L_p = \tau(p)$. 

  Define $L_q=R/J_q$ analogously. 
Write $T = S^{(3)}$ where $S$ is a Sklyanin algebra with $S/(g) = B(E, \sL, \sigma)$. Let $\sigma^3 = \tau$ and write 
$T/(g) = B(E, \sM, \tau)$.  By our standing conventions on elliptic algebras, $\tau$ and hence also $\sigma$ has infinite order.

Let $x \in S_1$ generate $H^0(E, \sL(-p-q))$ and let $I_x = xS_2 R \subseteq R$. 
 It is easy 
to calculate that $\overline{I_x}$ is a point ideal in $B$, associated to the third point of $E$ where $x$ vanishes.
We claim that $I_x$ is a line ideal in $R$.  To see this, define $U = V(p+q)V(\sigma^{-1} p + \sigma^{-1} q) \subseteq S_4$ and 
note that 
\[US_2 = V(p+q)V(\sigma^{-1} p + \sigma^{-1} q)S_2 = V(p+q)S_1V(p+q)S_1 = R_2,\] by Lemma~\ref{lem:product}(2).  
Then $(Ux^{-1}) (I_x)_1 = Ux^{-1}xS_2 = US_2 \subseteq R$.  Moreover, if $Ux^{-1} \subseteq R_1$ were to hold, then 
$U \subseteq R_1 x$, which is not true by looking at the images in $S/(g)$.  Thus we can choose $y \in Ux^{-1} \setminus R$ 
such that $y I_x \subseteq R$.  By Lemma~\ref{lem:lineideal}  $I_x$ is a line ideal, as claimed. 
    Write $I_x=J_r,$ with $L_r=R/J_r$ for the corresponding line module $L_r$, where $r = \Div(L_r) \in E$.
   
By \cite[Proposition~4.5.3]{Simon-thesis}, $\gldim R^\circ < \infty$ and so by \cite[Lemma~6.8]{blowdownI}, $\rqgr R$ is smooth.

We have $(J_p)_1 = \{ z \in R_1  \, | \,  T(q)_1 z \subseteq R \}$ since $T(q)_1 R/R \cong L_p[-1]$ by \cite[Theorem~8.3]{blowdownI}. 
   Since 
\[W(q) S_2 W(\tau p) V(\sigma p + \sigma q)  = W(q)W(\sigma p) S_1 V( p +  q)S_1 \subseteq R_2,\] 
we have  $W(\tau p) V(\sigma p + \sigma q) \subseteq (J_p)_1$.
However, from Lemma~\ref{lem:product}(3) we see $\dim W(\tau p)V(\sigma p+ \sigma q) = \dim (J_p)_1$ so
\beq
\label{uselater} 
J_p = W(\tau p) V(\sigma p + \sigma q) R.
\eeq
From this and Lemma~\ref{lem:product}(2) we get \[(J_p)_1 T_1=W(\tau p) V(\sigma p + \sigma q) T_1 =W(\tau p) S_2 V(\tau p  + \tau q) S_1 \ni g^2.\]
It follows that  that $\GKdim_TT/J_pT =1$, and so the $T$-module double dual $(J_pT)^{**}=T$, by the CM condition on $T$ (see Notation~\ref{elliptic-defn}). On the other hand  $(J_rT)^{**} = (xS_2T)^{**} = xS_2T$.
  Therefore
\[ \Hom_R(J_p, J_r) \subseteq \Hom_T(J_pT, J_r T) \subseteq \Hom_T((J_pT)^{**}, (J_rT)^{**}) = \Hom_T(T, xS_2 T) = xS_2 T.\]
Thus $\Hom_R(J_p, J_r)_1 = xS_2 = (J_r)_1$ and by Lemma~\ref{lem1-examples} and 
Corollary~\ref{cor:symm}, $(L_r \dotms L_p) = 1$.  Likewise, $(L_r \dotms L_q) = 1$.

Let $X= W(\tau q) W(\sigma q)S_1$.
Then 
\[W(p) S_2 X = W(p) W(\sigma q) S_1 W(q) S_2 \subseteq R_1 W(q) S_2.\]
From the discussion before Lemma~\ref{Jr-description},   
$J_p^* = R T(q)_1 + R$ and $J_q^*= R T(p)_1 + R$.    Therefore, 
$X \subseteq \Hom_R(J_q^*, J_p^*)_1 = \Hom_R(J_p, J_q)_1.$
As $\dim X = 7$ by Lemma~\ref{lem:product}(1), it follows from  Lemma~\ref{lem1-examples} that $(L_p \dotms L_q) = 0$.

Since $\rqgr R$ is smooth,  \cite[Proposition~1.3(2)]{blowdownI}
  implies that
 $(L_p\dotms L_p) = (L_q \dotms L_q) =-1$.
It is easy to calculate that $\End_R(J_r) = x T(\sigma^{-2}p + \sigma^{-2}q)x^{-1}$, which has the same Hilbert series as $R$, and so 
$ (L_r \dotms L_r) = -1$,  by \cite[Theorem~7.1]{blowdownI}. \end{proof}

 %%%%%%%%%%%%%%%%%%%
  %%%%%%%%%%%%%%%%%%%
 
\section{The noncommutative Cremona transform}\label{CREMONA}

In this section we use Theorem~\ref{ithm:recog}   to give a noncommutative version of the classical Cremona transform.  Recall that if $X$ is the blowup of $\PP^2$ at three non-collinear points $a,b,c$ then $X$ contains a hexagon of $(-1)$ lines, given by the three exceptional lines and the strict transforms of the lines through two of $a,b,c$; further, blowing down the strict transforms of the lines gives  a birational map from $\PP^2 \dra \PP^2$.
 The next theorem is our version of this construction.

\begin{theorem}\label{thm:Cremona}
 Let $T = S^{(3)}$ be a Sklyanin elliptic algebra with   associated elliptic curve $E=E(T)$. 
 Let $p,q,r \in E$ be distinct points such that $\sigma p, \sigma q, \sigma r$ are not collinear in $\PP(S_1^*)$.  Set $R = T(p+q+r)$.
   Then $\rqgr R$ is smooth, and there is a subring $T' \cong T $ of $T_{(g)}$ such that   $R = T'(p_1, q_1, r_1)$ for    points $p_1,q_1,r_1 \in E$.   
 
 Assume  that $\sigma x, \sigma y, \sigma z$ are not collinear in $\PP(S_1^*)$, for any  $x,y,z\in\{p,q,r\},$ including possible repetitions. 
  Then  the 6 points $\{p,q,r,p_1, q_1, r_1\}$ are pairwise distinct.
 \end{theorem}

 \begin{remark}\label{rem:Cremona}  As will be apparent from the proof,   $R$ contains a hexagon of lines of self-intersection $(-1)$, as described in 
 Figure~\ref{fig33}.   Then 
 $T'$ is obtained by blowing down the 
 three lines in $R$ that   are not exceptional    for the first blowup. 
 \end{remark}
 
 The proof of this theorem will be through a series of subsidiary results that  will take the whole section.  The  strategy is to first show
  that $R$ has six line modules of self-intersection $(-1)$, each of which can be contracted.   
 After constructing these line modules over $R$ and computing their intersection theory (see Figure~\ref{fig33}), we then contract one of these lines  to give an overring $\wh{R}$ of $R$. 
 We  then  compute the intersection
  theory of $\wh{R}$ (see  Proposition~\ref{prop:disjoint}) and show that $\wh{R}$ 
satisfies the hypotheses of Theorem~\ref{ithm:recog}.  Thus, by that result,  we can then  contract two further line modules to give a ring isomorphic to $T$.   

 { \begin{figure}
\includegraphics[scale=.5]{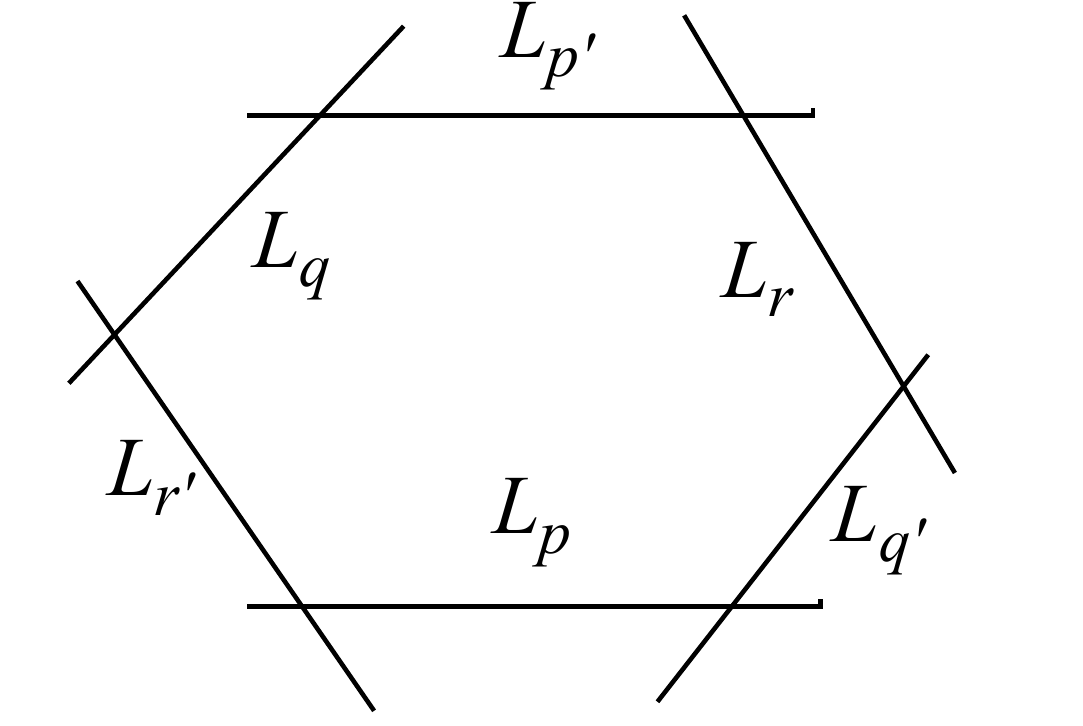}
\caption{The lines on $\rqgr R$.}
\label{fig33}
\end{figure}}

   As usual,   write $S/(g) = B(E, \sL, \sigma)$, and $T/(g) = B(E, \sM, \tau)$ where   $\sM = \sL_3 := \sL \otimes \sigma^* \sL \otimes \sigma^{2*} \sL$ and $\tau=\sigma^3$.  Fix a group law $\oplus$ on $E$ so that  three collinear points sum to zero  and define $p' := \ominus q \ominus r$, $q' := \ominus p \ominus r$, and $r' := \ominus p \ominus q$.        We continue to use the notations $W(a)$ and
   $ V(a+b)$  from  Definition~\ref{defn:vanishing}.

We first construct the six lines on $R$.
Let $L_p$ be the exceptional  line module obtained by  writing $R$ as the blowup of $T(q+r)$ at $p$,  with line ideal $J_p$.  Likewise, construct $L_q = R/J_q$ and $L_r = R/J_r$.

\begin{lemma}\label{lem:pprime}
Let $x \in S_1 $  define the line through $q,r$ and define $J_{p'} = x W(\sigma p)S_1 R$.
Then $L_{p'} := R/J_{p'}$ is a line module with divisor $p'=\Div L_{p'}$.
\end{lemma}

\begin{proof}
It is easy to check  by using \cite[Lemma~3.1]{R-Sklyanin} that $\overline{J_{p'}}$ is a point ideal in $\overline{R}$.
 Let 
\[
U = V(p+q+r)V(\sigma^{-1} q + \sigma^{-1} r) \subseteq S_4.
\]
Then using Lemma~\ref{lem:product}(2) we have 
\[
Ux^{-1} (J_{p'})_1 = V(p+q+r)V(\sigma^{-1} q + \sigma^{-1} r) W( \sigma p) S_1 =  V(p+q+r)S_1 V(q + r) W(\sigma^2 p) \subseteq R_2.
\]
If $Ux^{-1} \subseteq R$, then $U \subseteq Rx$, which is not true by considering the images in $\overline{S}$.  Thus  $J_{p'}^*\supsetneqq R$   and so  $J_{p'}$ is a line ideal by Lemma~\ref{lem:lineideal}.  Since $p'\oplus q\oplus  r=0$, clearly $x$ vanishes at these three points and so $p'=\Div L_{p'}$ as $\overline{J_{p'}}$ is the point ideal of $p'$. 
\end{proof}

Likewise, we construct $L_{q'}$, with divisor $q'$, and $L_{r'}$, with divisor $r'$.
We caution the reader that, by \cite[Lemma~9.1]{R-Sklyanin},   $\Div L_p = \tau(p)$ and similarly for $L_q$ and $ L_r$.
On the other hand, $\Div L_{p'} = p'$ and similarly for $L_{q'}$ and $L_{r'}$.

\begin{lemma}\label{lem:H}
The six  lines $L_p, L_q, L_r, L_{p'}, L_{q'}, L_{r'}$  all have self-intersection $(-1)$.
\end{lemma}

\begin{proof}
For $L_p$, $L_q$, and $L_r$ this follows from \cite[Theorem~1.4]{blowdownI}. We give the proof for  $L_{p'}$.  By \cite[Theorem~7.1]{blowdownI},  it is enough to show that $\hilb \End_R(J_{p'}) = \hilb R$, and  for this it suffices to prove that $\dim\End_R(J_{p'})_1 \geq \dim R_1$.

Using Lemma~\ref{lem:product} and the fact that $\dim T(\sigma p + \sigma^{-2} q+\sigma^{-2} r)_1=7$, we calculate:
\begin{align*}
 (x T(\sigma p &+ \sigma^{-2} q+\sigma^{-2} r)_1 x^{-1})\cdot (x W(\sigma p) S_1)  \ = \  
x T(\sigma p + \sigma^{-2} q+\sigma^{-2} r)_1 W(\sigma p) S_1 \\
&=\  x W(\sigma p) V(\sigma^{-1} q + \sigma^{-1} r) W(\sigma p)  S_1  \  \subseteq \ x W(\sigma p) T(\sigma^{-1} p + \sigma^{-1} q + \sigma^{-1} r)_1 S_1 
 \\ &\quad = \ x W(\sigma p) S_1 T(p + q + r)_1 \ = \  (x W(\sigma p) S_1) R_1.
\end{align*}
Thus $ x T(\sigma p + \sigma^{-2} q+\sigma^{-2} r)_1 x^{-1} \subseteq \End_R(J_{p'})_1$ and  $\dim \End_R(J_{p'})_1 \geq \dim T(\sigma p + \sigma^{-2} q+\sigma^{-2} r)_1 = \dim R_1$, as required. \end{proof}

Since by hypothesis $\sigma p, \sigma q, \sigma r$ are not collinear, it follows from Lemma~\ref{lem:product}(2) that
$V(\sigma p + \sigma q + \sigma r) S_1$ is defined by vanishing conditions on $E$, in the sense of Definition~\ref{defn:vanishing}.
We now give explicit generators of $J_p, J_q$, and $J_r$.  It suffices to do this for $J_p$.

\begin{lemma}\label{lem:C}
$(J_p)_1 = W(\tau p)V(\sigma p + \sigma q + \sigma r)$ and $J_p=W(\tau p)V(\sigma p + \sigma q + \sigma r)$R.
\end{lemma}

\begin{proof}
Let $C = W(\tau p)V(\sigma p + \sigma q + \sigma r)\subseteq R_1$.   By the discussion before Lemma~\ref{Jr-description},   
  $\Hom_R(J_p, R) =: J_p^* = R V(q+r) S_1 + R$.  Note that $R_1 \subseteq V(q+r)S_1$ and so $(J_p^*)_1 = V(q+r)S_1$.  
  Thus 
\[ (J_p^*)_1 C = V(q+r)W(\sigma^2 p) S_1 V(\sigma p + \sigma q + \sigma r). 
\]
As $R_1$ is defined by vanishing conditions on $E$, both
$V(q+r)W(\sigma^2 p)$ and $S_1 V(\sigma p + \sigma q + \sigma r) $ are contained in $ R_1$,
 and so $(J_p^*)_1 C \subseteq R_2$ and  $C \subseteq (J_p)_1$.
Since $\dim C \geq \dim \overline{C} = 5 = \dim (J_p)_1$, we see that $C = (J_p)_1$.
The fact that $J_p=(J_p)_1R$  follows from \cite[Lemma~5.6(2)]{blowdownI}.  
\end{proof}

\begin{lemma}\label{lem:D}
$\rqgr R$ is smooth. 
\end{lemma}
 
\begin{proof}
Using the fact that $\sigma p,\,\sigma q,\, \sigma r$ are not collinear, we compute:
\[\begin{array}{rll}
(J_p J_p^*)_2 & \supseteq  W(\tau p) V(\sigma p + \sigma q + \sigma r) V(q+r) S_1  &\quad \mbox{by Lemma~\ref{lem:C}\hfill }\\
& = W(\tau p ) V(\sigma p + \sigma q + \sigma r) S_1 V(\sigma q+\sigma r) &\quad \mbox{by Lemma~\ref{lem:product}(2)\hfill }\\
& \supseteq W(\tau p) g V(\sigma q+\sigma r) &\quad \mbox{by Lemma~\ref{lem:product}(2)\hfill } \\
& \ni g^2 &\quad \mbox{by Lemma~\ref{lem:product}(3), using that $p \not \in \{q, r\}$}.\\
\end{array}\]
Therefore, by the Dual Basis Lemma,  $J_p^\circ$ is projective. 
Since $q \neq r$,  $\rqgr T(q+r) $ is smooth by \cite[Proposition~4.5.3]{Simon-thesis}  and  so by 
 \cite[Theorem~9.1]{blowdownI} $\rqgr R$ is smooth as well. 
 \end{proof}

\begin{lemma}\label{lem:E}
 {\rm (1)}  $(L_p \dotms L_q) = (L_p \dotms L_r ) = (L_q \dotms L_r) =0$.
 
 {\rm (2)} $\Hom_R(J_p, J_q)_1 = W(\tau q) V(\sigma q + \sigma r).$  
\end{lemma}

\begin{proof}  (1)
We compute $(L_p\dotms L_q)$.
By Lemma~\ref{lem:product}(3), $\dim W(\tau q) V(\sigma q + \sigma r) = 6 = \dim R_1 -1$, so by Lemma~\ref{lem1-examples} it suffices to prove that $W(\tau q) V(\sigma q + \sigma r) \subseteq \Hom_R(J_p, J_q)_1$.  We compute:
\[\begin{array}{rll}
W(\tau q) V(\sigma q + \sigma r) (J_p)_1 
& = W(\tau q) V(\sigma q + \sigma r)W(\tau p)  V(\sigma p + \sigma q+ \sigma r) &\quad \mbox{\ by Lemma~\ref{lem:C}} \\
& \subseteq  W(\tau q)  V(\sigma p + \sigma q + \sigma r) S_1  V(\sigma p + \sigma q + \sigma r) 
  \\ & \subseteq (J_q)_1 R_1 = (J_q)_2  &\quad \mbox{\ by Lemma~\ref{lem:C}}.
\end{array}\]

(2) This follows from  the proof of part~(1), combined with  Lemma~\ref{lem1-examples}.
\end{proof}

\begin{lemma}\label{lem:F}  \begin{enumerate}
\item 
$(L_p \dotms L_{p'}) = (L_q \dotms L_{q'} )= (L_r \dotms L_{r'})  =0$.

\item  Moreover, $ \Hom_R(J_p, J_{p'})_1 = xS_2$. 

\item Similarly, $(L_{p'} \dotms L_{q'}) = (L_{p'} \dotms L_{r'}) = ( L_{q'} \dotms L_{r'}) = 0$.
\end{enumerate}
\end{lemma}
 
\begin{proof}   (1, 2) 
We compute $(L_p \dotms L_{p'})$.  Recall  
that $x\in S_1$ defines  the line through $q,r,p'$ and that $J_{p'} = x W(\sigma p) S_1 R$.
Since  $(J_p^*)_1 \subseteq T_1$,  the calculation in the proof of  Lemma~\ref{lem:D} shows that $g^2 \in J_p T_1$. 
Thus $\GKdim T/J_p T \leq 1$ and so $(J_pT)^{*} := \Hom_T(J_pT, T) =T$,  by \cite[Lemma~4.5(1)]{blowdownI}. 
Hence
\[\Hom_R(J_p, J_{p'}) \subseteq \Hom_T(J_pT, J_{p'}T) \subseteq \Hom_T((J_pT)^{**}, (J_{p'}T)^{**}) \subseteq \Hom_T(T, xS_2T) = xS_2T.\]

Since $R_1$ is defined by vanishing conditions on $E$, Lemmata~\ref{lem:pprime} and~\ref{lem:product}(2) imply that 
\[xS_2(J_p)_1 = xS_2 W(\tau p) V(\sigma p + \sigma q + \sigma r) = \bigl (xW(\sigma p) S_1 \bigr) \bigl(S_1 V(\sigma p + \sigma q + \sigma r)\bigr) \subseteq (J_{p'})_1 R_1.\]
Thus $xS_2 \subseteq \Hom_R(J_p, J_{p'})_1$ and hence $ \Hom_R(J_p, J_{p'})_1 = xS_2$. 
 By Lemma~\ref{lem1-examples}, $(L_p \dotms L_{p'}) = 0$.
 
 (3) We show that $(L_{p'} \dotms L_{q'})  = 0$.  
As in part~(1),  $J_{p'} = x W(\sigma p) S_1 R$  while   $J_{q'} = y W(\sigma q) S_1 R$, where   $y$ defines the line through $p, q', r$.
By Lemmata~\ref{lem1-examples} and~\ref{lem:product}(3) it is enough to show that 
\[y W(\sigma q)V(\sigma^{-1} q + \sigma^{-1} r) x^{-1} \subseteq \Hom_R(J_{p'}, J_{q'})_1.\]
This follows from  a familiar  computation:
\[ \bigl(y W(\sigma q)V(\sigma^{-1} q + \sigma^{-1} r) x^{-1}\bigr)\bigl(x W(\sigma p) S_1\bigr) = y W(\sigma q)S_1 \bigl(V(q+r)  W(\sigma^2 p)\bigr) \subseteq y W(\sigma q) S_1 R_1 = (J_{q'})_2,\]
as required.  
\end{proof}

The final piece of intersection theory needed is to determine  the lines that intersect with multiplicity~1.

\begin{lemma}\label{lem:J}
If $a \neq b \in \{p, q, r\}$, then $(L_{a'} \dotms L_{b})=1$.
\end{lemma}
\begin{proof}
Without loss of generality, we compute $(L_{p'} \dotms L_{q})$.
Write $\wt{R} = T(q+r)$, which is the blowdown of $R$ at $L_p$.  
By Lemmata~\ref{lem:E} and \ref{lem:F}(1), $(L_p \dotms L_{q}) = 0 = (L_p \dotms L_{p'})$.  Thus by Lemma~\ref{lem:iterate}(1), we may blow down the line ideals $J_q$ and $J_{p'}$ at $L_p$, to obtain line ideals $\wt{J_q}$ and $\wt{J_{p'}}$ in  $\wt{R}$ such that the line modules $\wt{L_q} = \wt{R}/\wt{J_q}$ and $\wt{L_{p'}} = \wt{R}/\wt{J_{p'}}$ again have self-intersection $(-1)$.

By Lemmata~\ref{lem:bbd-ideal}(3) and \ref{lem:E}(2),   
$ \wt{J_q} = \Hom_R(J_p, J_q)_1 \wt{R} =  W(\tau q) V(\sigma q + \sigma r)\wt{R}.$
Note that by \eqref{uselater}, 
 $\wt{L_q}$ is the exceptional line module that comes from writing $\wt{R} $ as the blowup of $T(r)$ at $q$.  
Likewise, by Lemmata~\ref{lem:bbd-ideal}(2) and~\ref{lem:F}(2)  
$\wt{J_{p'}} = \Hom_R(J_p, J_{p'})_1 \wt{R} = xS_2 \wt{R}.$
Thus $\wt{L_{p'}}$ is the line module denoted by $L_r$ in the proof of Proposition~\ref{prop:converse}.

Finally, Proposition~\ref{prop:converse} shows that $(\wt{L_{p'}} \dotms \wt{L_q}) = 1$ and so $(L_{p'} \dotms L_q) = 1$ by Proposition~\ref{prop:disjoint} below.
\end{proof}

We can now almost complete the proof of Theorem~\ref{thm:Cremona}, modulo proving  one final result. 

\begin{proof}[Proof of Theorem~\ref{thm:Cremona}]
 Lemmata~\ref{lem:E}, \ref{lem:F},   \ref{lem:H}, and \ref{lem:J} together establish that  $R$ has a hexagon of $(-1)$ lines with the intersection theory indicated in Figure~\ref{fig33}.  
 We will use these computations without further comment.

 Since 
 $(L_{p'} \dotms L_{p'}) = -1$ we may, by \cite[Theorem~8.3]{blowdownI}, blow down $L_{p'}$ to obtain an overring $\wh{R} $ of $R$ so that $R = \wh{R}(\tau^{-1}(p'))$.  By that result, $\wh{R}$ is a degree 7 elliptic algebra while, 
 by \cite[Theorem~9.1]{blowdownI}, $\rqgr \wh{R}$ is smooth.  
 
 Now $(L_{p'} \dotms L_{q'}) = (L_{p'} \dotms L_p)  = (L_{p'} \dotms L_{r'}) = 0$ while  $L_{q'}, L_p, L_{r'}$ have self-intersection $(-1)$.
 So,  by Lemma~\ref{lem:iterate} there are induced $\wh{R}$-line modules $\wh{L_{q'}}$, $\wh{L_p}$, and $\wh{L_{r'}}$, each of which has self-intersection $(-1)$.  
Moreover, $\Div \wh{L_{q'}} = q'$ and  $\Div \wh{L_{r'}} =r'$ while, 
by construction, $q' = \ominus p \ominus r \neq r' = \ominus p \ominus q$.
 Thus,  by  Proposition~\ref{prop:disjoint}, below, $(\wh{L_{q'} }\dotms \wh{L_{r'}}) =(L_{q'} \dotms L_{r'}) = 0$ and
 $(\wh{L_p} \dotms \wh{L_{q'}}) = (\wh{L_p} \dotms \wh{L_{r'}}) = 1$.
 
 Therefore, after appropriately  renaming the line modules,  $\wh{R}$ satisfies the hypotheses of Theorem~\ref{ithm:recog} and so there is an overring $T'$ of $\wh{R}$ such that $T' \cong T$ and $\wh{R} = T'(\tau^{-1}q' + \tau^{-1}r')$.  
In other words,    $R=T(p_1,q_1,r_1)$, for $p_1=\tau^{-1}(p')$, $q_1=\tau^{-1}(q')$  and $r_1=\tau^{-1}(r')$.

It remains to check that the 6 points $\{p,q,r,p_1, q_1, r_1\}$ are distinct, under the additional assumption that $\sigma x, \sigma y, \sigma z$ are not collinear for any $x, y, z \in \{p,q,r \}$ chosen with possible repetition.  
 This is a  routine computation, combining   the definition of 
$p',q',r'$ with the hypotheses of the theorem  and  the fact that points $x,y,z\in E$ are collinear if and only if $x\oplus y\oplus z=0$. We leave details  to the reader.
\end{proof}
    
In order to complete the proof of Theorem~\ref{thm:Cremona}, it remains  to prove the following result, which generalises Lemma~\ref{lem:iterate}(1).
The point of the result is   that 
 contracting  a line $L$ does not affect the interaction of other lines which are disjoint from $L$,  as one would hope.

\begin{proposition}\label{prop:disjoint}
 Let $R$ be an  elliptic algebra such that $\rqgr R$ is smooth and let $L, L_p, L_q$ be line modules with line ideals $J, J_p, J_q$ respectively, with $\Div L_p = p$ 
and $\Div L_q = q$.  (We allow $L_p \cong L_q$ here.)
  Assume that:
 \begin{enumerate}
 \item $(L \dotms L) = -1$;
\item  for $x \in \{p,q\}$, $(L_x \dotms L_x) = -1$;
 \item for $x \in\{p,q\}$, $(L \dotms L_x) = 0$.
 \end{enumerate}
 Let $\wt{R}$ be the blowdown of $R$ along the line $L$.
 As in Lemma~\ref{lem:iterate}(1), for $x\in \{p,q\}$ let $\wt{J_x}$ be the blowdown to $\wt{R}$ of $J_x$,  and let $\wt{L_x} = \wt{R}/\wt{J_x}$, which by Lemma~\ref{lem:iterate} is a line module over $\wt R$.
 Then
 \[ (\wt{L_p} \dotms \wt{L_q}) = (L_p \dotms L_q),\]
 where the intersection product on the left hand side is in $\rqgr \wt{R}$, and on the right hand side is in $\rqgr R$.
  \end{proposition}

  \begin{remark}\label{rem:airplane1} 
 We note that the proposition still holds without Hypothesis~(2), although the proof is more complicated and  is omitted. 
\end{remark}
  
\begin{proof}[Proof of Proposition~\ref{prop:disjoint}]
 
Write $r = \Div L$, and let $R/gR = B(E, \sM, \tau)$.  
Throughout the proof a statement involving $x$ is being asserted to hold for both $x = p$ and $x = q$.

First, if $L_p \cong L_q$, then obviously $\wt{L_p} \cong \wt{L_q}$.  In this case, since $(L_p \dotms L_p) = -1$, we have $(\wt{L_p} \dotms \wt{L_q}) = (\wt{L_p} \dotms \wt{L_p}) = -1$ by Lemma~\ref{lem:iterate}, as required.  So from now on we can and will  assume that  $L_p \not \cong L_q$ and hence $J_p \neq J_q$.
Note that  $L\not\cong L_x$ and hence $J\not=J_x$  by comparing Hypotheses~(1) and ~(3). 

\medskip
{\bf Case I: Assume that $r\not=\tau^j(q)$ for $j\geq 0$.}

The point of this assumption is that it allows us to prove:

\begin{sublemma}\label{sublem:interesction}
Keep the hypotheses of the proposition and assume that $r\not=\tau^j(q)$ for $j> 0$. 
Then $\wt{J_q}\cap R = J_q$.
\end{sublemma}

\begin{proof}  
  Since  $J\not=J_q$, we have $i\geq 1$ in Lemma~\ref{lem:bbd-ideal}.  
Hence $\wt{J_q}\not=\wt{R}$   by Part~(3) of that lemma.  Write 
$X:=\wt{J_q} \cap R \subseteq N \subseteq \wt{J_q}$, where $N/X$ 
is  some finitely generated graded $R$-submodule of $\wt{J_q}/X$ with  $\GKdim(N/X)\leq 1$.
  Then $\GKdim(N+R)/R = 1$ and so, as $R$ is reflexive, \cite[Lemma~4.5]{blowdownI} implies that 
  $N \subseteq R$ and hence  $N = X$.  Therefore, $\wt{J_q}/X$ is  $2$-pure as 
an $R$-module.  
 
By \cite[Lemma~8.2]{blowdownI}, $Z: =\wt{J_q}/J_q = \bigoplus_{i \in \mathbb{I}} L[-a_i]$, where $a_i\geq 0$ 
for all $i$.  Let $Y:=X/J_q$, which embeds in $Z$.  Now apply \cite[Proposition 8.1]{blowdownI}: since 
$Z/Y \cong \wt{J_q}/X$ is $2$-pure, there is an internal direct sum 
$Z = Y \oplus (\bigoplus_{i \in \mathbb{J}} L[-a_i])$ for some subset $\mathbb{J}$ of $\mathbb{I}$, and thus  
$Y \cong  \bigoplus_{i \in \mathbb{I} \setminus \mathbb{J}} L[-a_i]$.  
Now $Y = X/J_q \subseteq R/J_q \cong L_q$ is a submodule of a line module and is also isomorphic to 
a direct sum of shifts of line modules.  Comparing Hilbert series, there must be only one line module in the sum, 
so $Y \cong L[-c]$ for some $c \geq 0$.
 
Finally,  $ L[-c]\cong Y$ embeds in the line module $R/J_q \cong L_q$, and by \cite[Lemma 5.5]{blowdownI} this forces $r = \tau^j(q)$ for some $j \geq 0$. Necessarily, $j>0$ since otherwise   $L \cong L_q$, which  we have excluded.   So $r = \tau^j(q)$ for some $j > 0$, contradicting the hypothesis of the sublemma.
\end{proof}

 We now return to the proof of the proposition.       Suppose first that $\wt{L_p}\cong \wt{L_q}$. Then $\wt{J_p}=\wt{J_q}$ and hence, by the sublemma,  $J_p\subseteq \wt{J_p}\cap R =\wt{J_q}\cap R = J_q.$
 Since $\hilb R/J_p = 1/(1-s)^2 = \hilb R/J_q$ this forces $J_p=J_q$; a contradiction. 
 We conclude that $\wt{L_p} \not \cong \wt{L_q}$. 
Since $(L_p \dotms L_p) = -1$, we also have $(\wt{L_p} \dotms \wt{L_p}) = -1$ by Lemma~\ref{lem:iterate}.  In particular,  Lemma~\ref{lem1-examples} 
 implies that  $(L_p \dotms L_q) \in \{0,1 \}$ and $(\wt{L_p} \dotms \wt{L_q}) \in \{0,1 \}$.

By Lemma~\ref{lem:bbd-ideal}(1),   
$$\Hom_R(J_p,J_q)  (\wt{J_p}) \ = \ \Hom_R (J_p,J_q)\bigl( \Hom_R(J,J_p)R \bigr) 
\ \subseteq  \Hom_R(J, J_q)  R \ = \ \wt{J_q}.$$
Thus $\Hom_R(J_p, J_q) \subseteq \Hom_{\wt{R}}(\wt{J_p}, \wt{J_q})$.   

Suppose next that $(L_p \dotms L_q) = 0$ but $(\wt{L_p} \dotms \wt{L_q}) = 1$.  
First, recall that $\dim \wt{R}_1 = \dim R_1+1$ by \cite[Theorem~8.3]{blowdownI}.  
By Lemma~\ref{lem1-examples},  twice,  $(J_q)_1 \subsetneqq \Hom_R(J_p, J_q)_1$ with $\dim \Hom_R(J_p, J_q)_1 = \dim R_1 - 1$ while   $\Hom_{\wt{R}}(\wt{J_p}, \wt{J_q})_1 = (\wt{J_q})_1$ with  $\dim \Hom_{\wt{R}}(\wt{J_p}, \wt{J_q})_1 = \dim \wt{R}_1 - 2 = \dim R_1 - 1$.  It therefore  follows from the last paragraph  
that  $\Hom_R(J_p, J_q)_1 = \Hom_{\wt{R}}(\wt{J_p}, \wt{J_q})_1 = (\wt{J_q})_1 = \Hom_R(J, J_q)_1$.
Now since $J \neq J_p$, Lemma~\ref{lem:new} implies that $\Hom_R(J_p, J_q)_1 = (J_q)_1$.  
By  Lemma~\ref{lem1-examples} this contradicts $(L_p \dotms L_q) = 0$. 

Finally, assume that $(L_p \dotms L_q) = 1$ but $(\wt{L_p} \dotms \wt{L_q}) = 0$.   By  Sublemma~\ref{sublem:interesction},
 $\wt{J_q} \cap R = J_q$ and so 
\[
\Hom_{\wt{R}}(\wt{J_p}, \wt{J_q}) \cap R \subseteq \Hom_R(\wt{J_p} \cap R, \wt{J_q} \cap R) \subseteq \Hom_R(J_p, \wt{J_q} \cap R) = 
\Hom_R(J_p, J_q).
\]
In particular,  using Lemma~\ref{lem1-examples}, 
$(\Hom_{\wt{R}}(\wt{J_p}, \wt{J_q}) \cap R)_1 \subseteq \Hom_R(J_p, J_q)_1 = (J_q)_1$.  
The same lemma shows   
that $\dim (\Hom_{\wt{R}}(\wt{J_p}, \wt{J_q}))_1  = \dim \wt{R}_1 - 1$.  Since $\dim R_1 = \dim \wt{R}_1 -1$, 
we get 
$\dim (\Hom_{\wt{R}}(\wt{J_p}, \wt{J_q}) \cap R)_1 \geq \dim \wt{R}_1 -2$, 
while $\dim (J_q)_1 = \dim R_1 -2 = \dim \wt{R}_1 - 3$.
This is a contradiction.

Therefore, the only possibility is that  $(L_p \dotms L_q) =  (\wt{L_p} \dotms \wt{L_q})$ and Case I is complete.

\medskip
{\bf Case II: Assume that $r=\tau^j(q)$ for some  $j> 0$.}

This part of the proof will actually work  whenever  $r\not=\tau^j(q)$ for $j< 0$. This case will parallel that of Case I, except that we will pass from right to left modules. 

Write $L_r= L$ and let $y\in \{r,p,q\}$ and $x\in\{p,q\}$.  By \cite[Lemma~5.6]{blowdownI},
 $L^{\vee}_y = \Ext^1_R(L_y, R)[1]$ is a left line module, and we write $N_y=L^{\vee}_y \cong R/K_y$ for the left line ideal $K_y$.  The relationship between the left line ideal $K_y$ and right line ideal $J_y$ is most easily 
expressed in terms of the $R$-linear duals.

\begin{sublemma}\label{sublem:dualrelation} 
Let $J$ be a right line ideal and $K$ a left line ideal in $R$.  Then $R/K \cong (R/J)^{\vee}$ if and only if 
$\Hom_R(J,R)_1 = \Hom_R(K, R)_1$.
\end{sublemma}

\begin{proof}
By \cite[Lemma 5.6(3)]{blowdownI}, $R \subseteq M = \Hom_R(J,R)$ is the unique extension of $R$ up to isomorphism such that $M/R \cong (R/J)^{\vee}[1]$.   In particular, if $(R/J)^{\vee} \cong R/K$ then  $M_1 K \subseteq R$ and so $\Hom_R(J,R)_1 \subseteq \Hom_R(K, R)_1$; the other inclusion follows analogously.
The converse is similar.
\end{proof}

 By  Corollary~\ref{cor:symm},  the $L_y^{\vee}$ satisfy  the same intersection theory as the $L_y$. Similarly, since $L_y^{\vee\vee}\cong L_y$ by \cite[Lemma 5.4]{blowdownI}, we also have $L^{\vee}_x\not\cong L_r^{\vee}$.  
The left hand analogue of \cite[Lemma~8.2]{blowdownI} 
defines  left modules $\wt{K_x}$   by attaching all possible copies of $L_r^{\vee}$ on top of $K_x$. 
 Crucially, by \cite[Theorem~8.3]{blowdownI},  blowing   down  $R$   on the left along $L_r^{\vee}$   leads
to the same overring $\wt{R}$ as   blowing   down  $R$   along $L$.   In particular, 
$\wt{K_x}$ is a left line ideal  of $\wt{R}$ and, by the left-right analogue of Lemma~\ref{lem:iterate},  $\wt{N}_x=\wt{R}/\wt{K}_x$ is a line module for $\wt{R}$.

The analogue of Sublemma~\ref{sublem:interesction} is:

\begin{sublemma}\label{sublem:interesction2}
Assume that $r\not=\tau^j(q)$ for $j< 0$. Then $\wt{K_q}\cap R = K_q$.
\end{sublemma}

\begin{proof}  By \cite[Lemma 5.4]{blowdownI}   $\Div L_r^{\vee} = \tau^{-2}(r)$ and $\Div L_x^{\vee} = \tau^{-2}(x)$.   
Using \cite[Remark~3.2]{blowdownI}, the left hand  analog of  \cite[Lemma 5.5]{blowdownI} 
  asserts that, if $L^{\vee}_r[-a]$ embeds into $L^{\vee}_q$ for some $a\geq 1$, then 
  $\tau^{-2}(r) = \Div L^{\vee} = \tau^{-j}(\Div L_q^{\vee}) = \tau^{-j-2}(q)$ for some $j \geq 0$.
  Therefore, the proof of Sublemma~\ref{sublem:interesction} will also work here, provided that this observation is used in place of the final paragraph of that proof.
\end{proof}
 
We claim that $\wt{N_x} := \wt{R}/\wt{K_x} \cong (\wt{L_x})^{\vee}$, where the dual is taken with respect to the ring $\wt{R}$.  By Sublemma~\ref{sublem:dualrelation}, since $R/K_x \cong (R/J_x)^{\vee}$, we have $X: = \Hom_R(J_x, R)_1 = \Hom_R(K_x, R)_1$, where $\dim X = \dim R_1 + 1$.  If $X \subseteq \wt{R}$, then $\Hom_R(J_x, R)_1 \subseteq \wt{R}_1 = \Hom_R(J, R)_1$.  Since $J \neq J_x$, Lemma~\ref{lem:new} gives $X = \Hom_R(J_x, R)_1 = R_1$, which is a contradiction.   Now choose $z \in X$ such that $X = \kk z + R_1$.  
Then $\wt{X} := \kk z + \wt{R}_1$ has $\dim \wt{X} = \dim \wt{R} + 1$, because $z \not \in \wt{R}$.  By a similar argument as used earlier in Case I, we have 
$\Hom_R(J_x, R) \subseteq \Hom_{\wt{R}}(\wt{J_x}, \wt{R})$ and $\Hom_R(K_x, R) \subseteq \Hom_{\wt{R}}(\wt{K_x}, \wt{R})$.
Thus $\wt{X} \subseteq \Hom_{\wt{R}}(\wt{J_x}, \wt{R})$ and since $\wt{J_x}$ is a right line ideal in $\wt{R}$, $\dim \Hom_{\wt{R}}(\wt{J_x}, \wt{R})_1 = \dim \wt{R}_1 + 1$, so $\wt{X} = \Hom_{\wt{R}}(\wt{J_x}, \wt{R})_1$.  An analogous argument on the left gives 
$\wt{X} = \Hom_{\wt{R}}(\wt{K_x}, \wt{R})_1$.  By Sublemma~\ref{sublem:dualrelation}, 
$\wt{R}/\wt{K_x} \cong (\wt{R}/\wt{J_x})^{\vee} = (\wt{L_x})^{\vee}$, as claimed.

 Now follow the proof of Part~I on the left to prove that  
 $(N_p \dotms N_q) =  (\wt{N_p} \dotms \wt{N_q})$.   Since 
 $\wt{N_x} \cong (\wt{L_x})^{\vee}$, it therefore  follows  from  Corollary~\ref{cor:symm}, twice,   that  
 \[  (\wt{L_p} \dotms \wt{L_q}) \ = \ \bigl((\wt{L_p})^{\vee} \dotms (\wt{L_q})^{\vee}\bigr) \ =\  ( \wt{N_p} \dotms \wt{N_q}) 
 \ = \ (N_p \dotms N_q) \ = \  (L_p^{\vee} \dotms L_q^{\vee}) \ = \  (L_p \dotms L_q ),\]
  and the proof is complete.  
 \end{proof}

  %%%%%%%%%%%%%%%%%%%
  %%%%%%%%%%%%%%%%%%%

\bibliographystyle{amsalpha}

%\bibliography{../../../../Dropbox/biblio.bib}
\def\cprime{$'$}
\providecommand{\bysame}{\leavevmode\hbox to3em{\hrulefill}\thinspace}
\providecommand{\MR}{\relax\ifhmode\unskip\space\fi MR }
\providecommand{\MRhref}[2]{%
  \href{https://urldefense.com/v3/__http://www.ams.org/mathscinet-getitem?mr=*1*7D*7B*2__;IyUlIw!!Mih3wA!XNgE_2AkS2bHAp7l93oERXdfno6QoA5Cu2rvwIs8okhCCn9_hR9PovtRlb_21zZG$ }
}
\providecommand{\href}[2]{#2}

 %%%%%%%%%%%%%%%%%%%%%%

\end{document}